\documentclass[oneside,reqno]{amsart}

\usepackage{hyperref}
\usepackage[totalheight=22.5 true cm, totalwidth=16 true cm]{geometry}

\usepackage[shortcuts]{extdash}

\usepackage{xypic}

\usepackage{enumerate}
\usepackage{amsthm}
\usepackage{amsmath}
\usepackage{amssymb}
\usepackage{amsfonts}
\usepackage{dsfont}
\usepackage{array}
\usepackage{tabularx}
\usepackage{graphicx}

\newtheorem{theo}{Theorem}[section]
\newtheorem{lemma}[theo]{Lemma}
\newtheorem{prop}[theo]{Proposition}
\newtheorem{cor}[theo]{Corollary}
\newtheorem{defi}[theo]{Definition}

\newtheorem{conj}[theo]{Conjecture}

\theoremstyle{definition}
\newtheorem{ex}[theo]{Example}

\newcommand{\f}{\phi}

\newcommand{\spec}{\operatorname{Spec}}

\newcommand{\G}{\mathcal{G}}

\newcommand{\Gl}{\operatorname{GL}}

\newcommand{\Aut}{\operatorname{Aut}}
\newcommand{\Hom}{\operatorname{Hom}}

\newcommand{\id}{\operatorname{id}}

\newcommand{\de}{\delta}

\newcommand{\C}{\mathbb{C}}

\newcommand{\Gm}{\mathbb{G}_m}

\newcommand{\Ga}{\mathbb{G}_a}

\newcommand{\Rep}{\operatorname{Rep}}

\newcommand{\Regsing}{\operatorname{RegSing}}

\newcommand{\Lie}{\operatorname{Lie}}

\newcommand{\rank}{\operatorname{rank}}

\newcommand{\B}{\mathcal{B}}

\makeatletter
\@namedef{subjclassname@2020}{\textup{2020} Mathematics Subject Classification}
\makeatother

	\usepackage{todonotes} % for comments

\title[A Gasch\"{u}tz Lemma for linear algebraic groups with applications]{A Gasch\"{u}tz Lemma for linear algebraic groups with applications to the Riemann-Hilbert correspondence}
\author{Jason Bell}
\author{Michael Wibmer}

\address{Jason Bell, Department of Pure Mathematics, University of Waterloo, Canada N2L 3G1, \url{https://www.math.uwaterloo.ca/~jpbell}}
\email{jpbell@uwaterloo.ca}

\address{Michael Wibmer, School of Mathematics, University of Leeds,  LS2 9JT, United Kingdom, \url{https://sites.google.com/view/wibmer}}
\email{m.wibmer@leeds.ac.uk}

\date{\today}

\begin{document}
\maketitle

\begin{abstract}
The classical Gasch\"{u}tz Lemma states that for a surjective morphism $G\to H$ of finite groups, generators $h_1,\ldots,h_d$ of $H$ can be lifted to generators $g_1,\ldots,g_d$ of $G$, provided that $G$ can be generated by $d$ elements. We prove an analog of this statement for linear algebraic groups. This has applications to the algebraic-geometric theory of linear differential equations: We show that every regular singular differential embedding problem over the rational function field has a regular singular solution. We also establish a purely algebraic variant of the Riemann-Hilbert correspondence on the punctured projective line that works over any countable algebraically closed field of characteristic zero.
\end{abstract}

\tableofcontents

\section{Introduction}

For a finite subset $S$ of the Riemann sphere $\mathbb{P}^1(\C)$ and a base point $x_0\in X=\mathbb{P}^1(\C)\smallsetminus S$, the \emph{Riemann-Hilbert correspondence} asserts that the monodromy representation of the topological fundamental group $\pi_1(X,x_0)$ on the local solution space at $x_0$, defines an equivalence of tannakian categories
$$\Regsing(\C(x),S)\longrightarrow \Rep(\pi_1(X,x_0))$$
between the category $\Regsing(\C(x),S)$ of regular singular differential modules over $\C(x)$ with singularities in $S$ and the category $\Rep(\pi_1(X,x_0))$ of finite-dimensional $\C$-linear representations of $\pi_1(X,x_0)$. Hilbert's 21st problem, also known as the Riemann-Hilbert problem, concerns the surjectivity of this functor, although, strictly speaking, Hilbert asked to realize any representation of $\pi_1(X,x_0)$ via a Fuchsian differential equation, rather than a regular singular differential equation. For a discussion of these subtle differences and the interesting history of the problem see e.g., \cite[Sections 5 and 6]{SingerPut:differential}, \cite[Section 12.2]{Sauloy:DifferentialGaloisTheoryThroughRiemannHilbertCorrespondence} or \cite{AnasovBolibruch:TheRiemannHilbertProblem}.

 The Riemann-Hilbert correspondence is an intrinsically analytic statement, as it involves analytic continuation and the topological fundamental group as key ingredients. Nevertheless, it seems natural to wonder if there is a purely algebraic variant that works for any algebraically closed field $k$ of characteristic zero in place of the field of complex numbers. 
While regular singularities are classically defined via a growth condition on the solutions, there also is a purely algebraic characterization of regular singularities and so it is possible to make sense of the category $\Regsing(k(x),S)$ of regular singular differential modules over $k(x)$ with singularities in $S$, where $S$ is a finite subset of $\mathbb{P}^1(k)$.
In the analytic context, as an abstract group, $\pi_1(X,x_0)$ is a free group on $d=|S|-1$ generators. So it makes sense to replace $\Rep(\pi_1(X,x_0))$ with $\Rep(F_d)$, the category of finite-dimensional $k$-linear representations of the abstract free group $F_d$ on $d$ generators. We would then expect:
\begin{conj} \label{conj: RH}
There exists an equivalence of tannakian categories between $\Regsing(k(x),S)$ and $\Rep(F_d)$, where $d=|S|-1$.
\end{conj}	
An equivalent reformulation of Conjecture \ref{conj: RH} in the language of differential Galois theory is:
\begin{conj} \label{conj: equivalent}
	The differential Galois group of the family of all regular singular linear differential equations over $k(x)$ with singularities in $S$, is the free proalgebraic group on $|S|-1$ generators. 
\end{conj}
The term ``proalgebraic group'' here is synonymous with ``affine group scheme'' but emphasizes the analogy with profinite groups. We stress the fact that Conjecture~\ref{conj: RH} only asserts the existence of an equivalence, no specific (canonical) functor realizing the proposed equivalence is suggested. Conjecture ~\ref{conj: RH} may seem like a statement that P. Deligne and his school might have proved in much greater generality in the 70s or 80s. However, this does not seem to be the case. Indeed,  Conjecture~\ref{conj: RH} implies the solution of the inverse problem in differential Galois theory, i.e., the statement that every linear algebraic group is a differential Galois group over $k(x)$ and this has only been established in 2005 (\cite{Hartmann:OnTheInverseProblemInDifferentialGaloisTheory}).

On the other hand, the case $d=1$ of Conjecture~\ref{conj: RH} already appeared in \cite[Chapter VI, Section~1.2.6~ a)]{Saavedra:CategoriesTannakien} (cf. \cite{HaiDosSantosTam:AlgebraicTheoryOfFormalRegularSingularConnectionsWithParameters}) and the nilpotent variant of Conjecture \ref{conj: RH} does indeed hold in far greater generality (\cite[Section 10]{Deligne:LeGroupeFondamentalDeLaDroiteProjectiveMoinsTroisPoints}), i.e., in the context of Conjecture \ref{conj: RH} there exists an equivalence if we restrict to the nilpotent objects on both sides, i.e., objects that allow a filtration such that the corresponding quotients are trivial. In the language of differential Galois theory, this means that the differential Galois group of the family of all regular singular linear differential equations over $k(x)$ with singularities in $S$ and unipotent differential Galois group, is the free prounipotent group on $|S|-1$ generators. 

\medskip

In this article, we prove Conjecture \ref{conj: RH} in case $k$ is countable (so including the important special case $k=\overline{\mathbb{Q}}$). In fact, for $k$ countable, we prove a more general result, that allows $S$ to be an arbitrary (not necessarily finite) subset of $\mathbb{P}^1(k)$:

\begin{theo}[Theorem \ref{theo: main main}] \label{theo: main}
	Let $k$ be a countable algebraically closed field of characteristic zero and let $S$ be a subset of $\mathbb{P}^1(k)$. Then the differential Galois group of the family of all regular singular linear differential equations over $k(x)$ with singularities in $S$, is the free proalgebraic group on a set of cardinality $|S|-1$. In particular (choosing $S=\mathbb{P}^1(k)$), the differential Galois group of the family of all regular singular linear differential equations over $k(x)$ is the free proalgebraic group on a set of cardinality $|k|$.
\end{theo}	 
We note that the statement in Theorem \ref{theo: main} has recently been proved for $k=\mathbb{C}$ in \cite{Wibmer:RegularSingularDifferentialEquationsAndFreeProalgebraicGroups} based on the Riemann-Hilbert correspondence and a study of projective limits of free proalgebraic groups. Moreover, in \cite[Section 4.1]{Wibmer:RegularSingularDifferentialEquationsAndFreeProalgebraicGroups} it was suggested that Theorem \ref{theo: main} should hold for any algebraically closed field $k$ of characteristic zero. In Section \ref{subsec: Future work} we outline a strategy for removing the assumption that $k$ is countable in Theorem \ref{theo: main}, so in particular, this includes a strategy for tackling Conjecture \ref{conj: RH} in full generality.

For uncountable base fields, our results are not strong enough to imply the freeness of the groups in Theorem \ref{theo: main} but they imply the weaker result that the groups in Theorem \ref{theo: main} are projective proalgebraic groups (Corollary \ref{cor: projective}). Recall that a proalgebraic group $P$ is projective if for every surjective morphism $G\twoheadrightarrow H$ of proalgebraic groups and every morphism $P\to H$, there exists a morphism $P\to G$ making
$$
\xymatrix{
	P \ar[rd] \ar@{..>}[d] & \\
	G \ar@{->>}[r] & H
}
$$
commutative.
%
%\begin{theo}
%	Let $k$ be an algebraically closed field of characteristic zero and  let $S$ be a subset of $\mathbb{P}^1(k)$. Then the differential Galois group of the family of all regular singular linear differential equations over $k(x)$ with singularities in $S$ is a projective proalgebraic group.
%\end{theo}	

In the general algebraic-geometric theory of linear differential equations on an algebraic variety one considers a smooth connected algebraic variety $X$ over $k$ and the category % $\operatorname{MIC}^{\operatorname{rs}}(X/k)$
 of all coherent $\mathcal{O}_X$-sheaves endowed with a regular integrable connection. This is a tannakian category and % any base point $x_0\in X(k)$ defines a fibre functor $\omega_{x_0}$ on $\operatorname{MIC}^{\operatorname{rs}}(X/k)$. 
 by tannakian duality it defines a proalgebraic group $\pi_1^{\operatorname{rs}}(X)$, the \emph{regular singular fundamental group of $X$}. For $k=\mathbb{C}$, the upshot of Deligne's Riemann-Hilbert correspondence (\cite{Deligne:EquationesDifferentiellesAPointsSingulierReguliers}) is that $\pi_1^{\operatorname{rs}}(X)$ agrees with the proalgebraic completion  $\pi_1(X(\mathbb{C}))^{\operatorname{alg}}$ of the topological fundamental group $\pi_1(X(\C))$ of $X(\C)$. For base fields different from $\mathbb{C}$, the determination of $\pi_1^{\operatorname{rs}}(X)$ has proven difficult. It seems so far there are essentially only two cases in which the group is known and non-trivial:
\begin{itemize}
%	\item $\mathbb{P}^1$, the group is trivial;
%	\item $\mathbb{P}^1$ with one point removed, the group is trivial;
	\item $\mathbb{P}^1$ with two points removed, the group is $\mathbb{Z}^{\operatorname{alg}}$, the free proalgebraic group on one generator (\cite[Chapter VI, Section~1.2.6~ a)]{Saavedra:CategoriesTannakien});
	\item an elliptic curve, the group is $(\mathbb{Z}\oplus\mathbb{Z})^{\operatorname{alg}}$ (\cite[Rem. 5.3]{BaoHaiThinh:CohomologyOfTheDifferentialFundamentalGroupOfAlgebraicCurves})\footnote{The result in \cite{BaoHaiThinh:CohomologyOfTheDifferentialFundamentalGroupOfAlgebraicCurves} is not formulated in this way, but the exact sequence $0\to H^0(X,\Omega^1_X)\to \mathbb{X}\to \operatorname{Pic}^0(X)\to 0 $ of abelian groups, determining the character group $\mathbb{X}$ of $\pi_1^{\operatorname{rs}}(X)$, is split and of the form $0\to \mathbb{Q}^{(|k|)} \to \mathbb{X}\to \mathbb{Q}^{(|k|)}\oplus \mathbb{Q}/\mathbb{Z}\oplus\mathbb{Q}/\mathbb{Z}\to 0$, so that $\mathbb{X}\cong \mathbb{Q}^{(|k|)}\oplus\mathbb{Q}/\mathbb{Z}\oplus\mathbb{Q}/\mathbb{Z}\cong k^\times\oplus k^\times\simeq\Hom(\mathbb{Z}\oplus\mathbb{Z}, k^\times)$.} 
\end{itemize}	
Our main result augments this picture by showing that for $\mathbb{P}^1$ with $d+1$ points removed, the group is $F_d^{\operatorname{alg}}$, the proalgebraic completion of the free group $F_d$ on $d$ generators, if $k$ is countable. One may wonder, if not in general, for an algebraically closed subfield $k$ of $\mathbb{C}$, the regular singular fundamental group $\pi_1^{\operatorname{rs}}(X)$ agrees with the proalgebraic completion $\pi_1(X(\mathbb{C}))^{\operatorname{alg}}$ of the topological fundamental group of $X(\mathbb{C})$. There are two major obstructions to such a result. Firstly, the formation of the regular singular fundamental group, as well as the formation of the proalgebraic completion, is not compatible with base change, so there is no obvious way to reduce the claim to the known case of the complex numbers (as one has for the \'{e}tale fundamental group, which is the profinite completion of $\pi_1(X(\C))$). Secondly, there is no canonical map from $\pi_1(X(\C))$ to $\pi_1^{\operatorname{rs}}(X)(k)$, because the monodromy representation need not be defined over $k$.

The example $X=\mathbb{P}^1\smallsetminus\{0,\infty\}$ serves well to illustrate the problem with the lack of canonical maps. In this case, $\pi^{\operatorname{rs}}(X)$ is canonically isomorphic to $\Ga\times D(k/\mathbb{Z})$, where for any abelian group $M$, $D(M)$ denotes the diagonalizable proalgebraic group with character group $M$, i.e., $D(M)(R)=\Hom(M,R^\times)$ for any $k$-algebra $R$. On the other hand, the free proalgebraic group $\mathbb{Z}^{\operatorname{alg}}$ on one generator is $\Ga\times D(k^\times)$ with free generator $(1,\id)\in k\times\Hom(k^\times,k^\times)$. So to give an isomorphism between $\pi^{\operatorname{rs}}(X)$ and $\mathbb{Z}^{\operatorname{alg}}$ amounts to choosing an isomorphism of abelian groups between $k/\mathbb{Z}$ and $k^\times$. For $k=\C$, we can choose the isomorphism induced by the morphism of groups $\C\to \C^\times,\ a\mapsto e^{2\pi i a}$ with kernel $\mathbb{Z}$, and this recovers the isomorphism between $\pi^{\operatorname{rs}}(X)$ and $\mathbb{Z}^{\operatorname{alg}}$ provided by the Riemann-Hilbert correspondence. However, for other base fields, like $k=\overline{\mathbb{Q}}$, although an isomorphism $k/\mathbb{Z}\cong k^\times$ always exists, there is no natural choice. 

We circumvent these difficulties by abandoning the idea of constructing an explicit isomorphism between $\pi^{\operatorname{rs}}(X)$ and $\pi_1(X(\mathbb{C}))^{\operatorname{alg}}$. Instead, we only establish the existence of an isomorphism, based on a characterization of the proalgebraic group $F_d^{\operatorname{alg}}$ that does not rely on free generators.

%
%
% To illustrate the latter point, consider the $k=\overline{\mathbb{Q}}$, $X=\mathbb{P}^1\smallsetminus\{0,\infty\}$ and the baby example $y'=\frac{\alpha}{x}y$ with $\alpha\in\overline{\mathbb{Q}}$. This differential equation is regular singular with two singularities $0$ and $\infty$. The fundamental group $\pi_1(X(\C))$ is generated by a simple loop $\gamma$ around $0$. The action of $\gamma$ on the solution $y=x^\alpha$ is given by multiplication with $e^{2\pi i\alpha}$.
%

With several notable exceptions, such as \cite{Deligne:LeGroupeFondamentalDeLaDroiteProjectiveMoinsTroisPoints}, \cite{EsnaultViehweg:SemistableBundlesOnCurvesAndIrreducibleRepresentationsOfTheFundamentalGroup}, \cite{EsnaultHai:TheGaussManinconnectionAnsTannakaDuality}, \cite{Esnault:OnFlatBundlesInCharacteristic0Andp}, \cite{Zhang:TheHomotopySequenceOfTheAlgebraicFundamentalGroup}, \cite{AndreBaldassarriCailotto:DeRhamCohomologyOfDifferentialModulesOnAlgebraicVarieties}, \cite{BiswasHaiDosSantos:ConnectionsOnTrivialVectorBundlesOverProjectiveSchemes}, \cite{BaoHaiThinh:CohomologyOfTheDifferentialFundamentalGroupOfAlgebraicCurves}, \cite{Xiadong:OnTheTannakianGroupScheme}
much of the literature around the regular singular fundamental group and the Riemann-Hilbert correspondence in characteristic zero, is restricted to the field of complex numbers. There has also been a significant interest in positive characteristic analogs (see e.g., \cite{Gieseker:FlatVectorBundlesAndTheFundamentalGroupInNonZeroCharacteristics}, \cite{MatzatPut:IterativeDifferentialEquationsAndTheAbhyankarConjecture}, \cite{DosSantos:FundamentalGroupSchemesForStratifiedSheaves},  \cite{Maurischat:GaloisTheoryForIterativeConnectionsAndNonreducedGaloisGroups}, \cite{EsnaultMehta:SimplyConnectedProjectiveManifoldsInCharacteristicp>0HaveNoNontrivialStratifiedBundles}, \cite{EsnaultHogadi:OnTheAlgebraicFundamentalGroupOfSmoothVarietiesInCharacteristicp},  \cite{Kindler:RegularSingularStratifiedBundlesAndTameRamification}, \cite{DosSantos:TheHomotopyExactSequenceForTheFundamentalGroupSchemeAndInfinitesimalEquivalenceRelations}, \cite{BattistonKindler:TheHomotopySequenceForRegularSingularStratifiedBundles}, \cite{VanDerPut:StratifiedBundlesOnCurvesAndDifferentialGaloisGroupsInPositiveCharacteristic}, \cite{EsnaultSrinivas:ARelativeVersionOfGiesekerProblemOnStratificationsInCharacteristicp}). While the results in this article are restricted to algebraically closed fields of characteristic zero, we hope that some of the ideas presented here, in particular the use of regular singular differential embedding problems, will also be useful for tackling the case of not necessarily algebraically closed fields of characteristic zero (e.g., the case of large fields, cf. \cite{BachmayrHartmannHarbaterPopLarge}) or fields of positive characteristic (cf. \cite{Ernst:IterativeDifferentialEmbeddingProblemsInPositiveCharacteristic}).
%A key idea is that for a smooth irreducible algebraic variety over an algebraically closed field of positive characteristic an appropriate positive characteristic analog of regular singular differential modules (or connections), such as regular singular stratified bundles, allows one to define, via the tannakian formalism, a fundamental group scheme that is seen as the positive characteristic analog of the proalgebraic completion of the topological fundamental group over the complex numbers. In the light of this philosophy, it may seem even more surprising that even in the basic case of  the projective line with finitely many points removed, this fundamental group scheme has thus far not been shown to be of the expected form over an algebraically closed field of characteristic zero (different from $\mathbb{C}$).

\medskip

Our proof of Theorem \ref{theo: main} is based on two key ingredients, both of which we deem of independent interest. The first key ingredient is the following purely group theoretic result: 
\begin{theo}[Theorem \ref{theo: Geschuetz}, Gasch\"{u}tz Lemma for linear algebraic groups] \label{theo: Gaschuetz intro}
	Let $k$ be an algebraically closed field of characteristic zero and let $\f\colon G\to H$ be a surjective morphism of linear algebraic groups over $k$. Assume that $h_1,\ldots,h_d$ generate $H$ and that $G$ can be generated by $d$ elements. Then there exist generators $g_1,\ldots,g_d$ of $G$ such that $\f(g_i)=h_i$ for $i=1,\ldots,d$.
\end{theo}
Here $h_1,\ldots,h_d$ are said to generate $H$ (as an algebraic group) if the smallest closed subgroup of $H$ containing $h_1,\ldots,h_d$ is equal to $H$. Recall that the classical Gasch\"{u}tz Lemma  (\cite{Gaschuetz:ZuEinemVonNeumannGestelltenProblem}) is a statement about finite groups: For a surjective morphism $\f\colon G\to H$ of finite groups and generators $h_1,\ldots,h_d$ of $H$, if $G$ can be generated by $d$ elements, then there exist generators $g_1,\ldots,g_d$ of $G$ such that $\f(g_i)=h_i$ for $i=1,\ldots,d$. It is well-known that the statement of Gasch\"{u}tz's Lemma does not generalize to infinite groups. For example, the generator $2$ of $\mathbb{Z}/\mathbb{Z}5$ does not lift to a generator of $\mathbb{Z}$, under the canonical projection $\mathbb{Z}\to \mathbb{Z}/\mathbb{Z}5$.

However, Gasch\"{u}tz's Lemma generalizes from finite groups to profinite groups rather directly (\cite[Lemma~17.7.2]{FriedJarden:FieldArithmetic}), as long as the morphism of groups is assumed to be continuous and generation is interpreted in the sense of topological groups. It was shown in \cite{TalTsachik:GaschuetzLemmaForCompactGroups} that the Gasch\"{u}tz Lemma also holds for all compact metrisable topological groups but not for all compact topological groups. The case of connected Lie groups is treated quite comprehensively in \cite{CohenVigdorovichLLiftingGeneratorsInconnectedLieGroups}. In particular, it is shown there that the Gasch\"{u}tz Lemma holds for connected perfect Lie groups but does not hold for a large class of connected abelian Lie groups.

We show below (Example \ref{ex: need rank assumption }) that the Gasch\"{u}tz Lemma does not hold for proalgebraic groups without additional assumptions. However, we conjecture that the Gasch\"{u}tz Lemma does hold for proalgebraic groups as long as we add a rather mild extra assumption (Conjecture \ref{conj: Gaschuetz for proalgebraic}). 
This extra assumption is not a condition on $G$, it is only a condition on the size of $H$ relative to the size of $k$.
The extra assumption asserts that $H$ can be written as a projective limit of linear algebraic groups over an index set of cardinality strictly less than $|k|$.
We prove this conjecture for the case that $k$ is countable (Theorem \ref{theo: Conjectures for k countable}) and for the case $d=1$ (Theorem \ref{theo: conjecture for d=1}). 

The connection between Theorem \ref{theo: Gaschuetz intro} and the proalgebraic group $F_d^{\operatorname{alg}}$ is that Theorem \ref{theo: Gaschuetz intro} precisely describes which morphisms of algebraic groups occur in the canonical projective system of algebraic groups whose projective limit is $F_d^{\operatorname{alg}}$. Indeed, if $G\to H$ is a surjective morphism of linear algebraic groups such that $G$ can be generated by $d$ elements, then Theorem \ref{theo: Gaschuetz intro} shows that for every surjective morphism $F_d^{\operatorname{alg}}\to H$, there exists a surjective morphism $F_d^{\operatorname{alg}}\to G$, making
$$
\xymatrix{F_d^{\operatorname{alg}}
	 \ar@{..>}[d] \ar[rd] & \\
	G \ar[r] & H	
}
$$ 
commutative. It turns out that, if $k$ is countable, $F_d^{\operatorname{alg}}$ is the unique (up to isomorphism) proalgebraic group that can be generated by $d$ elements and satisfies this lifting (or embedding) property. So this yields a characterization of $F_d^{\operatorname{alg}}$ that does not refer to a specific set of free generators.

\medskip

%\begin{conj}[Gasch\"{u}tz Lemma for proalgebraic groups] \label{conj: Gaschuetz for proalgebraic}
%		Let $k$ be an algebraically closed field of characteristic zero and let $\f\colon G\to H$ be a quotient map of proalgebraic groups over $k$ such that $\rank(H)<|k|$. Assume that $h_1,\ldots,h_d\in H(k)$ generate $H$ and that $G$ can be generated by $d$ elements. Then there exist $g_1,\ldots,g_d\in G(k)$ generating $G$ such that $\f(g_i)=h_i$ for $i=1,\ldots,d$.
%\end{conj}
%Here the rank $\rank(H)$ of a proalgebraic group $H$ is defined as the smallest cardinality of an index set that can be used to write $H$ as a projective limit of linear algebraic groups. We prove Conjecture \ref{conj: Gaschuetz for proalgebraic} below in the case that $k$ is countable (Theorem ??), as a rather straight forward consequence of the Gasch\"{u}tz Lemma for linear algebraic groups.

The second key ingredient of our proof of Theorem \ref{theo: main} is the solvability of all regular singular differential embedding problems. To explain this result, we need to recall some background from differential Galois theory:
% (See \cite{CrespoHajto:AlgebraicGroupsAndDifferentialGaloisTheory}, \cite{Magid:LecturesOnDifferentialGaloisTheory}, \cite{Sauloy:DifferentialGaloisTheoryThroughRiemannHilbertCorrespondence}, or \cite{SingerPut:differential} for introductory textbooks.)
% 
  To a linear differential equation $y'=Ay$, where $A\in k(x)^{n\times n}$, there is associated a \emph{Picard-Vessiot extension} $L$ of $k(x)$; the analog of the splitting field in classical Galois theory. The \emph{differential Galois group}, i.e., the group of automorphisms of the differential field extension $L/k(x)$, naturally carries the structure of a linear algebraic group. A \emph{differential embedding problem} over $k(x)$ consists of a surjective morphism $G\to H$ of linear algebraic groups together with a Picard-Vessiot extension $L/k(x)$ with differential Galois group $H$. A \emph{solution} is a Picard-Vessiot extension $M/k(x)$ containing $L$ and with differential Galois group $G$, such that the restriction morphism $G\to H$ of differential Galois groups coincides with the given morphism $G\to H$. Note that the question of solvability of a differential embedding problem with $L=k(x)$ reduces to the question ``Is $G$ a differential Galois group over $k(x)$?'', i.e., the inverse problem of differential Galois theory. Solving the inverse problem tells us which linear algebraic groups are quotients of the absolute differential Galois group. Solving differential embedding problems tells us, more generally, how these algebraic quotients fit together into a projective system of linear algebraic groups. Building on previous work that implied the special cases $k=\mathbb{C}$ (\cite{BachmayrHarbaterHartmannWibmer:DifferentialEmbeddingProblems}) and $k$ of infinite transcendence degree (\cite{BachmayrHartmannHarbaterPopLarge}), it was recently shown in \cite{FengWibmer:DifferentialGaloisGroupsSpecializationsAndMatzatsConjecture} that every differential embedding problem over $k(x)$ has a solution. 

Fixing a finite subset $S$ of $\mathbb{P}^1(k)$ of cardinality $d+1$, we call a differential embedding problem $(G\to H,\ L/k(x))$ over $k(x)$ \emph{regular singular} (with respect to $S$), if $G$ can be generated by $d$ elements and $L/k(x)$ is a Picard-Vessiot extension for a regular singular differential equation with singularities in $S$. A solution $M/k(x)$ of $(G\to H,\ L/k(x))$ is \emph{regular singular} (with respect to $S$) if $M/k(x)$ is a Picard-Vessiot extension for a regular singular differential equation with singularities in $S$. The following regular singular variant of the solvability of all differential embedding problems over $k(x)$ is our second key ingredient of the proof of Theorem \ref{theo: main}.

\begin{theo}[{Theorem \ref{theo: solve regular singular embedding problem}}] \label{theo: second key ingredient}
	Every regular singular differential embedding problem over $k(x)$ has a regular singular solution.
\end{theo}

We note that Theorem \ref{theo: second key ingredient} is a new result, even in the classical case $k=\mathbb{C}$, that does not seem to follow from the Riemann-Hilbert correspondence in an obvious way.
In the case of a regular singular embedding problem with $L=k(x)$, Theorem \ref{theo: second key ingredient} reduces to the statement that every linear algebraic group that can be generated by $d$ elements is the differential Galois group of a regular singular linear differential equation over $k(x)$ with singularities in $S$. This special case of Theorem \ref{theo: second key ingredient} has already been established in \cite{SerafiniWibmer:RegularSingularInverseProblem}. 

There are no direct implications between Theorem \ref{theo: second key ingredient} and the solvability of all differential embedding problems over $k(x)$, because not every differential embedding problem of $k(x)$ is regular singular and not every solution of a regular singular differential embedding problem has to be regular singular.

\medskip

We now sketch our proof of Theorem \ref{theo: main}. Firstly, the case that $S$ is arbitrary can be reduced via a limit argument and the characterization of free proalgebraic groups on infinitely many generators via embedding problems, to the case that $S$ is finite. So we are reduced to proving Conjecture \ref{conj: equivalent}. Let $\Gamma_S$ denote the differential Galois group of the family of all regular singular differential equations over $k(x)$ with singularities in $S$. By the Riemann-Hilbert correspondence, Conjecture \ref{conj: equivalent} holds for $k=\mathbb{C}$, i.e., $\Gamma_S$ is the free proalgebraic group on $d$ generators for $k=\mathbb{C}$.
The Gasch\"{u}tz Lemma for linear algebraic groups (Theorem \ref{theo: Gaschuetz intro}) then shows that for $k=\mathbb{C}$, the proalgebraic group $\Gamma_S$ satisfies the following embedding property: 
For every surjective morphism $G\to H$ of linear algebraic groups such that $G$ can be generated by $d=|S|-1$ elements and every surjective morphism $\Gamma_S\to H$, there exists a surjective morphism $\Gamma_{S}\to G$ such that
\begin{equation} \label{eq: embedding prop for intro}
\xymatrix{
	\Gamma_{S} \ar@{..>}[d] \ar[rd] & \\
	G \ar[r] & H	
}
\end{equation}
commutes. Via the differential Galois correspondence, this embedding property for $\Gamma_S$ for $k=\mathbb{C}$ translates to the statement that all regular singular differential embedding problems over $\C(x)$ have a regular singular solution, i.e., Theorem \ref{theo: second key ingredient} holds for $k=\mathbb{C}$.
 Using the specialization theorem from \cite{FengWibmer:DifferentialGaloisGroupsSpecializationsAndMatzatsConjecture}, we show that the validity of Theorem \ref{theo: second key ingredient} for $k=\mathbb{C}$, implies the validity of Theorem \ref{theo: second key ingredient} for all (algebraically closed, characteristic zero) base fields $k$.
  Now Theorem \ref{theo: second key ingredient} translates back to $\Gamma_{S}$ having the embedding property illustrated in (\ref{eq: embedding prop for intro}).
 For $k$ countable, this is enough to deduce that $\Gamma_{S}$ is free on $d=|S|-1$ generators. %isomorphic to the for We stress the fact that our proof does not yield a specific (canonical) isomorphism between $\Gamma_{S,k}$ and $\Gamma_{d,k}$. We only show that there exists an isomorphism.

We conclude this introduction with an outline of the article. In Section \ref{sec: The Gaschuetz Lemma for algebraic groups} we prove the Gasch\"{u}tz Lemma for linear algebraic groups (Theorem \ref{theo: Gaschuetz intro}) via a reduction to three special cases concerning the structure of the kernel of the surjective morphism. In Section \ref{sec: Finite generation under base change} we show that the minimal number of generators of a linear algebraic group is invariant under base change. This is needed later to deduce that the differential Galois group of a regular singular differential equation with $d+1$ singularities can be generated by $d$ elements, which is clear for $k=\C$ but not otherwise. In Section \ref{sec: Embedding problems and free proalgebraic groups} we deduce from the Gasch\"{u}tz Lemma for linear algebraic groups the embedding property satisfied by $F_d^{\operatorname{alg}}$ and show that, for $k$ countable, this embedding property uniquely characterizes $F_d^{\operatorname{alg}}$. We also discuss the variant of the 
Gasch\"{u}tz Lemma for proalgebraic groups that would be needed to extend this characterization of $F_d^{\operatorname{alg}}$ to uncountable base fields. In Section \ref{sec: Regular singular differential embedding problems} we introduce and solve regular singular differential embedding problems and derive Theorem \ref{theo: main} from this. To deduce the solvability over $k$ from the solvability over $\C$, the key ingredient is that a regular singular differential equation over $\mathbb{C}(x)$ can be specialized to a regular singular differential equation, say over $\overline{\mathbb{Q}}(x)$, without changing the algebraic properties of the solutions.

\medskip

\vspace{5mm}

{\bf Notation:} Throughout this article $k$ is an algebraically closed field of characteristic zero. 
All group schemes are assumed to be over $k$ unless the contrary is clearly indicated. We use the term ``proalgebraic group'' (over $k$) for affine group scheme (over $k$). All algebraic groups are assumed to be affine (equivalently linear). A closed subgroup of a proalgebraic group is a closed subgroup scheme.
Following the nomenclature of \cite{Milne:AlgebraicGroupsTheTheoryOfGroupSchemesOfFiniteTypeOverAField}, a morphism  $\f\colon G\to H$ of proalgebraic groups is a quotient map if the dual map $\f^*\colon k[H]\to k[G]$ is injective. We use $G\twoheadrightarrow H$ to graphically indicate quotient maps. As usual, $\mathbb{G}_a$ denotes the additive group and $\mathbb{G}_m$ the multiplicative group.

If $G$ is a proalgebraic group and $g_1,\ldots,g_d\in G(k)$, then $\langle g_1,\ldots,g_d\rangle$ denotes the smallest closed subgroup of $G$ containing $g_1,\ldots,g_d$ in its $k$-points. If $G=\langle g_1,\ldots,g_d\rangle$ we say that $g_1,\ldots,g_d$ generate $G$ (as a proalgebraic group) and that $G$ is finitely generated. We let $d(G)$ denote the smallest natural number such that $G$ can be generated by $d(G)$ elements. If $G$ is not finitely generated, this is understood to be infinity. It is well-known that algebraic groups are finitely generated (see e.g., \cite[Lemma 5.13]{SingerPut:differential}) but note that this uses our characteristic zero assumption. In positive characteristic the additive group $\mathbb{G}_a$ is not finitely generated. If $\f\colon G\twoheadrightarrow H$ is a quotient map of proalgebraic groups and $G=\langle g_1,\ldots,g_d\rangle$, then $H=\langle \f(g_1),\ldots,\f(g_d)\rangle$. In particular, $d(H)\leq d(G)$.

A non-commutative connected algebraic group $G$ is \emph{almost-simple} if any closed normal subgroup of $G$ is equal to $G$ or finite.  A non-commutative connected algebraic group $G$ is \emph{simple} if any closed normal subgroup of $G$ is equal to $G$ or trivial. Semisimple and reductive algebraic groups are not assumed to be connected. As usual, $G^o$ denotes the identity component of the algebraic group $G$.

For an abelian group $M$ (usually written multiplicatively), we denote with $D(M)$ the diagonalizable proalgebraic group with character group $M$, i.e., for any $k$-algebra $R$ the abelian group $D(M)(R)=\Hom(M,R^\times)$ is defined as the group of morphisms of groups from $M$ to the multiplicative group of $R$. The functor $D(M)$ is represented by the group ring of $M$, i.e., $k[D(M)]=k[M]$.

As a matter of practicality, throughout Section \ref{sec: The Gaschuetz Lemma for algebraic groups} we identify an algebraic group $G$ with its $k$-points $G(k)$. As $k$ is algebraically closed and of characteristic zero, this is harmless. We do not usually make this identification in the later sections, especially when dealing with proalgebraic groups, which we typically treat as functors on the category of $k$-algebras.

\medskip

\vspace{5mm}

{\bf AI Declaration:} Most of the work presented in this article was conducted before AI usage in pure mathematics became of practical relevance. Google Gemini was used to research the background literature and to find references for intermediate results. The text was written by the authors in its entirety and all claims made have been verified by the authors. GPT-5.6 Sol was used to check for mistakes in the final manuscript.

\section{The Gasch\"{u}tz Lemma for algebraic groups}
\label{sec: The Gaschuetz Lemma for algebraic groups}

In this section we prove the Gasch\"{u}tz Lemma for algebraic groups (Theorem \ref{theo: Gaschuetz intro} from the introduction). The proof proceeds through a series of reduction steps to special cases. For the sake of clarity of exposition, we first treat the special cases.

\medskip

\subsection{Almost-minimal normal closed subgroups}

To shorten our statements, we make the following definition analogous to \cite[2. Section 1]{TalTsachik:GaschuetzLemmaForCompactGroups}.

\begin{defi} \label{defi: Gaschuetz}
	A quotient map $\f\colon G\twoheadrightarrow H$ of algebraic groups is \emph{Gasch\"{u}tz} if it satisfies the following property: For all $d\geq 1$ and all $h_1,\ldots,h_d\in H$ with $H=\langle h_1,\ldots,h_d\rangle$ and $d(G)\leq d$, there exist $g_1,\ldots,g_d\in G$ such that $\f(g_i)=h_i$ for $i=1,\ldots,d$ and $G=\langle g_1,\ldots,g_d\rangle$.
\end{defi}	

In other words, a quotient map $G\twoheadrightarrow H$ is Gasch\"{u}tz if and only if every family of generators of $H$ of size at least $d(G)$, lifts to a family of generators of $G$. 
With Definition \ref{defi: Gaschuetz}, the Gasch\"{u}tz Lemma for algebraic groups can be reformulated as follows: every quotient map of algebraic groups is Gasch\"{u}tz. For later use, we record two useful lemmas.

\begin{lemma} \label{lemma: composition of Gaschuetz is Gaschuetz}
	Let $\f\colon G\twoheadrightarrow G'$ and $\f'\colon G'\twoheadrightarrow G''$ be Gasch\"{u}tz quotient maps of algebraic groups. Then the composition $\f'\f\colon G\twoheadrightarrow G''$ is Gasch\"{u}tz.
\end{lemma}	
\begin{proof}
	Let $g_1'',\ldots,g_d''$ be generators of $G''$ and assume that $d(G)\leq d$. Then $d(G')\leq d$ and because $\f'$ is Gasch\"{u}tz, there exist generators $g_1',\ldots,g_d'$ of $G'$ with $\f'(g_i')=g_i''$ for $i=1,\ldots,d$. Since $\f$ is Gasch\"{u}tz, there exist generators $g_1,\ldots,g_d$ of $G$ with $\f(g_i)=g_i'$ for $i=1,\ldots,d$. Thus  $g_1,\ldots,g_d$ are generators of $G$ with $\f'(\f(g_i))=g_i''$ for $i=1,\ldots,d$.
\end{proof}

\begin{lemma} \label{lemma: finite kernel is Gaschuetz}
	A quotient map of algebraic groups with finite kernel is Gasch\"{u}tz.	
\end{lemma}
\begin{proof}
	The original counting argument (see \cite{Gaschuetz:ZuEinemVonNeumannGestelltenProblem} or \cite[Lemma 17.7.2]{FriedJarden:FieldArithmetic}) of Gasch\"{u}tz applies in this context. For the sake of completeness and because the argument is quite nice, we include the proof. 
	
	Let $\f\colon G\twoheadrightarrow H$ be a quotient map of algebraic groups with finite kernel $N$ and let $h_1,\ldots,h_d$ be generators of $H$ with $d(G)\leq d$. 	
	For a generating tuple $\bold{h}'=(h'_1,\ldots,h'_d)$ of $H$ and a closed subgroup $G'$ of $G$ such that the restriction of $\f$ to $G'$ is a quotient map, we let $S(G',\bold{h}')$ be the set of all $(g_1',\ldots,g_d')\in G'^d$ such that $\langle g'_1,\ldots,g'_d\rangle=G'$ and $\f(g_i')=h_i'$ for $i=1,\ldots,d$. 
	Note that $S(G',\bold{h}')$ is a finite set because every $h_i'$ has only $|N|$ lifts to $G$. We have to show that $|S(G,\bold{h})|\geq 1$, where $\bold{h}=(h_1,\ldots,h_d)$.
	
	The key step of the proof is to show that $|S(G',\bold{h}')|$ does not depend on the choice of the generating set $\bold{h}'$ of $H$. Suppose, for a contradiction, that $S(G',\bold{h}')$ does depend on $\bold{h}'$. As every descending chain of closed subgroups of $G$ is finite, there then exists a minimal such $G'$, i.e., $S(G'',\bold{h}')$ does not depend on $\bold{h}'$ for every proper closed subgroup $G''$ of $G'$ such that $\f$ restricted to $G''$ is a quotient map. There are $|N\cap G'|^d$ tuples $(g'_1,\ldots,g'_d)\in G'^d$ with $\f(g'_i)=h'_i$ for $i=1,\ldots,d$ and each one of them generates some closed subgroup $G''$ of $G'$ such that the restriction of $\f$ to $G''$ is a quotient map. Thus
	\begin{equation} \label{eq: counting sum}
		|N\cap G'|^d=|S(G',\bold{h}')|+\sum_{G''}|S(G'',\bold{h}')|,
	\end{equation} 
	where the sum is taken over all proper closed subgroups $G''$ of $G'$ such that the restriction of $\f$ to $G''$ is a quotient map. As $|S(G'',\bold{h}')|$ does not depend on $\bold{h}'$, formula (\ref{eq: counting sum}) shows that $|S(G',\bold{h}')|$ does not depend on $\bold{h}'$; a contradiction. So, in particular, $|S(G,\bold{h}')|$ does not depend on $\bold{h}'$. 
	
	As $d(G)\leq d$, there exist generators $g_1,\ldots,g_d$ of $G$. For $\f(\bold{g})=(\f(g_1),\ldots,\f(g_d))$ we clearly have $|S(G,\f(\bold{g}))|\geq 1$. So $|S(G,\bold{h})|=|S(G,\f(\bold{g}))|\geq 1$ as desired.
\end{proof}

A non-trivial connected normal closed subgroup $N$ of an algebraic group $G$ is called \emph{almost-minimal} if for every normal closed subgroup $N'$ of $G$ with $N'\leq N$ we have $N'=N$ or $N'$ is finite. The following lemma gives a fairly precise description of the almost-minimal normal closed subgroups of an algebraic group.

\begin{lemma} \label{lemma: almost-minimal subgroups}
	Let $N$ be an almost-minimal normal closed subgroup of an algebraic group $G$. Then there exists an $r\geq 1$ such that $N=\Ga^r$, $N=\Gm^r$ or $N/Z=N_1^r$, where $Z$ is a finite closed subgroup of $N$ that is normal in $G$ and $N_1$ is a simple algebraic group.
\end{lemma}
\begin{proof}
	Let $R(N)$ be the (solvable) radical of $N$. As $R(N)$ is characteristic in $N$, it is normal in $G$. Because $N$ is an almost-minimal normal closed subgroup of $G$, we have two cases to consider. Firstly, $R(N)=N$, i.e., $N$ is solvable and secondly, $R(N)$ is finite and hence trivial (since $R(N)$ is connected), i.e., $N$ is semisimple.
	
	{\bf Case 1:} $N$ is solvable. Then the commutator subgroup $[N,N]$ is a proper closed subgroup of $N$. It is moreover characteristic in $N$ and therefore normal in $G$. Moreover, $[N,N]$ is connected because $N$ is connected (\cite[Cor. 6.19]{Milne:AlgebraicGroupsTheTheoryOfGroupSchemesOfFiniteTypeOverAField}). As $N$ is almost-minimal, this implies $[N,N]=1$ so that $N$ is commutative. Thus $N=\Ga^r\times \Gm^s$ for some integers $r,s\geq 0$ (\cite[Cor. 16.15]{Milne:AlgebraicGroupsTheTheoryOfGroupSchemesOfFiniteTypeOverAField}).
	
	As the unipotent radical $R_u(N)=\Ga^r$ is a characteristic subgroup of $N$, it is normal in $G$. By the almost-minimality of $N$, we either have $\Ga^r=1$ (so that $N=\Gm^s$) or $\Ga^r=N$. So in either case, $N$ is a power of the additive or of the multiplicative group. 
	
	{\bf Case 2:} $N$ is semisimple. Let us briefly recall some generalities concerning semisimple algebraic groups. A semisimple algebraic group $G'$ is an almost-direct product of its almost-simple factors (\cite[Theorem~21.51]{Milne:AlgebraicGroupsTheTheoryOfGroupSchemesOfFiniteTypeOverAField}), i.e., there are only finitely many minimal connected normal closed subgroups of $G'$, say $G_1',\ldots,G_r'$. They are almost-simple, $[G_i',G_j']=1$ for $i\neq j$, and the morphism $G'_1\times\ldots\times G'_r\to G'$ is a quotient map with finite kernel.
	An element of the center of some $G_i'$ lies in the center of $G'$. Thus, if $G'$ is centerless, then the $G_i'$'s are centerless and therefore simple. Furthermore, if $(g'_1,\ldots,g'_r)$ lies in the kernel of $G'_1\times\ldots\times G'_r\to G'$, then $g_i'$ must lie in the center of $G'_i$ for $i=1,\ldots,r$. Therefore, a centerless semisimple algebraic group is a finite direct product of simple algebraic groups.

	%	 g_{i-1}g_{i+1}\ldots g_n\in G'_i\cap G'_1\ldots G'_{i-1}G'_{i+1}\ldots G'_n$$ lies in the center of $G_i'$.

	Let $Z$ be the center of $N$. As $N$ is semisimple, $Z$ is finite and because $Z$ is characteristic in $N$, we see that $Z$ is normal in $G$. Moreover, $N/Z$ is an almost-minimal normal closed subgroup of $G/Z$. As $N/Z$ is a centerless semisimple algebraic group, by the previous paragraph, $N/Z=N_1\times\ldots\times N_r$ for some simple algebraic groups $N_1,\ldots,N_r$. The conjugation action of $G/Z$ on $N/Z$ permutes $N_1,\ldots,N_r$. If $O\subseteq \{N_1,\ldots,N_r\}$ is an orbit under the $G/Z$-action, then all elements of $O$ are isomorphic and $\prod_{N_i\in O} N_i$ is a normal closed subgroup of $G/Z$. As $N/Z$ is almost-minimal in $G/Z$, we must have $O=\{N_1,\ldots,N_r\}$ and $N/Z=N_1^r$.	
\end{proof}

In the following three subsections we will prove Gasch\"{u}tz's Lemma for the three special cases provided by Lemma \ref{lemma: almost-minimal subgroups}. More precisely, we show that a quotient map $\f\colon G\twoheadrightarrow H$ with almost-minimal kernel $N=\ker(\f)$ is Gasch\"{u}tz in each of the cases $N=\Ga^r$, $N=\Gm^r$ and $N=N_1^r$ with $N_1$ a simple algebraic group.

\subsection{The case $N=\Ga^r$}
%Let $H$ be an algebraic group acting on $\Ga^r$. 
In this subsection we show that every quotient map $G\twoheadrightarrow H$ of algebraic groups with almost-minimal kernel $N=\Ga^r$ is Gasch\"{u}tz.
%provided that $\Ga^r$ is almost-minimal.
The following lemma is essential for the special case of the projection $G= \Ga\times H\twoheadrightarrow H$.
%\mw{I feel like the subcase $G=\Ga\times H$ was somewhat overlooked in our previous argument (towards the end of the proof of Lemma \ref{lemma: case Ga semidirect}). So I added the following two lemmas to deal with this comparatively easy case.}
Note that for an algebraic group $H$, the set $\Hom(H,\Ga)$ of all morphisms of algebraic groups from $H$ to $\Ga$ naturally is a $k$-vector space.

\begin{lemma} \label{lemma: strict inequality}
	Let $H$ be an algebraic group. Then 
	$\dim(\Hom(H,\Ga))<d(\Ga\times H).$
\end{lemma}
\begin{proof}
	Every algebraic group $H$ has a maximal unipotent quotient $H\twoheadrightarrow \pi_u(H)$. It is characterized by the property that every morphism from $H$ to a unipotent algebraic group factors uniquely through $H\twoheadrightarrow \pi_u(H)$ (see e.g., \cite[Lemma 2.4]{Wibmer:FreeProalgebraicGroups}). We then have 
	$$\Hom(H,\Ga)\simeq\Hom(\pi_u(H),\Ga)\simeq\Hom(\pi_u(H)/[\pi_u(H),\pi_u(H)],\Ga).$$
	As $\pi_u(H)/[\pi_u(H),\pi_u(H)]$ is commutative and unipotent, we have $\pi_u(H)/[\pi_u(H),\pi_u(H)]\simeq \Ga^m$ for some $m\geq 0$ (\cite[Prop. 14.32]{Milne:AlgebraicGroupsTheTheoryOfGroupSchemesOfFiniteTypeOverAField}). It follows that $\dim(\Hom(H,\Ga))=m$.
	
	On the other hand, as $\Ga^m$ is a quotient of $H$, we see that $\Ga^{m+1}$ is a quotient of $\Ga\times H$. So $$d(\Ga\times H)\geq d(\Ga^{m+1})=m+1>m=\dim(\Hom(H,\Ga))$$ as desired.
\end{proof}	

Recall that a closed subgroup $G$ of a direct product $G_1\times\ldots\times G_n$ of proalgebraic groups is a \emph{subdirect product} of $G_1,\ldots,G_n$ if all the projections $H\to G_i$ are quotient maps. Goursat's Lemma yields a description of a subdirect product as a fibre product. The following is a version of this for proalgebraic groups. 

\begin{lemma} \label{lemma: Goursat}
	Let $G_1,G_2$ be proalgebraic groups and let $H$ be a closed subgroup of $G_1\times G_2$ such that the projections $H\to G_1$ and $H\to G_2$ are quotient maps. Then there exists a proalgebraic group $G$ and quotient maps $G_1\twoheadrightarrow G$ and $G_2\twoheadrightarrow G$ such that $H=G_1\times_G G_2$.
\end{lemma}	
\begin{proof}
	For normal closed subgroups $N_1$, $N_2$ of $H$, it follows from the universal properties that
	\begin{equation}
		\label{eq: can isom}
		H/N_1\times_{H/N_1N_2}H/N_2\simeq H/(N_1\cap N_2).		
	\end{equation}	
	If we choose $N_1$ as the kernel of $H\twoheadrightarrow G_1$ and $N_2$ as the kernel of $H\twoheadrightarrow G_2$, then $N_1\cap N_2=1$ and (\ref{eq: can isom}) reduces to
	$G_1\times_G G_2= H$.
\end{proof}

We are now prepared to handle the special case of the projection $\Ga\times H\twoheadrightarrow H$.

\begin{lemma} \label{lemma: Gaschuetz for kernel Ga}
	Let $H$ be an algebraic group. Then the projection $\Ga\times H\twoheadrightarrow H$ is Gasch\"{u}tz.
\end{lemma}
\begin{proof}
	Let $h_1,\ldots,h_d$ be generators of $H$ such that $d(\Ga\times H)\leq d$. 
	Because $\dim(\Hom(H,\Ga))<d(\Ga\times H)\leq d$ by Lemma \ref{lemma: strict inequality}, the $k$-linear map
	$\Hom(H,\Ga)\to \Ga^d,\ \psi\mapsto (\psi(h_1),\ldots,\psi(h_d))$ is not surjective. So we can pick $(n_1,\ldots,n_d)\in\Ga^d$ not in the image of this map. We claim that $H(n_1,\ldots,n_d)=\langle (n_1,h_1),\ldots,(n_d,h_d)\rangle=G$. Suppose, for a contradiction, that $H(n_1,\ldots,n_d)$ is a proper subgroup of $\Ga\times H$.
	
	As any non-zero element generates $\Ga$ and not all $n_i$'s can be zero, the projection $H(n_1,\ldots,n_d)\to\Ga$ is a quotient map. By construction, also the projection $H(n_1,\ldots,n_d)\to H$ is a quotient map. So Lemma~\ref{lemma: Goursat} applies. As the only non-trivial quotient of $\Ga$ is $\Ga$ itself, we find that $H(n_1,\ldots,n_d)=\{(\psi(h),h)|\ h\in H\}$ for some morphism $\psi\colon H\to \Ga$. This, however, contradicts our choice of $n_1,\ldots,n_d$.
\end{proof}

For the convenience of the reader, we recall the basic facts about crossed homomorphisms (\cite[Section 15 a]{Milne:AlgebraicGroupsTheTheoryOfGroupSchemesOfFiniteTypeOverAField}).
Let $H$ be an algebraic group acting on an algebraic group $N$ (via algebraic group automorphisms). A morphism $\psi\colon H\to N$ of algebraic varieties is a \emph{crossed homomorphism} if $\psi(h_1h_2)=\psi(h_1)h_1(\psi(h_2))$ for all $h_1,h_2\in H$. If $\psi\colon H\to N$ is a crossed homomorphism, then $s\colon H\to N\rtimes H,\ h\mapsto (\psi(h),h)$ is a section of the projection $\pi\colon N\rtimes H\twoheadrightarrow H$, i.e., $s$ is a morphism of algebraic groups with $\pi s=\id_H$. Conversely, any section $s\colon H\to N\rtimes H$ comes from a unique crossed homomorphism $\psi\colon H\to N$.

The following lemma establishes the case $N=\Ga^r$ in case the quotient map $G\twoheadrightarrow H$ is split.

\begin{lemma} \label{lemma: case Ga semidirect}
	Let $H$ be an algebraic group acting on $N=\Ga^r$ such that $N$ is an almost-minimal normal closed subgroup of $G=N\rtimes H$. Then the projection $G\twoheadrightarrow H$ is Gasch\"{u}tz.
\end{lemma}
\begin{proof}
	Let $h_1,\ldots,h_d$ be generators of $H$ such that $d(G)\leq d$.
	For $n_1,\ldots,n_d\in N$ we set $H(n_1,\ldots,n_d)=\langle (n_1,h_1),\ldots,(n_d,h_d)\rangle\leq G$. We have to show that there exist $n_1,\ldots,n_d\in N$ such that $H(n_1,\ldots,n_d)=G$. Suppose, for a contradiction, that $H(n_1,\ldots,n_d)$ is a proper subgroup of $G$ for all $n_1,\ldots,n_d\in N$.
	
	The subgroup $N\cap H(n_1,\ldots,n_d)$ of $G$ is normalized by $H(n_1,\ldots,n_d)$ and by $N$, because $N$ is abelian. Therefore, it is normalized by $H(n_1,\ldots,n_d)N$. But $H(n_1,\ldots,n_d)N=G$ because $H(n_1,\ldots,n_d)N$ surjects onto $H$ and contains $N$. So $N\cap H(n_1,\ldots,n_d)$ is normal in $G$. As $N$ is almost-minimal, we have $N\cap H(n_1,\ldots,n_d)=N$ or $N\cap H(n_1,\ldots,n_d)$ is finite. In the first case, $N\subseteq H(n_1,\ldots,n_d)$ and so $G=H(n_1,\ldots,n_d)$. This contradicts our assumption that $H(n_1,\ldots,n_d)$ is a proper subgroup of $G$. Thus $N\cap H(n_1,\ldots,n_d)$ is finite and since $N=\Ga^r$ does not have any non-trivial finite subgroups, we have $N\cap H(n_1,\ldots,n_d)=1$ for all $n_1,\ldots,n_d\in N$.
	
	Therefore, for all $n_1,\ldots,n_d\in N$, the projection $G\twoheadrightarrow H$ restricts to an isomorphism $H(n_1,\ldots,n_d)\to H$ and so $H(n_1,\ldots,n_d)$ is of the form
	$$H(n_1,\ldots,n_d)=\{(\psi_{n_1,\ldots,n_d}(h), h)|\ h\in H\}\leq N\rtimes H$$
	for a crossed homomorphism $\psi_{n_1,\ldots,n_d}\colon H\to N$. %i.e., $\psi_{n_1,\ldots,n_d}$ is a morphism of varieties such that $\psi_{n_1,\ldots,n_d}(h_1h_2)=\psi_{n_1,\ldots,n_d}(h_1)h_1(\psi_{n_1,\ldots,n_d}(h_2))$ for all $h_1,h_2\in H$.% where $H$ acts on $N$ via conjugation.

	A classical result of Mostow (see \cite{Mostow:FullyReducibleSubgroupsOfAlgebraicGroups} for the original or \cite[Prop. 5.4.1]{Conrad:ReductiveGroupSchemes} for a modern treatment)
	%\cite[Chapter VII, Theorem~4.3]{Hochschild:BasicTheoryOfAlgebraicGroupsAndLieAlgebras} for a more modern treatment),
	relying on our characteristic zero assumption, states that there exists a reductive closed subgroup $R$ of $H$ such that $H=U\rtimes R$, where $U=R_u(H)$ is the unipotent radical of $H$. 
	
	Note that (in characteristic zero) every action of an algebraic group $G$ on a vector group $V$ is linear. (For example, because the exponential map $\exp\colon \Lie(V)\to V$ is a $G$-equivariant isomorphism of algebraic groups and the action of $G$ on $\Lie(V)$ is linear.) We can therefore regard $N=\Ga^r$ as a representation of $H$. 

	We next show that $U$ acts trivially on $N$, i.e., the elements of $U$ and $N$ commute so that $G=(N\times U)\rtimes R$. Regarding $N$ as a representation of $U$, it follows that there exists a non-zero $n\in N$ that is fixed by all $u\in U$. (In fact, in \cite[Paragraph 6.45]{Milne:AlgebraicGroupsTheTheoryOfGroupSchemesOfFiniteTypeOverAField}
	the existence of a non-zero fixed vector in every representation is taken as the definition of unipotent algebraic groups.)
	So $N'=\{n\in N|\ u(n)=n \ \forall \ u\in U\}$ is a non-trivial closed subgroup of $N$. Clearly it is normalized by $U$. Because $U$ is normalized by $R$, we have $u(r(n))=r((r^{-1}ur)(n))=r(n)$ for all $u\in U$, $r\in R$ and $n\in N'$ and so $r(n)\in N'$. Therefore $N'$ is normalized by $R$ and hence by $H$. Since $N$ is commutative, $N'$ is in fact normal in $G$. Because $N$ is an almost-minimal normal closed subgroup of $G$ and $N'$ is non-trivial, we must have $N=N'$, i.e., $U$ acts trivially on $N$ as claimed.

	We next show that $[U,U]\leq H(n_1,\ldots,n_d)$ for all $n_1,\ldots,n_d\in N$. For $u,u'\in U\leq H$, let $n, n'\in N$ be such that $(n,u), (n',u')\in H(n_1,\ldots,n_d)$. Then $[(n,u),(n',u')]\in H(n_1,\ldots,n_d)$ and 
	$$[(n,u),(n',u')]=([n,n'],[u,u'])=(0,[u,u'])$$
	because $U$ acts trivially on $N$ and $N$ is commutative. So $[U,U]\leq H(n_1,\ldots,n_d)$ as claimed.
	
	Because $[U,U]$ is a characteristic subgroup of $U$, we see that $[U,U]$ is normal in $H$. As, moreover, $U$ commutes with $N$, it follows that $[U,U]$ is normal in $G$. 	
	If the images of $(n_1,h_1),\ldots,(n_d,h_d)$ in $G/[U,U]=(N\times U/[U,U])\rtimes R$ generate $G/[U,U]$, then the contradiction $H(n_1,\ldots,n_d)=G$ follows, because $[U,U]\subseteq H(n_1,\ldots,n_d)$. We can thus assume, without loss of generality, that $U$ is abelian.
	
	So $U=\Ga^s$ for some $s\geq 0$ (\cite[Cor. 14.33]{Milne:AlgebraicGroupsTheTheoryOfGroupSchemesOfFiniteTypeOverAField}). As noted above, we can consider $N=\Ga^r$ and $U=\Ga^s$ as representations of $R$.
	
	%	
	%	Note that (in characteristic zero) every action of an algebraic group $G$ on a vector group $V$ is linear. (For example, because the exponential map $\exp\colon \Lie(V)\to V$ is a $G$-equivariant isomorphism of algebraic groups and the action of $G$ on $\Lie(V)$ is linear.??) We can therefore regard $N=\Ga^r$ and $U=\Ga^s$ as representations of $R$.
	%	
	
	We next describe the crossed homomorphisms $H\to N$. Again using our characteristic zero assumption, we have $H^1(R,N)=0$ (\cite[Prop. 15.15 and Cor. 22.43]{Milne:AlgebraicGroupsTheTheoryOfGroupSchemesOfFiniteTypeOverAField}), i.e., every crossed homomorphism $\psi\colon R\to N$ is principal, that is, of the form $\psi(r)=n^{-1}r(n)$ for some $n\in N$. Thus the restriction of a crossed morphism $H\to N$ to $R\to N$ is principal. On the other hand, the restriction $U\to N$ of a crossed homomorphism $H\to N$ is an $R$-equivariant morphism of algebraic groups because $U$ acts trivially on $N$. 
	Conversely, given an $R$-equivariant morphism $\varphi\in W=\Hom_R(U,N)$ and an element $n\in N$, we can define a crossed homomorphism $\psi_{\varphi,n}\colon H\to N$ by $\psi_{\varphi,n}((u,r))=\varphi(u)n^{-1}r(n)$ for $u\in U$ and $r\in R$.
	
	The crossed homomorphisms $\psi_{n_1,\ldots,n_d}\colon H\to N$ constructed above satisfy $\psi_{n_1,\ldots,n_d}(h_i)=n_i$ for $i=1,\ldots,d$. This implies that the map $$\eta\colon N\times W\to N^d,\ (n,\varphi)\mapsto(\psi_{\varphi,n}(h_1),\ldots,\psi_{\varphi,n}(h_d))$$ is surjective. Moreover, as $\psi_{\varphi+\varphi',nn'}=\psi_{\varphi,n}+\psi_{\varphi',n'}$ and $\psi_{\lambda\varphi,\lambda n}=\lambda\psi_{\varphi,n}$ for $\varphi,\varphi'\in W,\ n,n'\in N$ and $\lambda\in k$, we see that $\eta$ is $k$-linear.
	
	%	
	%	So, for all $n_1,\ldots,n_d\in N$, the crossed morphism $\psi_{n_1,\ldots,n_d}\colon H\to N$ restricts to a principal crossed  morphism $R\to N$. On the other hand, the restriction $U\to N$ of $\psi_{n_1,\ldots,n_d}$ to $U$ is a morphism of algebraic groups (and hence $k$-linear) because $U$ acts trivially on $N$. 
	%	So if $W=\Hom(U,N)$ is the finite dimensional $k$-vector space of morphism of algebraic groups from $U$ to $N$, then the map
	%	 
	%	 $\varphi: N \times W \to N^d$ given by sending $(x,\phi)$ to the $d$-tuple $((x,\phi)(x_1),\ldots ,(x,\phi)(x_d))$ (where $(x,\phi)$ is the crossed homomorphism from $H$ to $N$ defined by sending $r\to x^r-x$ on the reductive part and $u\mapsto \phi(u)$ on the unipotent part of $H=R\ltimes U$) must be onto. 
	%	
	%	Note that, by construction, $\psi_{n_1,\ldots,n_d}(h_i)=n_i$ for $i=1,\ldots,d$.
	%	
	
	By assumption, $G$ can be generated by $d$ elements. Say, $G=\langle (n_1',h_1'),\ldots,(n_d',h_d')\rangle$. For every crossed homomorphism $\psi\colon H\to N$ we have a proper closed subgroup $H_\psi=\{(n,h)\in G|\ \psi(h)=n\}$ of $G$. As not all $(n_i',h_i')$ can lie in $H_\psi$, we see that the $k$-linear map 
	$$\eta'\colon N\times W\to N^d,\ (n,\varphi)\mapsto(\psi_{\varphi,n}(h'_1),\ldots,\psi_{\varphi,n}(h'_d))$$ 
	is not surjective, because its image does not contain $(n_1',\ldots,n_d')$. Therefore, as $\eta$ is surjective, the kernel of $\eta'$ is non-trivial, i.e., there exists a non-zero $(n,\varphi)\in N\times W$ such that $\psi_{\varphi,n}(h'_1)=\ldots=\psi_{\varphi,n}(h'_d)=1$. As the kernel of a crossed homomorphism $H\to N$ is a closed subgroup of $H$ and $h_1',\ldots,h_d'$ generate $H$, we see that $\psi_{\varphi,n}$ is trivial. So $\varphi$ is trivial and $r(n)=n$ for all $r\in R$. Since $(n,\varphi)$ is non-zero, $n$ is non-zero and so 
	$N'=\{n'\in N|\ r(n')=n' \ \forall \ r\in R\}$ is non-trivial. Because $U$ acts trivially on $N$, we see that $N'$ is a normal closed subgroup of $G$. As $N$ is an almost-minimal normal closed subgroup of $G$, we find $N=N'$. That is, $R$ acts trivially on $N$. Since also $U$ acts trivially on $N$, we see that in fact $H$ acts trivially on $N$, so that $G=N\rtimes H=N\times H$.
	
	Because $N$ is an almost-minimal normal closed subgroup of $G$, we must have $N=\Ga$. So $N\rtimes H\twoheadrightarrow H$ is Gasch\"{u}tz by Lemma \ref{lemma: Gaschuetz for kernel Ga}; a contradiction.
\end{proof}

The following proposition establishes the case $N=\Ga^r$.

\begin{prop} \label{prop: Gaschuetz for Ga kernel}
	Let $\f\colon G\twoheadrightarrow H$ be a quotient map of algebraic groups with almost-minimal kernel $N=\Ga^r$. Then $\f$ is Gasch\"{u}tz.
\end{prop}	
\begin{proof}
	Suppose, for a contradiction, that $\f$ is not Gasch\"{u}tz. Then there exist $h_1,\ldots,h_d\in H$ with $H=\langle h_1,\ldots,h_d\rangle$ and $d(G)\leq d$ such that for all $g_1,\ldots,g_d\in G$ with $\f(g_i)=h_i$ for $i=1,\ldots,d$ the group $\langle g_1,\ldots,g_d\rangle$ is a proper closed subgroup of $G$.
	
	Let us fix $g_1,\ldots,g_d\in G$ with $\f(g_i)=h_i$ for $i=1,\ldots,d$ and set $H(g_1,\ldots,g_d)=\langle g_1,\ldots,g_d\rangle$. Then $H(g_1,\ldots,g_d)\cap N$ is normalized by $H(g_1,\ldots,g_d)$ and $N$, because $N$ is commutative. Thus $H(g_1,\ldots,g_d)\cap N$ is normal in $H(g_1,\ldots,g_d)N=G$. As $N$ is an almost-minimal normal closed subgroup of $G$, we have $H(g_1,\ldots,g_d)\cap N=N$ or $H(g_1,\ldots,g_d)\cap N$ is finite. In the first case, $N\subseteq H(g_1,\ldots,g_d)$ so that $G=H(g_1,\ldots,g_d)N=H(g_1,\ldots,g_d)$; a contradiction. In the second case, $H(g_1,\ldots,g_d)\cap N=1$, because $\Ga^r$ has no non-trivial finite subgroups. So $\f$ induces an isomorphism $H(g_1,\ldots,g_d)\to H$ and $\f$ is split, i.e., $\f$ can be identified with the projection $G=N\rtimes H\twoheadrightarrow H$. It follows from Lemma \ref{lemma: case Ga semidirect} that $\f$ is Gasch\"{u}tz; a contradiction. %This contradicts our asumption that $h_1,\ldots,h_d$ do not lift to generators of $G$. 
\end{proof}

\subsection{The case $N=\Gm^r$}
%Let $H$ be an algebraic group acting on $\Gm^r$. In this section we show that the projection $\Gm^r\rtimes H\twoheadrightarrow H$ is Gasch\"{u}tz, provided that $\Gm^r$ is almost-minimal. We first treat the special case that $H$ is finite.
In this section we show that every quotient map $G\twoheadrightarrow H$ of algebraic groups with almost-minimal kernel $N=\Gm^r$ is Gasch\"{u}tz. We need some preparatory lemmas.

\begin{lemma} \label{lemma: identity component}
	Let $T$ be an (abstract) subgroup of an algebraic group $G$ with Zariski closure $\overline{T}$ equal to $G$. Then $\overline{T\cap G^o}=G^o$.
\end{lemma}
\begin{proof}
	Writing $G=\cup_{i=1}^n G_i$ as the disjoint union of its connected components with $G_1=G^o$, we have $T=\cup_{i=1}^n(T\cap G_i)$ and therefore $G=\overline{T}=\overline{\cup_{i=1}^n(T\cap G_i)}=\cup_{i=1}^n\overline{T\cap G_i}.$ This implies $\overline{T\cap G_i}=G_i$ for $i=1,\dots,n$.
\end{proof}

\begin{lemma} \label{lemma: case Gm with H finite}
	Let $H$ be a finite group acting on $N=\Gm^r$ such that $N$ is an almost-minimal normal closed subgroup of $G=N\rtimes H$. Let $h_1,\ldots,h_d$ be generators of $H$ such that $d(G)\leq d$. Then $\langle (n_1,h_1),\ldots,(n_d,h_d)\rangle= G$ for all $n_1,\ldots,n_d\in N=\Gm^r$ whose coordinates generate a free abelian subgroup of $k^\times$ of rank $dr$.
\end{lemma}
\begin{proof}
	For $h_1',\ldots,h_d'\in H$ with $H=\langle h_1',\ldots,h_d'\rangle$ and $n_1',\ldots,n_d'\in N$, we first describe the conditions that $n_1',\ldots,n_d'$ must satisfy in order that $\langle (n'_1,h_1'),\ldots,(n'_d,h_d')\rangle=G$.
	
	Let $F_{x_1,\ldots,x_d}$ be the (abstract) free group with free generators $x_1,\ldots,x_d$. For a word $w\in F_{x_1,\ldots,x_d}$ and elements $g_1,\ldots,g_d$ of $G$, we write $w(g_1,\ldots,g_d)\in G$ for the image of $w$ under the group homomorphism $F_{x_1,\ldots,x_d}\to G$, determined by sending $x_i$ to $g_i$ for $i=1,\ldots,d$.
	Write $n'_i=(a'_{i,1},\ldots,a'_{i,r})\in \Gm^r$ for $i=1,\ldots,d$.
	Recall that every endomorphism of $\Gm^r$ can be described by Laurent monomials and that also all closed subgroups of $\Gm^r$ can be defined by equations of the form $P=1$ with $P$ a Laurent monomial in $r$ variables (\cite[Theorem 12.9]{Milne:AlgebraicGroupsTheTheoryOfGroupSchemesOfFiniteTypeOverAField}). It follows that, for every word $w\in F_{x_1,\ldots,x_d}$, we have $w((n'_1,h'_1),\ldots,(n'_d,h'_d))=(L_w,w(h'_1,\ldots,h'_d))\in N\rtimes H$ with $L_w\in\Gm^r$ an $r$-tuple of Laurent monomials in the $a'_{i,j}$'s. Let $T'$ be the (abstract) subgroup of $G$ generated by $(n_1',h_1'),\ldots,(n_d',h_d')$. Then $\overline{T'}=G$ if and only if $N=G^o\subseteq \overline{T'}$ and the latter condition is, via Lemma \ref{lemma: identity component}, equivalent to $\overline{T'\cap N}=N$. The elements of $T'\cap N$ are exactly those of the form $L_w$ with $w(h_1',\ldots,h_d')=1$ and these are Zariski dense in $N$ if and only if there is no Laurent monomial $P$ such that $P(L_w)=1$ for all $w$ with $w(h_1',\ldots,h_d')=1$. In summary, we see that $\langle (n'_1,h_1'),\ldots,(n'_d,h_d')\rangle=G$ is equivalent to certain collections of Laurent monomials in the $a'_{i,j}$'s not being equal to $1$.
	
	By assumption, there exist $(n_1',h_1'),\ldots,(n_d',h_d')\in N\rtimes H=G$ such that $\langle (n_1',h_1'),\ldots,(n_d',h_d')\rangle=G$. Thus some relevant Laurent monomials do not evaluate to $1$ on the coordinates of $n_1',\ldots,n_d'$. Let $n_1,\ldots,n_d\in \Gm^r$ be such that their coordinates generate a free subgroup of $k^\times$ of rank $dr$. Then the Laurent monomials considered above (for the $n_i'$'s) cannot evaluate to $1$ on the coordinates of $n_1,\ldots,n_d$. So $\langle(n_1,h_1'),\ldots,(n_d,h_d')\rangle=G$.
	
	Set $H(n_1,\ldots,n_d)=\langle (n_1,h_1),\ldots,(n_d,h_d)\rangle$. We have to show that $H(n_1,\ldots,n_d)=G$. Suppose, for a contradiction, that $H(n_1,\ldots,n_d)$ is a proper subgroup of $G$. The group $H(n_1,\ldots,n_d)\cap N$ is normalized by $H(n_1,\ldots,n_d)$ and by $N$, because $N$ is commutative. Thus $H(n_1,\ldots,n_d)\cap N$ is normal in $G=H(n_1,\ldots,n_d)N$. Because $N$ is an almost-minimal normal closed subgroup of $G$, we have $H(n_1,\ldots,n_d)\cap N=N$ or $H(n_1,\ldots,n_d)\cap N$ is finite. In the first case, $N\subseteq H(n_1,\ldots,n_d)$ and so $G=H(n_1,\ldots,n_d)N=H(n_1,\ldots,n_d)$; a contradiction. Therefore $H(n_1,\ldots,n_d)\cap N$ must be finite.
	As $H(n_1,\ldots,n_d)^o\leq H(n_1,\ldots,n_d)\cap G^o=H(n_1,\ldots,n_d)\cap N$, this implies that $H(n_1,\ldots,n_d)$ is finite. So $H(n_1,\ldots,n_d)$ agrees with the (abstract) group generated by $(n_1,h_1),\ldots,(n_d,h_d)$ and $H(n_1,\ldots,n_d)\cap N$ consists of elements of the form $L_w$, where $L_w$ is an $r$-tuple of Laurent monomials in the coordinates of $n_1,\ldots,n_d$. As there is no non-trivial torsion in the subgroup of $k^\times$ generated by the coordinates of $n_1,\dots,n_d$, such an $L_w$ can only be torsion if it is trivial. As $H(n_1,\ldots,n_d)\cap N$ is finite, this shows that $H(n_1,\ldots,n_d)\cap N=1$. So the projection $N\rtimes H\twoheadrightarrow H$ induces an isomorphism between $H(n_1,\ldots, n_d)$ and $H$.

	Because the $h_i$'s generate $H$, we can find words $w_1,\ldots,w_d\in F_{x_1,\ldots,x_d}$ such that $h_i'=w_i(h_1,\ldots,h_d)$ for $i=1,\ldots,d$. Define $m_1',\ldots,m_d'\in N$ by $(m_i',h'_i)=w_i((n_1,h_1),\ldots,(n_d,h_d))\in H(n_1,\ldots,n_d)$ for $i=1,\ldots,d$. Then the coordinates of the $m_i'$'s cannot generate a free abelian subgroup of $k^\times$ of rank $dr$ because otherwise, by the argument explained in the beginning of this proof, we would have $\langle (m_1',h_1'),\ldots,(m_d',h_d')\rangle=G$, yielding the contradiction $H(n_1,\ldots,n_d)=G$.
	
	As the $h'_i$'s generate $H$, we can also write $h_i=w_i'(h_1',\ldots,h_d')$ for $i=1,\ldots,d$, with words $w_1',\ldots,w_d'\in F_{x_1,\ldots,x_n}$. Then $w_i'((m_1',h_1'),\ldots,(m_d',h_d'))\in H(n_1,\ldots,n_d)$ and $w_i'((m_1',h_1'),\ldots,(m_d',h_d'))=(m_i,h_i)$ with $m_1,\ldots,m_d\in N$ for  $i=1,\ldots,d$. As $(m_i,h_i),(n_i,h_i)\in H(n_1,\ldots,n_d)$ and the projection $N\rtimes H\twoheadrightarrow H$ induces an isomorphism between  $H(n_1,\ldots,n_d)$ and $H$, we see that $m_i=n_i$ for $i=1,\ldots,d$. By construction, the $m_i$'s have coordinates in the subgroup of $k^\times$ generated by the coordinates of the $m_i'$'s, which, as noted above, do not generate a free abelian group of rank $dr$. This contradicts the assumption that the coordinates of the $n_i$'s generate a free abelian group of rank $rd$.
\end{proof}

We are now in a position to handle the split case of $N=\Gm^r$.

\begin{lemma} \label{lemma: case Gm semidirect product}
	Let $H$ be an algebraic group acting on $N=\Gm^r$ such that $N$ is an almost-minimal normal closed subgroup of $G=N\rtimes H$. Then the projection $G\twoheadrightarrow H$ is Gasch\"{u}tz.
\end{lemma}
\begin{proof}
	Let us fix $h_1,\ldots,h_d\in H$ such that $H=\langle h_1,\ldots,h_d\rangle$ and $d(G)\leq d$. For $n_1,\ldots,n_d\in N$ we set $H(n_1,\ldots,n_d)=\langle (n_1,h_1),\ldots,(n_d,h_d)\rangle\leq G$.

	We first show that $[H^o,H^o]\leq H(n_1,\ldots,n_d)$ for all $n_1,\ldots,n_d\in N$. 
	Let $h_1',h_2'\in H^o$. Because $H(n_1,\ldots,n_d)$ projects onto $H$, there exist $n_1',n_2'\in N$ such that $(n_1',h_1'), (n_2',h_2')\in H(n_1,\ldots,n_d)$. Then also the commutator $[(n_1',h_1'),(n'_2,h'_2)]$ lies in $H(n_1,\ldots,n_d)$.	
	By the rigidity of tori (\cite[Cor. 12.37]{Milne:AlgebraicGroupsTheTheoryOfGroupSchemesOfFiniteTypeOverAField}), the identity component $H^o$ of $H$ acts trivially on $\Gm^r$. Therefore $[(n_1',h_1'),(n'_2,h'_2)]=([n_1',n_2'],[h_1',h_2'])=(1,[h_1',h'_2])$. It follows that $[h'_1,h_2']\in H(n_1,\ldots,n_d)$ and so $[H^o,H^o]\leq H(n_1,\ldots,n_d)$ as claimed.
	
	Because $[H^o,H^o]$ is a characteristic subgroup of $H^o$, it is normal in $H$. As $H^o$ acts trivially on $N$, it is in fact normal in $G$. If the images $\overline{h_1},\ldots,\overline{h_d}$ of $h_1,\ldots,h_d$ in $H/[H^o,H^o]$ have lifts to generators $(n_1,\overline{h_1}),\ldots,(n_d,\overline{h_d})$ of $G/[H^o,H^o]=N\rtimes H/[H^o,H^o]$, then $H(n_1,\ldots,n_d)[H^o,H^o]=G$. But as $[H^o,H^o]\leq H(n_1,\ldots,n_d)$, this means $H(n_1,\ldots,n_d)=G$. Moreover, $N$ is an almost-minimal normal closed subgroup of $N\rtimes H/[H^o,H^o]$. In summary, this shows that we can assume, without loss of generality, that $H^o$ is abelian.
	
	So $H^o=\Gm^s\times\Ga^t$ for some $s,t\geq 0$ (\cite[Cor. 16.15]{Milne:AlgebraicGroupsTheTheoryOfGroupSchemesOfFiniteTypeOverAField}). We next show that $\Ga^t\leq H(n_1,\ldots,n_d)$ for all $n_1,\ldots,n_d\in N$. Let $h\in\Ga^t\leq H$. Then there exists an $n\in N$ such that $(n,h)\in H(n_1,\ldots,n_d)$. Recall that, according to the Jordan decomposition in algebraic groups (\cite[Section 9, b]{Milne:AlgebraicGroupsTheTheoryOfGroupSchemesOfFiniteTypeOverAField}), an element $g$ of an algebraic group can uniquely be written as a commuting product $g=g_sg_u=g_ug_s$ of a semisimple and a unipotent element. Moreover, since morphisms of algebraic groups preserve the Jordan decomposition, if $g$ belongs to a closed subgroup, also $g_u$ and $g_s$ must belong to it. Now $n$ is semisimple, $h$ is unipotent and $n$ and $h$ commute. Therefore, $n=g_s$ and $h=g_u$ are the semisimple and unipotent parts of $g=(n,h)\in H(n_1,\ldots,n_d)$. Consequently, $h$ belongs to $H(n_1,\ldots,n_d)$ and $\Ga^t\leq H(n_1,\ldots,n_d)$ as claimed.
	
	As $\Ga^t$ is a characteristic subgroup of $H^o$, it is normal in $H$ and hence normal in $G$ because its elements commute with elements of $N$. As above, by considering the morphism $G/\Ga^t=N\rtimes H/\Ga^t\twoheadrightarrow H/\Ga^t$, we can reduce to the case $t=0$. So we can assume, without loss of generality, that $H^o=\Gm^s$.
	
	Let $T$ be the (abstract) subgroup of $H$ generated by $h_1,\ldots,h_d$. As $T/T\cap H^o$ injects into $H/H^o$, we see that $T\cap H^o$ has finite index in $T$. As a finite index subgroup of a finitely generated group is itself finitely generated (Schreier's lemma), it follows that $T\cap H^o\leq \Gm^s$ is a finitely generated abelian group. Fix a generating set $(a_{1,1},\ldots,a_{1,s}),\ldots,(a_{m,1},\ldots,a_{m,s})\in \Gm^s$ of $T\cap H^o$ and let $A$ be the (abstract) subgroup of $\Gm=k^\times$ generated by all $a_{i,j}$'s. As an abelian group, $k^\times$ is isomorphic to $\mathbb{Q}/\mathbb{Z}\oplus V$, where $V$ is $\mathbb{Q}$-vector space of dimension $|k|$. In particular, for a finitely generated subgroup $A$ of $k^\times$, one can find a free abelian subgroup $B$ of $k^\times$ of arbitrary high rank such that $A\cap B=1$. We can thus choose $n_1,\ldots,n_d\in \Gm^r$ in such a way that the coordinates of $n_1,\ldots,n_d$ generate a free abelian subgroup $B$ of $k^\times$ of rank $dr$ with $B\cap A=1$.
	
	Eventually we will show that $H(n_1,\ldots,n_d)=G$, so that $G\twoheadrightarrow H$ is Gasch\"{u}tz as desired. To this end, we first show that $H^o\leq H(n_1,\ldots,n_d)$. Set $Z=H(n_1,\ldots,n_d)\cap N$. Then $Z$ is normalized by $H(n_1,\ldots,n_d)$ and by $N$ (because $N$ is commutative). Thus $Z$ is normal in $H(n_1,\ldots,n_d)N=G$. Because $N$ is an almost-minimal normal closed subgroup of $G$, we have $Z=N$ or $Z$ is finite. In the first case, $N\subseteq H(n_1,\ldots,n_d)$ and so $G=H(n_1,\ldots,n_d)N=H(n_1,\ldots,n_d)$ shows that $G\twoheadrightarrow H$ is Gasch\"{u}tz. For a contradiction, we can thus suppose that $Z$ is finite. The image $\overline{H(n_1,\ldots,n_d)}$ of $H(n_1,\ldots,n_d)$ in $G/Z=(N/Z)\rtimes H$ then maps isomorphically onto $H$ under the projection $(N/Z)\rtimes H\twoheadrightarrow H$. Thus there exists a crossed homomorphism $\psi\colon H\to N/Z$ such that $\overline{H(n_1,\ldots,n_d)}=\{(\psi(h),h)|\ h\in H\}\leq (N/Z)\rtimes H$. 
	
	As in the proof of Lemma \ref{lemma: case Gm with H finite}, we will use words $w$ from the abstract free group $F_{x_1,\dots,x_d}$.	Let $h\in T\cap H^o\leq \Gm^s$. Then $h\in A^s$ and $h=w(h_1,\ldots,h_d)$ for some word $w$. Define $n\in N$ by $(n,h)=w((n_1,h_1),\ldots,(n_d,h_d))\in H(n_1,\ldots,n_d)$. Because the $h_i$'s act on $N=\Gm^r$ via Laurent monomials, we see that $n\in B^r$. As $(n,h)\in H(n_1,\dots,n_d)$, the image $\overline{(n,h)}=(\overline{n},h)\in \overline{H(n_1,\ldots,n_d)}\leq (N/Z)\rtimes H$ of $(n,h)$ in $(N/Z)\rtimes H$ satisfies $\psi(h)=\overline{n}$. As $H^o$ acts trivially on $N$, the restriction of $\psi\colon H\to N/Z$ to $H^o$ is a morphism of algebraic groups. Thus, fixing an isomorphism $N/Z\simeq \Gm^r$, the restriction $H^o\to \Gm^r$ of $\psi$ is given by Laurent monomials. Therefore $\psi(h)\in A^r$. As the morphism $N\to N/Z\simeq \Gm^r$ is also given by Laurent monomials and $n\in B^r$, we see that $\overline{n}\in B^r$. So $\psi(h)=\overline{n}\in A^r\cap B^r$. As $A\cap B=1$, this shows that $n\in Z$. 
	
	Because $Z$ is finite, there exists an integer $m\geq 1$ such that $z^m=1$ for all $z\in Z$. Remembering that $H^o$ acts trivially on $N$, we have $(n,h)^m=(n^m,h^m)=(1,h^m)\in H(n_1,\ldots,n_d)$. We have thus shown that $h^m\in H(n_1,\ldots,n_d)$ for all $h\in T\cap H^o$. As $T$ is Zariski dense in $H$, also $T\cap H^o$ is Zariski dense in $H^o=\Gm^s$ (Lemma \ref{lemma: identity component}). This means that there is no non-zero $\alpha=(\alpha_1,\ldots,\alpha_s)\in\mathbb{Z}^s$ such that $h_1^{\alpha_1}\ldots h_s^{\alpha_s}=1$ for all $h=(h_1,\ldots,h_s)\in T\cap H^o$. Thus there is also no non-zero $\alpha=(\alpha_1,\ldots,\alpha_s)\in\mathbb{Z}^s$ such that $h_1^{m\alpha_1}\ldots h_s^{m\alpha_s}=1$ for all $h=(h_1,\ldots,h_s)\in T\cap H^o$, i.e., the group $\{h^m|\ h\in T\cap H^o\}$ is Zariski dense in $H^o$. Therefore $H^o\leq H(n_1,\ldots,n_d)$ as claimed.
	
	As $H^o$ acts trivially on $N$, $H^o$ is normal in $G$ and we can consider the projection $G/H^o=N\rtimes (H/H^o)\twoheadrightarrow H/H^o$. As $N$ is an almost-minimal normal closed subgroup of $N\rtimes (H/H^o)$, Lemma \ref{lemma: case Gm with H finite} shows that $\langle (n_1,\overline{h_1}),\ldots,(n_d,\overline{h_d})\rangle =N\rtimes (H/H^o)$, i.e., the image of $H(n_1,\ldots,n_d)$ in $G/H^o$ is all of $G/H^o$. Since $H^o\leq H(n_1,\ldots,n_d)$, we find $H(n_1,\ldots,n_d)=G$ as desired.
\end{proof}

The following proposition establishes the case $N=\Gm^r$.

\begin{prop} \label{prop: Gaschuetz for Gm kernel}
	Let $\f\colon G\twoheadrightarrow H$ be a quotient map with almost-minimal kernel $N=\Gm^r$. Then $\f$ is Gasch\"{u}tz.
\end{prop}
\begin{proof}
	Let $h_1,\ldots, h_d$ be generators of $H$ with $d(G)\leq d$. We have to show that $h_1,\ldots,h_d$ lift to generators of $G$. Let $g_1,\ldots,g_d\in G$ with $\f(g_i)=h_i$ for $i=1,\ldots,d$ and set $G'=\langle g_1,\ldots,g_d\rangle$. Then $G'\cap N$ is normalized by $G'$ and $N$ because $N$ is abelian. Thus $G'\cap N$ is normal in $G'N=G$. As $N$ is almost-minimal, we must have $G'\cap N=N$ or $G'\cap N$ is finite. In the first case, $N\leq G'$ and so $G'=G'N=G$. Thus $g_1,\ldots,g_d$ are generators of $G$ that lift $h_1,\ldots,h_d$.
	
	Hence we may assume that $D=G'\cap N$ is finite. Then $N/D\simeq \Gm^r$ is an almost-minimal normal closed subgroup of $G/D$ and the induced morphism $\overline{\f}\colon G/D\twoheadrightarrow H$ with kernel $N/D$ is split because $\overline{\f}$ restricts to an isomorphism $G'/D\to H$. So it follows from Lemma \ref{lemma: case Gm semidirect product} that $\overline{\f}$ is Gasch\"{u}tz. Now $\f\colon G\twoheadrightarrow H$ factors as $G\twoheadrightarrow G/D\twoheadrightarrow H$. As the kernel $D$ of $G\twoheadrightarrow G/D$ is finite, it follows from Lemma \ref{lemma: finite kernel is Gaschuetz} that $G\twoheadrightarrow G/D$ is Gasch\"{u}tz. Therefore $\f\colon G\twoheadrightarrow H$ is Gasch\"{u}tz by Lemma \ref{lemma: composition of Gaschuetz is Gaschuetz} and consequently $h_1,\ldots,h_d$ can be lifted to a generating set of $G$.
\end{proof}

\subsection{The case $N=N_1^r$} In this section we show that every quotient map $G\twoheadrightarrow H$ of algebraic groups with almost-minimal kernel $N=N_1^r$, where $N_1$ is a simple algebraic group, is Gasch\"{u}tz.

\medskip

To treat the case where $G=N\rtimes A$ with $A$ a finite group acting on $N$, we will need some preparatory results about automorphisms of $N$ and regular elements. 
Recall (see e.g. \cite{Steinberg:RegularElementsOfSemisimpleAlgebraicGroups}) that a semisimple element of a connected semisimple algebraic group $N$ is \emph{regular} if its centralizer has minimal dimension, or equivalently, its conjugacy class has maximal dimension. This minimal dimension is then equal to the rank of $N$. A semisimple element of $N$ is regular if and only if it is contained in a unique maximal torus (\cite[2.11]{Steinberg:RegularElementsOfSemisimpleAlgebraicGroups})). 

The main ingredient to show that the projection $N\rtimes A\twoheadrightarrow A$ is Gasch\"{u}tz, is a kind of generic $2$\=/generation for not necessarily connected semisimple algebraic groups (Lemma \ref{lemma: generic 2 generation}), and this is our main goal for now. 

\begin{lemma} \label{lemma: Springer Steinberg}
	Let $a$ be an automorphism of a connected semisimple algebraic group $N$ of finite order. Then $N^a=\{n\in N|\ a(n)=n\}$ is a reductive closed subgroup of $N$ and every maximal torus of $N^a$ is contained in a unique maximal torus of $N$.
\end{lemma}
\begin{proof}
	We note that (in characteristic zero) an automorphism of finite order is semisimple. The claim of the lemma is then contained in the beginning of the proof of \cite[Chapter II, Theorem~ 5.16]{SpringerSteinberg:ConjugacyClasses}. 
\end{proof}

\begin{lemma} \label{lemma: get nice regular element}
	Let $N_1$ be a simple algebraic group and $N=N_1^r$ for some $r\geq 1$. Let $a$ be an automorphism of $N$ of finite order. Then there exists a regular semisimple element $g=(g_1,\ldots,g_r)\in N_1^r=N$ such that $a(g)=g$, $g_i^2\neq 1$ for $i=1,\ldots,r$ and $g$ is contained in a maximal torus of $N^a=\{n\in N|\ a(n)=n\}$.	
\end{lemma}	
\begin{proof}
	Let $S$ be a maximal torus of $N^a=\{n\in N|\ a(n)=n\}$ and $g\in S$ a generator of $S$ (as an algebraic group). If $T$ is a maximal torus of $N$ containing $g$, then $T$ contains $S$ and so $T$ is the unique torus of $N$ containing $g$ by Lemma~\ref{lemma: Springer Steinberg}. Thus $g$ is a regular semisimple element of $N$, contained in a maximal torus of $N^a$ satisfying $a(g)=g$. Suppose one of the coordinates $g_i$ of $g$ has finite order. Then the image of $S$ under the projection to the $i$-th coordinate is disconnected unless $g_i=1$. As the image is connected, we must have $g_i=1$. But this contradicts the assumption that $g$ is regular because the $i$-th component of $N_1^r$ would then centralize $g$. This shows that $g_i$ does not have finite order for $i=1,\ldots,r$. In particular, $g_i^2\neq 1$ for $i=1,\ldots,r$.
\end{proof}

\begin{lemma} \label{lemma: get conjugacy class of regular element}
	Let $N_1$ be a simple algebraic group and $N=N_1^r$ for some $r\geq 1$.
	Let $a$ be an automorphism of $N$ of finite order $s$ and consider the morphism $\psi\colon N\to N,\ n\mapsto na(n)\ldots a^{s-1}(n)$ of varieties.	
	 Then there exists a regular semisimple element $g=(g_1,\ldots,g_r)\in N_1^r=N$ such that $g_i^2\neq 1$ for $i=1,\ldots,r$ and the conjugacy class of $g$ in $N$ is contained in $\psi(N)$.	
\end{lemma}
\begin{proof}
By Lemma \ref{lemma: get nice regular element} there exists a maximal torus $T$ of $N^a=\{n\in N|\ a(n)=n\}$ and an element $g\in T$ such that $g=(g_1,\ldots,g_r)\in N_1^r$ is a regular semisimple element of $N$ with $g_i^2\neq 1$ for $i=1,\ldots,r$. As the morphism $T\to T,\ t\mapsto t^s$ is surjective, we can find a $t\in T$ such that $t^s=g$. So $\psi(t)=t^s=g$.

	Let $m\in N$. We have to show that $mgm^{-1}\in\psi(N)$. For $n=mta(m)^{-1}\in N$ we have 
	\begin{align*}
		\psi(n)&=mta(m)^{-1}a(mta(m)^{-1})\ldots a^{s-2}(mta(m)^{-1})a^{s-1}(mta(m)^{-1})\\
		&=mta(t)\ldots a^{s-1}(t)a^s(m)^{-1}=mt^sm^{-1}=mgm^{-1}
	\end{align*}
	so that  $mgm^{-1}\in\psi(N)$ as desired.	
%	
%	
%	 for some $n\in N$. To this end, we consider the  algebraic group $G=N\rtimes A$. Note that $A_m=mAm^{-1}$ is a finite subgroup of $G$ and $G$ can also be written as the semidirect product $G=N\rtimes A_m$. Furthermore, the component $(N,a)$ of $N\rtimes A$ and the component $(N,mam^{-1})$ of $N\rtimes A_m$ represent the same connected component of $G$. We have $mN^am^{-1}=N^{mam^{-1}}$ and for $(msm^{-1},mam^{-1})\in N\rtimes A_m$ we have $(msm^{-1},mam^{-1})^e=(mgm^{-1},1)$. So if we write $(msm^{-1},mam^{-1})\in  N\rtimes A_m$ as an element $(n,a)$ of $N\rtimes A$, then $(n,a)^e=mgm^{-1}=c$. Thus $c\in \langle (n,a)\rangle$.    
\end{proof}

We will need the following basic result of the subdirect products in $N_1^r$.

\begin{lemma} \label{lemma: special Goursat}
	Let $N_1$ be a simple algebraic group and $H$ a closed proper subgroup of $N_1^r$ ($r\geq 2$) that is a subdirect product, i.e., all the projections $H\to N_1$ are surjective. Then there exist indices $1\leq i<j\leq r$ and an automorphism $\psi$ of $N_1$ such that $H$ is contained in the closed subgroup $\{(n_1,\ldots,n_r)\in N_1^r|\ \psi(n_i)=n_j\}$ of $N_1^r$.
\end{lemma}	
\begin{proof}
	The projection $H\to N_1$ to the first factor is surjective, and because $H$ is a proper subgroup of $N_1^r$, the projection $H\to N_1^r$ to all factors is not surjective. Let $1\leq i<r$ be maximal such that the projection $H\to N_1^i$ to the first $i$ factors is surjective. Let $H'\leq N_1^{i+1}$ be the image of $H$ under the projection to the first $i+1$ factors. Then $H'$ projects surjectively onto the first $i$ factors and onto the last factor of $N_1^{i+1}$. In other words, $H'$ is a subdirect product of $N_1^{i}$ and $N_1$. By Lemma \ref{lemma: Goursat}, there exist an algebraic group $G$ and quotient maps $N_1^i\twoheadrightarrow G$ and $N_1\twoheadrightarrow G$ such that $H'=N_1^i\times_G N_1$. But because $N_1$ is simple, the quotient map $N_1\twoheadrightarrow G$ is either trivial (i.e., $G=1$) or an isomorphism. In the first case, we have $H'=N_1^{i+1}$, contradicting the maximality of $i$. Thus $G=N_1$. Furthermore, exactly one factor, say the $\ell$-th one ($1\leq \ell\leq i$) does not lie in the kernel of the quotient map $N_1^i\twoheadrightarrow N_1$. This quotient map then factors through the projection $\pi_\ell\colon N_1^i\to N_1$ to the $\ell$-th coordinate. So we can write $N_1^i\twoheadrightarrow N_1$ as $N_1^i\xrightarrow{\pi_\ell}N_1\xrightarrow{\psi}N_1$ for some automorphism $\psi$ of $N_1$. By construction, this shows that $\psi(n_\ell)=n_{i+1}$ for all $(n_1,\ldots,n_r)\in H$.
\end{proof}

It is well known that any connected semisimple algebraic group can be generated by two elements. The earliest reference for this seems to be \cite[Theorem 3]{Tits:FreeSubgroupsInLinearGroups}, where, in particular, it is shown that a connected semisimple algebraic group contains a Zariski dense free group on two generators.
We will need some more precise results on generic $2$-generation of connected semisimple algebraic groups. Indeed, we need a generalization to not necessarily connected groups (Lemma \ref{lemma: generic 2 generation}).

\begin{lemma} \label{lemma: generic two generation}
	Let $N$ be a connected semisimple algebraic group. Then the set 
	$$\{(n,m)\in N^2|\ \langle n,m\rangle=N \}$$
	is a non-empty open subset of $N^2$.	
\end{lemma}
\begin{proof}
	This is a special case of \cite[Theorem 4.1]{BreuillardGreenGuralnikTao:StronglyDenseFreeSubgroupsOfSemisimpleAlgebraicGroups}.
	%	
	%	In the introduction to \cite{BurnessGerhardtGuralnik:TopologicalGenerationOfExceptionalAlgebraicGroups} it is claimed that the above set is non-empty and open for any semisimple algebraic group $N$. However, only the case when $N$ is a simply connected almost-simple algebraic group is formally stated as a theorem and that result is attributed to \cite{Guralnik:SomeApplicationsOfSubgroupStructureToProbabilisticGenerationAndCoversOfCurves}.
\end{proof}

\begin{lemma} \label{lemma: Guralnik}
	Let $N_1$ be a simple algebraic group and $N=N_1^r$ for some $r\geq 1$. Let $g=(g_1,\ldots,g_r)\in N_1^r$ be a regular semisimple element of $N$ with $g_i^2\neq 1$ for $i=1,\ldots,r$. Similarly, let $h=(h_1,\ldots,h_r)\in N_1^r$ be a regular semisimple element of $N$ with $h_i^2\neq 1$ for $i=1,\ldots,r$ and let $C_g,C_h$ be the conjugacy classes of $g$ and $h$ respectively. Then the set $\{(n,m)\in C_g\times C_h|\ \langle n,m\rangle=N\}$ is a non-empty Zariski open subset of $C_g\times C_h$.  
\end{lemma}
\begin{proof}
	The openness follows from Lemma \ref{lemma: generic two generation}. So we only need to prove the non-emptiness.
	For $r=1$ this is Theorem 8.6 in \cite{GuralnikSaxl:GenerationOfFiniteAlmostSimpleGroupsByConjugates} and we will reduce to this case via Lemma \ref{lemma: special Goursat}. 
	
	Henceforth we assume $r\geq 2$.
	Note that each $g_i$ and $h_i$ is a regular semisimple element of $N_1$. Furthermore, $C_g=C_{g_1}\times\ldots\times C_{g_r}$ and similarly for $h$.
	By the $r=1$ case, there exists for every $i=1,\ldots,r$ a non-empty open subset $U_i$ of $C_{g_i}\times C_{h_i}$ such that $\langle n_1,m_1\rangle=N_1$ for all $(n_1,m_1)\in U_i$. Recall that the automorphism group $\Aut(N_1)$ of $N_1$ naturally has the structure of an algebraic group with identity component, the group of inner automorphisms equal to $N_1$ (\cite[Expose XXIV]{SGA3.3} or \S 5, in particular \S 5.3. (c), in \cite{BorelSerre:TheoremsDeFinitudeEnCohomologieGaloisienne})). For any pair $1\leq i<j\leq r$ consider the set
	$$V_{i,j}=\big\{(n_1,\ldots,n_r,m_1,\ldots,m_r)\in C_g\times C_h\subseteq N^2=N_1^{2r}|\ \psi(n_i)=n_j,\ \psi(m_i)=m_j \text{ for some } \psi\in\Aut(N_1)\big\}.$$
	Then $V_{i,j}$ is contained in the image of the morphism
	$$\alpha_{i,j}\colon C_{g_1}\times\ldots\times C_{g_{j-1}}\times C_{g_{j+1}}\times\ldots\times C_{g_r}\times C_{h_1}\times\ldots\times C_{h_{j-1}}\times C_{h_{j+1}}\times\ldots\times C_{h_r}\times\Aut(N_1)\to N^2$$
	of varieties, that inserts $\psi(n_i)$ and $\psi(m_i)$ in the $j$-th positions.
	Because $g_i$ is regular in $N_1$, the dimension of $C_{g_i}$ has the maximal possible value $\dim(N_1)-\rank(N_1)$. Therefore, the dimension of the domain of $\alpha_{i,j}$ is $$2(r-1)(\dim(N_1)-\rank(N_1))+\dim(N_1)=(2r-1)\dim(N_1)-2(r-1)\rank(N_1).$$ It follows that $V_{i,j}$ is contained in a closed subset of $C_g\times C_h$ of dimension at most $(2r-1)\dim(N_1)-2(r-1)\rank(N_1)$. The dimension of $C_g\times C_h$ is $2r(\dim(N_1)-\rank(N_1))$. 
	We have $(2r-1)\dim(N_1)-2(r-1)\rank(N_1)<2r(\dim(N_1)-\rank(N_1))$ if and only if $\dim(N_1)>2\rank(N_1)$.
	As $\dim(N_1)=\rank(N_1)+|R|$, where $R$ is the set of roots for $N_1$ (with respect to some maximal torus), this is equivalent to $|R|>\rank(N_1)$ and the latter is clearly a true statement because $\rank(N_1)$ is the number of simple roots.
%	
%	
%	, a quick inspection of a classification table for the irreducible root systems shows that  $\dim(N_1)>2\rank(N_1)$ does indeed hold for any simple algebraic group $N_1$. 
Thus $V_{i,j}$ is contained in a closed proper subset of $C_g\times C_h$. Therefore, removing all the $V_{i,j}$'s from $U_1\times\ldots\times U_r$ will leave us with a non-empty subset $V$ of $C_g\times C_h$.
	
	We claim that $\langle n,m\rangle=N$ for every $(n,m)\in V$. 
	Suppose, for a contradiction, that $H=\langle n,m\rangle$ is a proper subgroup of $N$ for some $(n,m)\in V$.	
	By the choice of $U_i$, the image of $H$ under each projection is $N_1$. So $H$ is a subdirect product and it follows from Lemma \ref{lemma: special Goursat} that there exist indices $1\leq i<j\leq r$ and an automorphism $\psi$ of $N_1$ such that $\psi(h_i)=h_j$ for every $h=(h_1,\ldots,h_r)\in H$. But that would mean that $(n,m)\in V_{i,j}$; a contradiction.
\end{proof}	

We are now prepared to prove generic $2$-generation for some not necessarily connected semisimple groups.

\begin{lemma} \label{lemma: generic 2 generation}
	Let $N_1$ be a simple algebraic group and $N=N_1^r$. Let $A$ be a finite group acting on $N$ such that $A$ is generated by $a,a'\in A$. Then the set $\{(n,n')\in N^2|\ \langle (n,a),(n',a')\rangle=N\rtimes A\}$ contains a non-empty  open subset of $N^2$.
\end{lemma}
\begin{proof}
	Let $s,s'\geq 1$ be the order of $a$ and $a'$ respectively. Note that $(n,a)^s=(na(n)\ldots a^{s-1}(n),1)$ so that $na(n)\ldots a^{s-1}(n)\in N\cap \langle (n,a),(n',a')\rangle$ for all $n\in N$. Similarly, $n'a'(n')\ldots a'^{s'-1}(n')\in N\cap \langle (n,a),(n',a')\rangle$ for all $n'\in N$. We consider the morphism $$\psi\colon N^2\to N^2,\ (n,n')\mapsto (na(n)\ldots a^{s-1}(n),n'a'(n')\ldots a'^{s'-1}(n'))$$
	of varieties. Let $U=\{(n,n')\in N^2|\ \langle n,n'\rangle=N \}$ be the non-empty open subset from Lemma~\ref{lemma: generic two generation}. If $\psi(N^2)\cap U\neq\emptyset$, then $\psi^{-1}(U)$ is a non-empty open subset of $N^2$ satisfying the requirement of the lemma, because $N=\langle \psi(n,n')\rangle\leq \langle (n,a),(n',a')\rangle$ for all $(n,n')\in \psi^{-1}(U)$.
	
	It thus suffices to show that $\psi(N^2)\cap U\neq\emptyset$. By Lemma~\ref{lemma: get conjugacy class of regular element} there exist regular semisimple elements $g=(g_1,\ldots,g_r)\in N_1^r$ and $h=(h_1,\ldots,h_r)\in N_1^r$ with $g_i^2\neq 1$ and $h_i^2\neq 1$ for $i=1,\ldots,r$ such that for their conjugacy classes $C_g,C_h$ in $N$, we have $C_g\times C_h\subseteq\psi(N^2)$. On the other hand, $(C_g\times C_h)\cap U\neq\emptyset$ by Lemma \ref{lemma: Guralnik}. Thus $\psi(N^2)\cap U\neq\emptyset$ as desired.
\end{proof}

%\mw{There might/should be a reference for the following lemma but I did not find anything.}
%
%\begin{lemma}
%	Let $k_0$ be an algebraically closed subfield of $\C$ and let $X$ be an affine algebraic variety over $k_0$. Then $X(k_0)$ is dense in $X(\C)$ with respect to the analytic topology on $X(\C)$.
%\end{lemma}
%\begin{proof}
%	?? Maybe just assume smooth and then use etale cover $U\to \A^n_{k_0}$ The local analytic  inverse maps $k_0$ points to $k_0$ points. 
%\end{proof}
%

With Lemma \ref{lemma: generic 2 generation} in hand, it is now a simple matter to show that the projection $N\rtimes A\twoheadrightarrow A$ is Gasch\"{u}tz.

\begin{lemma} \label{lemma: case N1 finite}
	Let $N_1$ be a simple algebraic group and $A$ a finite group acting on $N=N_1^r$. Furthermore, let $a_1,\ldots,a_d$ be generators of $A$ with $d(N\rtimes A)\leq d$. 
	Then the set $$\{(n_1,\ldots,n_d)\in N^d|\ \langle (n_1,a_1),\ldots,(n_d,a_d)\rangle=N\rtimes A\}$$ contains a non-empty open subset of $N^d$. In particular, the projection $N\rtimes A\twoheadrightarrow A$ is Gasch\"{u}tz.  
\end{lemma}	
\begin{proof}
	Because $N$ is non-commutative, we have $2\leq d(N\rtimes A)\leq d$.
	By Lemma \ref{lemma: generic 2 generation}, there exists a non-empty open subset $U$ of $N^2$ such that $\langle (n_1,a_1) ,(n_2,a_2)\rangle =N\rtimes\langle a_1,a_2\rangle$ for all $(n_1,n_2)\in U$. Set $V=U\times N^{d-2}\subseteq N^d$. Then $V$ is a non-empty open subset of $N^d$ and $\langle (n_1,a_1),\ldots,(n_d,a_d)\rangle=N\rtimes A$ for all $(n_1,\ldots,n_d)\in V$, because $N\leq \langle (n_1,a_1) ,(n_2,a_2)\rangle\leq \langle (n_1,a_1),\ldots,(n_d,a_d)\rangle$ and $\langle (n_1,a_1),\ldots,(n_d,a_d)\rangle$ surjects onto $A$.	
\end{proof}

We note that the proof of Lemma \ref{lemma: case N1 finite} shows that $d(N\rtimes A)=\max\{2,d(A)\}$. To tackle the general $N=N_1^r$ case, we need three more preparatory lemmas.

%
%\begin{lemma}
%	Let $L$ be a finite dimensional semisimple Lie algebra. Then the set 
%	$$\{(x,y)\in L\times L|\ k\langle x,y\rangle=L\}$$ is non-empty and Zariski-open in $L\times L$.
%\end{lemma}
%\begin{proof}
%	?? Use \cite[Prop. 1.1.3]{Bois:GeneratorsOfSimpleLieAlgebrasInArbitraryCharacteristic} and \cite[Theorem 6]{Kuranishi:OnEverywhereDenseImbeddingOfFreeGroupsInLieGroups}.
%\end{proof}

\begin{lemma} \label{lemma: subdirect product of simple}
	Let $H=H^o$ be a subdirect product of $r\geq 1$ copies of $N_1$, where $N_1$ is a simple algebraic group. Then $H\simeq N_1^s$ for some $s\geq 1$.
\end{lemma}	
\begin{proof}
	First note that $H$ is semisimple. Indeed, if $R=R(H)$ is the radical of $H$ and $\pi_i\colon H\twoheadrightarrow N_1$ is the projection onto the $i$-th factor ($i=1,\ldots,r$), then $\pi_i(R)$ is a connected normal closed solvable subgroup of $N_1$ and therefore trivial. Thus $R$ must be trivial.
	
	Similarly, if $Z$ is the center of $H$, then $\pi_i(Z)$ is a normal closed abelian subgroup of $N_1$ and therefore trivial. Hence $H$ is a centerless semisimple algebraic group and therefore $H=H_1\times\ldots\times H_s$ is a direct product of simple algebraic groups $H_1,\ldots,H_s$.
	
	It suffices to show that $H_j\simeq N_1$ for $j=1,\ldots,s$. As $H_j$ is normal in $H$, also $\pi_i(H_j)$ is normal in $N_1$. So $\pi_i(H_j)=1$ or $\pi_i(H_j)=N_1$. As $\pi_i(H_j)=1$ for $i=1,\ldots,r$ is impossible, there exists an $i$ (depending on $j$) such that $\pi_{i}(H_j)=N_1$. In other words, the composition $H_j\to H\xrightarrow{\pi_i} N_1$ is surjective. As $H_j$ is simple, this composition is an isomorphism.
\end{proof}

\begin{lemma} \label{lemma: Zariski closure of conjugacy class}
	Let $G$ be an algebraic group and $g\in G$. If the Zariski closure of the conjugacy class of $g$ contains $1$, then $g$ is unipotent.
\end{lemma}
\begin{proof}
	We may embed $G$ into some $\Gl_n$. Conjugate elements in $\Gl_n$ have the same characteristic polynomial and the coefficients of the characteristic polynomial $\chi_h(t)$ of $h\in\Gl_n$ are polynomials in the matrix entries of $h$. It follows that $\chi_g(t)=\chi_1(t)=(t-1)^n$. So $1$ is the only eigenvalue of $g$ and therefore $g$ is unipotent.	
%	Let $C[G]\subseteq k[G]$ denote the algebra of class-functions, i.e., the set of all regular functions on $G$ that are constant on all conjugacy classes of $G$. Let $g_s\in G$ denote the semisimple part of $g$. According to \cite[Corollary 1, Section 3.4]{Steinberg:ConjugacyClassesInAlgebraicGroups}, we have $f(g)=f(g_s)$ for all $f\in C[G]$.
%	On the other hand, by assumption, $f(g)=f(1)$ for all $f\in C[G]$. Thus $f(g_s)=f(1)$ for all $f\in C[G]$. By \cite[Corollary 3, Section 3.4]{Steinberg:ConjugacyClassesInAlgebraicGroups}, two semisimple elements of $G$ that yield the same value for all class functions are conjugate. So $g_s$ and $1$ are conjugate. Therefore $g_s=1$ and $g$ is unipotent.
\end{proof}

\begin{lemma} \label{lemma: finitely many conjugacy classes}
	Let $N_1$ be a simple algebraic group, $r,s\geq 1$ integers and $b\in N_1^s$. Then there exist finitely many conjugacy classes $C_1,\ldots,C_m$ of $N_1^r$ such that $\psi(b)\in C_1\cup\ldots\cup C_m$ for all morphisms $\psi\colon N_1^s\to N_1^r$ of algebraic groups. 
\end{lemma}
\begin{proof}
	First note that a normal closed subgroup $N$ of $N_1^s$ is a product of some of the factors. (This is not hard to see: If $(n_1,\ldots,n_s)\in N_1^s$ with $n_1\neq 1$ say, then there exists an $n'\in N_1$ with $n_1n'\neq n'n_1$ and so
	$$(n',1,\ldots,1)(n_1,\ldots,n_s)(n',1,\ldots,1)^{-1}(n_1,\ldots,n_s)^{-1}=(n'n_1n'^{-1}n_1^{-1},1,\ldots,1)$$ is a non-trivial element of $N\cap(N_1\times \{1\}^{s-1})$. Thus $N_1\times\{1\}^{s-1}\leq N$.)
	
	Let $\psi\colon N_1^s\to N_1$ be a non-trivial morphism of algebraic groups. As the kernel is a product of some of the factors, for dimension reasons, it must be a product of all but one factor. It follows that $\psi$ is of the form $\psi(n_1,\ldots,n_s)=\psi'(n_i)$ for a unique $i\in\{1,\ldots,s\}$ and a non-trivial morphism $\psi'\colon N_1\to N_1$. As $N_1$ is simple, the kernel of $\psi'$ is trivial and for dimension reasons $\psi'$ must be an isomorphism. In summary, we see that every morphism $\psi\colon N_1^s\to N_1^r$ is of the form
	$\psi(n_1,\ldots,n_s)=(\psi_1(n_{\sigma(1)}),\ldots,\psi_r(n_{\sigma(r)}))$ for a map $\sigma\colon\{1,\ldots,r\}\to\{1,\ldots,s\}$ and $\psi_1,\ldots,\psi_r\colon N_1\to N_1$ morphisms that are either trivial or automorphisms.
	
	For $n\in N_1$, the set $D(n)=\{\psi(n)|\ \psi\colon N_1\to N_1 \text{ a morphism}\}$ is a finite union of conjugacy classes of $N_1$, because the outer automorphism group of $N_1$ is finite (\cite[Section 27.4]{Humphreys:LinearAlgebraicGroups}).
	
	Write $b=(b_1,\ldots,b_s)\in N_1^s$. The above discussion shows that $\psi(b)\in (D(b_1)\cup\ldots\cup D(b_s))^r$ for all morphisms $\psi\colon N_1^s\to N_1^r$. The set $(D(b_1)\cup\ldots\cup D(b_s))^r$ is an $r$-fold product of a finite union of conjugacy classes of $N_1$. It therefore also is a finite union of conjugacy classes of $N_1^r$.	
\end{proof}

The following proposition establishes the case $N=N_1^r$.

\begin{prop} \label{prop: Gaschuetz for N1 kernel}
	Let $\f\colon G\twoheadrightarrow H$ be a quotient map of algebraic groups with almost-minimal kernel $N=N_1^r$, where $N_1$ a simple algebraic group. Then $\f$ is Gasch\"{u}tz.
\end{prop}
\begin{proof}
	As a first step, we show that we can assume that the exact sequence $1\to G^o\to G\to G/G^o\to 1$ is split. By \cite[Lemma 5.11]{BorelSerre:TheoremsDeFinitudeEnCohomologieGaloisienne}, there exists a finite closed subgroup $A$ of $G$ that meets every connected component of $G$. Among all such $A$'s, let us choose one such that $|A|$ is minimal. Then the morphism $G^o\rtimes A\to G$ is a quotient map with finite kernel. Moreover, $A$ acts on $N$, as $N$ is normal in $G$, and we have an induced action of $A$ on $G^o/N$ and the morphism $(G^o/N)\rtimes A\to H$ is also a quotient map with finite kernel. In fact,
	\begin{equation} \label{eq: comm diagram for BorelSerre}
		\xymatrix{
			G^o\rtimes A \ar@{->>}[r]  \ar@{->>}[d] & (G^o/N)\rtimes A  \ar@{->>}[d] \\
			G \ar@{->>}[r] & H	
		}
	\end{equation}
	is a commutative diagram of quotient maps. Assuming that $G^o\rtimes A \twoheadrightarrow (G^o/N)\rtimes A$ is Gasch\"{u}tz, we will show that $\f\colon G\twoheadrightarrow H$ is Gasch\"{u}tz. Let $h_1,\ldots,h_d$ be generators of $H$ such that $d(G)\leq d$.
	% Let $h'_1,\ldots,h_d'$ be lifts of $h_1,\ldots,h_d$ to $(G^o/N)\rtimes A$. We claim that $\langle h_1',\ldots,h_d'\rangle=(G^o/N)\rtimes A$. As $\langle h_1',\ldots,h_d'\rangle\to H$ is a quotient map, the dimension of $\langle h_1',\ldots,h_d'\rangle$ is that of $H$, and so $\langle h_1',\ldots,h_d'\rangle$ contains the identity component $G^o/N$ of $(G^o/N)\rtimes A$. Therefore, $\langle h_1',\ldots,h_d'\rangle$ is of the form $\langle h_1',\ldots,h_d'\rangle=(G^o/N)\rtimes A'$ for some subgroup $A'$ of $A$. 
	
	%We claim that $A'$ meets every connected component of $G$. Let $G_1$ be a connected component of $G$. Then $\f(G_1)$ is a connected component of $H$. Because $\langle h_1',\ldots,h_d'\rangle=(G^o/N)\rtimes A'\to H$ is a quotient map, there exists an $a'$ in $A'$ such that $\f(a')$ lies in $\f(G_1)$. Thus there exists a $g_1$ in $G_1$ with $\f(a')=\f(g_1)$, i.e., $a'g_1^{-1}$ lie in $N\leq G^o$. Hence $a'$ and $g_1$ lie in the same connected component, namely $G_1$. By the minimality of $|A|$, we have $A'=A$ so that $\langle h_1',\ldots,h_d'\rangle=(G^o/N)\rtimes A$ as claimed. 
	
	We first show that $d(G^o\rtimes A)\leq d$. Because $d(G)\leq d$, the algebraic group $G$ can be generated by $d$ elements. Lifting these generators to elements $g'_1,\ldots,g_d'$ of $G^o\rtimes A$, we see that $\langle g_1',\ldots,g_d'\rangle$ contains $G^o$ for dimension reasons. Thus  $\langle g_1',\ldots,g_d'\rangle=G^o\rtimes A'$ for some subgroup $A'$ of $A$. As $\langle g_1',\ldots,g_d'\rangle=G^o\rtimes A'\to G$ is a quotient map, the group $A'$ must meet every connected component of $G$ and so $A'=A$ by minimality of $|A|$. Therefore $\langle g_1',\ldots,g_d'\rangle=G^o\rtimes A$ and $d(G^o\rtimes A)\leq d$.
	
	The right vertical map in (\ref{eq: comm diagram for BorelSerre}) is Gasch\"{u}tz by Lemma \ref{lemma: finite kernel is Gaschuetz} and the top horizontal map in (\ref{eq: comm diagram for BorelSerre}) is Gasch\"{u}tz by assumption. Thus the composition $G^o\rtimes A\twoheadrightarrow H$ is Gasch\"{u}tz (Lemma \ref{lemma: composition of Gaschuetz is Gaschuetz}). As $d(G^o\rtimes A)\leq d$, this implies that $h_1,\ldots,h_d$ lift to generators $g_1',\ldots,g_d'$ of $G^o\rtimes A$. The images $g_1,\ldots,g_d$ of $g_1',\ldots,g_d'$ in $G$, then are generators of $G$ and lifts of $h_1,\ldots,h_d$ by the commutativity of (\ref{eq: comm diagram for BorelSerre}).	
%	Assuming that $G^o\rtimes A \twoheadrightarrow (G^o/N)\rtimes A$ is Gasch\"{u}tz, we can thus lift the generators $h_1',\ldots,h_d'$ of $(G^o/N)\rtimes A$ to generators $g_1',\ldots,g_d'$ of $G^o\rtimes A$. The images $g_1,\ldots,g_d$ in $G$ of $g_1',\ldots,g_d'$ under $G^o\rtimes A\twoheadrightarrow G$ then generate $G$. Moreover, $\f(g_i)=h_i$ for $i=1,\ldots,d$ by the commutativity of~(\ref{eq: comm diagram for BorelSerre}). 
		Thus $\f$ is Gasch\"{u}tz. 
		
		Furthermore, the kernel of $G^o\rtimes A \twoheadrightarrow (G^o/N)\rtimes A$ is $N$ and $N$ is an almost-minimal normal closed subgroup of $G^o\rtimes A$. In summary, we see that we can assume, without loss of generality, that $G=G^o\rtimes A$ for some finite algebraic group $A$. Then $H$ can be identified with $H=G/N=(G^o/N)\rtimes A=H^o\rtimes A$.

	\medskip
	
	The automorphism group $\Aut(N_1^r)$ of $N_1^r$ is naturally an algebraic group with identity component $\Aut(N_1^r)^o=N_1^r$ acting on $N_1^r$ via conjugation (see \cite[Expose XXIV]{SGA3.3} or \S 5, in particular \S 5.3. (c), in \cite{BorelSerre:TheoremsDeFinitudeEnCohomologieGaloisienne}) and the map $G\to \Aut(N_1^r)$ that maps $g\in G$ to the automorphism of $N_1^r$ given by conjugation with $g$, is a morphism of algebraic groups. We thus have an induced morphism $G^o\to \Aut(N_1^r)^o=N_1^r$. So, for every $g\in G^o$, there exists an $n\in N=N_1^r$ such that conjugation with $g$ and conjugation with $n$ define the same automorphism of $N_1^r$, i.e., $n^{-1}g$ centralizes $N_1^r$. Let $C=C_{G^o}(N)$ be the centralizer of $N$ in $G^o$. We have just seen that $G^o=NC$. Because the center of $N$ is trivial, $N\cap C=1$ so that in fact $G^o=N\times C$. Moreover, $A$ acts on $N$ and on $C$. Indeed, as $H^o$ can be identified with $G^o/N=C$, we have $G=(N\times H^o)\rtimes A=N\rtimes H$.
	
	\medskip
	
	Let us now fix generators $h_1,\ldots,h_d$ of $H$ with $d(G)\leq d$. For $n_1,\ldots,n_d\in N$ we set $H(n_1,\ldots,n_d)=\langle (n_1,h_1),\ldots,(n_d,h_d)\rangle\leq N\rtimes H=G$. We have to show that there exist $n_1,\ldots,n_d\in N$ with $H(n_1,\ldots,n_d)=G$. Suppose, for a contradiction, that this is not the case, i.e., $H(n_1,\ldots,n_d)$ is a proper subgroup of $G$ for all $n_1,\ldots,n_d\in N$.
	
	Write $h_i=(b_i,a_i)\in H^o\rtimes A=H$ for $i=1,\ldots,d$. Then $A=\langle a_1,\ldots,a_d\rangle$ and $d(N\rtimes A)\leq d$ because $N\rtimes A$ is a quotient of $G$.
	% Since a closed subgroup of $N$ that is normal in $N\rtimes A$ is also normal in $G$, we see that $N$ is almost-minimal in $N\rtimes A$.
	 So by Lemma \ref{lemma: case N1 finite}, the set 
	$$P=\{(n_1,\ldots,n_d)\in N^d|\ \langle (n_1,a_1),\ldots,(n_d,a_d)\rangle=N\rtimes A\}$$ 
	is Zariski dense in $N^d$. Let $(n_1,\ldots,n_d)\in P$. The group $H(n_1,\ldots,n_d)\cap N$ is normalized by the elements $(n_1,h_1),\ldots,(n_d,h_d)$ of $H(n_1,\ldots,n_d)$ and by $b_1,\ldots,b_d\in H^o$ (because $H^o$ commutes with $N$). As $(n_i,h_i)=(n_i,b_i,a_i)\in (N\times H^o)\rtimes A=N\rtimes H=G$ for $i=1,\ldots,d$, we see that $(n_1,1,a_1),\ldots,(n_d,1,a_d)\in (N\times H^o)\rtimes A=G$ normalize $H(n_1,\ldots,n_d)\cap N$. Because $\langle (n_1,a_1),\ldots,(n_d,a_d)\rangle=N\rtimes A$, we find that moreover $N\rtimes A$ normalizes $H(n_1,\ldots,n_d)\cap N$. But as $H^o$ normalizes $H(n_1,\ldots,n_d)\cap N$, this implies that $H(n_1,\ldots,n_d)\cap N$ is normal in $G$. Because $N$ is an almost-minimal normal closed subgroup of $G$, we have $H(n_1,\ldots,n_d)\cap N=N$ or $H(n_1,\ldots,n_d)\cap N$ is finite. The first case leads to $N\leq H(n_1,\ldots,n_d)$ and the contradiction $G=H(n_1\ldots,n_d)N=H(n_1,\ldots,n_d)$. Therefore $H(n_1,\ldots,n_d)\cap N$ must be finite. In fact, because $N_1$ is simple, we have $H(n_1,\ldots,n_d)\cap N=1$. This means that $\f\colon N\rtimes H\twoheadrightarrow H$ restricts to an isomorphism between $H(n_1,\ldots,n_d)$ and $H$. Therefore $H(n_1,\ldots,n_d)=\{(\psi_{n_1,\ldots,n_d}(h),h)|\ h\in H\}$ for a crossed homomorphism $\psi_{n_1,\ldots,n_d}\colon H\to N$. Because $H^o$ acts trivially on $N$, the restriction $\psi_{n_1,\ldots,n_d}'\colon H^o\to N$ of $\psi_{n_1,\ldots,n_d}$ to $H^o$ is a morphism of algebraic groups.
	%We claim that $\psi_{n_1,\ldots,n_d}'$ is surjective. 
	The morphism $H(n_1,\ldots,n_d)\to (N\times  H^o)\rtimes A\to N\rtimes A$ sends $(n_i,h_i)$ to $(n_i,a_i)$. It therefore must be surjective, because  $\langle(n_1,a_1),\ldots,(n_d,a_d)\rangle=N\rtimes A$. Thus also the restriction $H(n_1,\ldots,n_d)^o\to (N\rtimes A)^o=N$ is surjective. This shows that $\psi_{n_1,\ldots,n_d}'$ is surjective.
	
	Set $K=\bigcap_{p\in P}\ker(\psi_p')$. Then $K$ is a closed normal  subgroup of $H^o$. Moreover, for $p\in P$, $b\in\ker(\psi'_p)$ and $a\in A$ we have
	$$\psi'_p(aba^{-1})=\psi_p(a)a(\psi_p(ba^{-1}))=\psi_p(a)a(b(\psi_p(a^{-1})))=\psi_p(a)a(\psi_p(a^{-1}))=\psi_p(aa^{-1})=1 $$
	because $b\in H^o$ acts trivially on $N$. This shows that $\ker(\psi'_p)$ is normalized by $A$. So $K$ is normalized by $A$ and therefore normal in $G$. Note that $\ker(\psi'_{n_1,\ldots,n_d})\leq H(n_1,\ldots,n_d)$ and therefore $K$ is contained in every $H(n_1,\ldots,n_d)$ with $(n_1,\ldots,n_d)\in P$.

	We will show that $H^o/K\simeq N_1^s$ for some $s\geq 1$. As $H^o$ is a Noetherian topological space, there exist $p_1,\ldots,p_m\in P$ such that $K=\ker(\psi'_{p_1})\cap\ldots\cap\ker(\psi'_{p_m})$. We thus get a closed  embedding of $H^o/K$ into $H^o/\ker(\psi'_{p_1})\times \ldots \times H^o/\ker(\psi'_{p_m})=N^m=N_1^{rm}$, such that $H^o/K$ is a subdirect product of $m$ copies of $N$. But then $H^o/K$ is also a subdirect product of $rm$ copies of $N_1$. So it follows from Lemma \ref{lemma: subdirect product of simple} that $H^o/K\simeq N_1^s$ for some $s\geq 1$ as claimed. 
	
	Now consider the induced morphism $\f'\colon G/K\twoheadrightarrow H/K$. We have $d(G/K)\leq d$, $$G/K=((N\times H^o)\rtimes A)/K=(N\times H^o/K)\rtimes A, \quad H/K=(H^o\rtimes A)/K=H^o/K\rtimes A$$ and $N$ is an almost-minimal normal closed subgroup of $G/K$. For every $(n_1,\ldots,n_d)\in P$ the images $\overline{h_1},\ldots,\overline{h_d}$ of $h_1,\ldots,h_d$ in $H/K$ lift to elements $(n_1,\overline{h_1}),\ldots, (n_d,\overline{h_d})$ of $G/K=N\rtimes H/K$, generating a closed subgroup $\langle (n_1,\overline{h_1}),\ldots, (n_d,\overline{h_d})\rangle$ of $G/K$. In fact, $\langle (n_1,\overline{h_1}),\ldots, (n_d,\overline{h_d})\rangle=H(n_1,\ldots,n_d)K/K$.
	If $\langle (n_1,\overline{h_1}),\ldots, (n_d,\overline{h_d})\rangle$ equals $G/K$, then $H(n_1,\ldots,n_d)K=G$, yielding $H(n_1,\ldots,n_d)=G$ because $K\leq H(n_1,\ldots,n_d)$. But, by assumption, $H(n_1,\ldots,n_d)$ is a proper closed subgroup of $G$ and therefore $\langle (n_1,\overline{h_1}),\ldots, (n_d,\overline{h_d})\rangle$ must be a proper closed subgroup of $G/K$ for all $(n_1,\ldots,n_d)\in P$. In summary, we see that $\f'\colon G/K\twoheadrightarrow H/K$ along with the generators $\overline{h_1},\ldots,\overline{h_d}$ satisfies similar properties as $\f\colon G\twoheadrightarrow H$ along with the generators $h_1,\ldots,h_d$.

	More precisely, we can assume that $H^o=N_1^s$, $G=(N\times H^o)\rtimes A$, $H=H^o\rtimes A$, $N$ is almost minimal in $G$ and $\f\colon G\twoheadrightarrow H$ restricts to an isomorphism from $H(n_1,\ldots,n_d)=\langle(n_1,h_1),\ldots,(n_d,h_d)\rangle\leq N\rtimes H=G$ to $H$ for all $(n_1,\ldots,n_d)\in P$.
	As before, we write $h_i=(b_i,a_i)\in H^o\rtimes A$ for $i=1,\ldots,d$ and $\psi_{n_1,\ldots,n_d}'\colon H^o\to N$ retains its meaning, as above, for all $(n_1,\ldots,n_d)\in P$.

	Fix $(n_1,\ldots,n_d)\in P$ and let $T$ be the (abstract) group generated by $(n_1,a_1),\ldots,(n_d,a_d)\in N\rtimes A$. As $T\cap N$ is Zariski dense in $N$ (Lemma \ref{lemma: identity component}) and the set of unipotent elements of $N$ is a Zariski closed proper subset of $N$, there exists an $n\in T\cap N$ that is not unipotent. In other words, there exists a word $w\in F_{x_1,\ldots,x_d}$ such that $w((n_1,a_1),\ldots,(n_d,a_d))=(n,1)\in N\rtimes A$ with $n\in N$ not unipotent. Then
	$$w((n_1,h_1),\ldots,(n_d,h_d))=w((n_1,b_1,a_1),\ldots,(n_d,b_d,a_d))=(n,b,1)\in (N\times H^o)\rtimes A$$
	for some $b\in H^o$. For $n_1',\ldots,n_d'\in N$ we have
	$$w((n'_1,b_1,a_1),\ldots,(n'_d,b_d,a_d))=(\psi(n_1',\ldots,n'_d),b,1)$$ with $\psi\colon N^d\to N$ a morphism of varieties. Moreover, if $(n'_1,\ldots,n_d')\in P$, then $\psi(n'_1,\ldots,n'_d)=\psi'_{n'_1,\ldots,n'_d}(b)$. 
	Recall that $\psi'_{n'_1,\ldots,n_d'}\colon H^o=N_1^s\to N=N_1^r$ is a morphism of algebraic groups for all $(n'_1,\ldots,n_d')\in P$. By Lemma~\ref{lemma: finitely many conjugacy classes}, there exist conjugacy classes $C_1,\ldots, C_m$ of $N=N_1^r$ such that $\psi'_{n'_1,\ldots,n_d'}(b)\in C_1\cup\ldots \cup C_m$ for all $(n_1',\ldots,n_d')\in P$. So $\psi(P)\subseteq C_1\cup\ldots\cup C_m$. As $P$ is Zariski dense in $N^d$, we see that $\psi(N^d)$ is contained in the Zariski closure of $C_1\cup\ldots\cup C_m$. In fact, since $N^d$ is an irreducible topological space, we see that $\psi(N^d)$ must be contained in the Zariski closure of a single conjugacy class $C$ of $N$. By construction, we have $\psi(n_1,\ldots,n_d)=n$ and $\psi(1,\ldots,1)=1$.
	So, by Lemma~\ref{lemma: Zariski closure of conjugacy class}, the conjugacy class $C$ is unipotent. As the set of unipotent elements of $N$ is Zariski closed, we see that $\psi(N^d)$ consists of unipotent elements. This contradicts our assumption that $n=\psi(n_1,\ldots,n_d)\in\psi(N^d)$ is not unipotent. 	
\end{proof}

\subsection{The general case} In this subsection we prove the Gasch\"{u}tz Lemma for algebraic groups by combining the three special cases from the previous three subsections. To this end, we need some preparatory results.

\begin{lemma} \label{lemma: almost-minimal kernel is Gaschuetz}
	Every quotient map of algebraic groups with almost-minimal kernel is Gasch\"{u}tz.
\end{lemma}	
\begin{proof}
	Let $\f\colon G\twoheadrightarrow H$ be a quotient map with almost-minimal kernel $N$. According to Lemma \ref{lemma: almost-minimal subgroups} there are three possible cases. In the first case ($N=\Ga^r$), it follows from Proposition \ref{prop: Gaschuetz for Ga kernel} that $\f$ is Gasch\"{u}tz. In the second case ($N=\Gm^r$), it follows from Proposition \ref{prop: Gaschuetz for Gm kernel} that $\f$ is Gasch\"{u}tz. In the third case, there exists a finite closed subgroup $Z$ of $N$, normal in $G$, such that $N/Z=N_1^r$, where $N_1$ is a simple algebraic group. Thus $\f$ factors as $G\twoheadrightarrow G/Z\twoheadrightarrow G/N=H$. The morphism  $G\twoheadrightarrow G/Z$ is Gasch\"{u}tz by Lemma \ref{lemma: finite kernel is Gaschuetz}. The morphism $G/Z\twoheadrightarrow G/N$ has almost-minimal kernel $N/Z=N_1^r$. It is therefore Gasch\"{u}tz by Proposition \ref{prop: Gaschuetz for N1 kernel}. It thus follows from Lemma \ref{lemma: composition of Gaschuetz is Gaschuetz} that $\f$ is Gasch\"{u}tz. 
\end{proof}

\begin{lemma} \label{lemma: Gaschuetz for connected kernel}
	Every quotient map of algebraic groups with connected kernel is Gasch\"{u}tz. 
\end{lemma}	
\begin{proof}
		Let $\f\colon G\twoheadrightarrow H$ be a quotient map with connected kernel $N$. There exists a descending chain $N=N_1\supsetneqq N_2\supsetneqq \ldots\supsetneqq N_r=1$ of connected normal closed subgroups $N_i$ of $G$ such that $N_i/N_{i+1}$ is almost-minimal in $G/N_{i+1}$. To see this, first choose among all proper closed subgroups of $N$ that are normal in $G$, one of maximal dimension, say $N_2'$. As the identity component $N_2$ of $N_2'$ is a characteristic subgroup of $N_2'$, it follows that $N_2$ is normal in $G$. Moreover, $N/N_2$ is almost-minimal in $G/N_2$ because a closed subgroup of $N/N_2$, normal in $G/N_2$, corresponds to a closed subgroup $N'$ of $N$, normal in $G$ and containing $N_2$. As the dimension of $N'$ cannot exceed the dimension of $N_2$, we see that $N'/N_2$ is finite. Thus $N_2$ has all the required properties. To construct $N_3$, we choose among all proper closed subgroups of $N_2$ that are normal in $G$, one of maximal dimension, say $N_3'$, and we let $N_3$ be the identity component of $N_3'$. As above, we see that $N_3$ is normal in $G$ and that $N_2/N_3$ is almost-minimal in $G/N_3$. Continuing like this, we construct the desired chain.

	Since $\f$ factors as
	$$G=G/N_r\twoheadrightarrow G/N_{r-1}\twoheadrightarrow \ldots\twoheadrightarrow G/N_2\twoheadrightarrow G/N_1=H$$ and every morphism $G/N_{i+1}\twoheadrightarrow G/N_{i}$ has an almost-minimal kernel $N_{i}/N_{i+1}$, it follows from Lemmas~\ref{lemma: almost-minimal kernel is Gaschuetz} and \ref{lemma: composition of Gaschuetz is Gaschuetz} that $\f$ is Gasch\"{u}tz.
\end{proof}	

With all of the above results in hand, it is now a simple matter to prove the Gasch\"{u}tz lemma for algebraic groups.

\begin{theo} \label{theo: Geschuetz}
	Every quotient map $\f\colon G\to H$ of algebraic groups is Gasch\"{u}tz, i.e., if $h_1,\ldots,h_d$ generate $H$ and $d(G)\leq d$, then there exist generators $g_1,\ldots,g_d$ of $G$ such that $\f(g_i)=h_i$ for $i=1,\ldots,d$.	
\end{theo}	
\begin{proof}
	Set $N=\ker(\f)$. As $N^o$ is a characteristic closed subgroup of $N$, it is normal in $G$. We thus have a factorization $G\twoheadrightarrow G/N^o\twoheadrightarrow G/N=H$ of $\f$. The kernel $N^o$ of $G\twoheadrightarrow G/N^o$ is connected, thus $G\twoheadrightarrow G/N^o$ is Gasch\"{u}tz by Lemma \ref{lemma: Gaschuetz for connected kernel}. On the other hand, the kernel $N/N^o$ of $G/N^o\twoheadrightarrow G/N$ is finite, and therefore $G/N^o\twoheadrightarrow G/N$ is Gasch\"{u}tz (Lemma \ref{lemma: finite kernel is Gaschuetz}). Hence it follows from Lemma \ref{lemma: composition of Gaschuetz is Gaschuetz} that $\f$ is Gasch\"{u}tz.	
\end{proof}

\section{Finite generation under base change}
\label{sec: Finite generation under base change}

 In this section we show that the minimal number of generators of an algebraic group is invariant under base change, i.e., if $G$ is an algebraic group over $k$, then $d(G_{k'})=d(G)$ for any algebraically closed field $k'$ containing $k$ (Proposition \ref{prop: d under base change}).
This is needed to show that the differential Galois group of a regular singular differential equation with $d+1$ singularities can be generated by $d$ elements (Lemma \ref{lemma: Galois of reg sing generated}), which is well-known for $k=\mathbb{C}$ but not otherwise. The proof of Proposition \ref{prop: d under base change} may seem surprisingly difficult. It relies on an old result of I. Schur and a fairly recent result of R. Feng.

\medskip

%Let $F_d$ be the abstract free group on $x_1,\ldots,x_d$. For an algebraic group $G$ and elements $g_1,\ldots,g_d\in G(k)$, we write $w(g)$ for the image of $w$ under the morphism of $F_d\to G(k)$ determined by sending $x_i$ to $g_i$ for $i=1,\ldots,d$. 

The Burnside problem asks if a finitely generated group in which every element has finite order is actually finite. In general the answer is no but for subgroups of algebraic groups the answer is yes.

\begin{theo}[I. Schur] \label{theo: Schur}
	Let $\Omega$ be an (abstract) subgroup of an algebraic group $G$ such that $\Omega$ is finitely generated and every element of $\Omega$ has finite order. Then $\Omega$ is finite.
\end{theo}
\begin{proof}
	There is no harm in assuming $G=\Gl_n$ and $k=\mathbb{C}$. In this form the result is stated in \cite[Theorem 36.2]{CurtisReiner:RepresentationTheoryOfFiniteGroupsAndAssociateAlgebras}.
\end{proof}

\begin{lemma} \label{lemma: finite step generation}
	Let $G$ be an algebraic group with generators $g_1,\ldots,g_d$ and let $\Omega$ be the (abstract) subgroup of $G$ generated by $g_1,\ldots,g_d$. Then there exist $\omega_1,\ldots,\omega_n\in \Omega$ such that $G=\langle \omega_{1}\rangle\ldots \langle \omega_{n}\rangle$. %$G=G_{i_1}\ldots G_{i_m}$, where $G_i=\langle \omega_i(g)\rangle$.
	% $$G^\circ=\langle w_m(g)\langle w_{m-1}(g)\langle\ldots \langle w_2(g)\langle w_1(g)\rangle^\circ\rangle\rangle\ldots\rangle\rangle.$$
\end{lemma}
\begin{proof}
	For a finite subset $\Psi$ of $\Omega$ let $H_\Psi$ denote the (abstract) subgroup of $G$ generated by all $\langle \omega\rangle^o$ for $\omega\in \Psi$. Then $H_\Psi$ is a closed connected subgroup of $G$ and $H_\Psi=\langle \omega_1\rangle^o\ldots \langle \omega_n\rangle^o$ for some (not necessarily distinct) $\omega_1,\ldots,\omega_n\in \Psi$ (\cite[Cor. 2.2.7]{Springer:LinearAlgebraicGroups}). Among all finite subsets of $\Omega$ choose $\Psi$ such that $\dim(H_\Psi)$ is maximal. As $H_\Psi\subseteq H_{\Psi'}$ for $\Psi\subseteq\Psi'$, it then follows that $H_\Psi$ is the union of all $H_{\Psi'}$, where $\Psi'$ ranges over all finite subsets of $\Omega$. In other words, $H_\Psi$ is the subgroup of $G$ generated by all $\langle\omega\rangle^o$ for $\omega\in\Omega$. We will eventually see that $H_\Psi=G^o$.
	
	We claim that $H_\Psi$ is normal in $G$. For $\omega$, $\omega'\in \Omega$ we have $\omega\langle \omega'\rangle^o\omega^{-1}=(\omega\langle\omega'\rangle\omega^{-1})^o=\langle\omega\omega'\omega^{-1}\rangle^o$. This shows that $\Omega$ normalizes $H_\Psi$. Since the normalizer of $H_\Psi$ is closed, it follows that $G$ normalizes $H_\Psi$, i.e., $H_\Psi$ is normal in $G$.
	
	Note that $\omega$ generates the finite group $\langle \omega\rangle/\langle \omega\rangle^o$. Thus for every $\omega\in \Omega$, there exists an $m\geq 1$ such that $\omega^m\in \langle \omega\rangle^o\leq H_\Psi$. Therefore every element of $\Omega/(\Omega\cap H_\Psi)$ has finite order. Now $\Omega/(\Omega\cap H_\Psi)$ embeds into $G/H_\psi$ and so it follows from Theorem \ref{theo: Schur} that $\Omega/(\Omega\cap H_\Psi)$ is finite. Because $\Omega/(\Omega\cap H_\Psi)$ is Zariski dense in $G/H_\Psi$ it follows that $G/H_\Psi$ is finite. As $H_\Psi$ is connected, this shows that $H_\Psi=G^o$.
	
	Recall (from the first paragraph) that $H_\psi=\langle \omega_1\rangle^o\ldots \langle \omega_n\rangle^o$ for appropriate $\omega_1,\ldots,\omega_n\in \Psi$. Let $G_1,\ldots,G_m$ be the connected components of $G$. As $\Omega$ is Zariski dense in $G$, the group $\Omega$ intersects every $G_i$. So we can choose $\omega_i'\in G_i\cap \Omega$ for $i=1,\ldots,m$.
	
	Set $H=\langle \omega_1\rangle\ldots \langle \omega_n\rangle\langle \omega'_1\rangle\ldots \langle \omega'_m\rangle$. We claim that $H=G$. We have $G^o=H_\Psi=\langle \omega_1\rangle^o\ldots \langle \omega_n\rangle^o\subseteq \langle \omega_1\rangle\ldots \langle \omega_n\rangle$ and therefore $G_i=G^o\omega_i'\subseteq H$ for $i=1,\ldots,m$. Thus $G=H$ as desired.
\end{proof}

Let $R$ be a finitely generated $k$-algebra that is an integral domain and set $X=\spec(R)$. Any $\xi\in X(k)=\Hom_k(R,k)$ defines specialization morphism $\xi\colon \Gm(R)\to \Gm(k)$. Let $\Omega$ be a finitely generated subgroup of $\Gm(R)=R^\times$. We set
$$W_X(\Gm,\Omega)=\{\xi\in X(k)|\ \xi\colon \Gm(R)\to \Gm(k) \text{ is injective on } \Omega\}.$$ 
More generally, in the above definition, $\Gm$ can be replaced by a (not necessarily affine) commutative group scheme of finite type of $R$. The corresponding subsets of $X(k)$ were first introduced by Hrushovski in \cite{Hrushovski:ComputingTheGaloisGroupOfaLinearDifferentialEequation}. They play an important role in the study of linear differential and difference equations under specialization. The main result concerning these sets that we shall need is:

\begin{theo}[R. Feng] \label{theo: Feng}
	The set $W_X(\Gm,\Omega)$ is Zariski dense in $X(k)$.
\end{theo}
\begin{proof}
	This is a special case of \cite[Theorem 2.16]{Feng:DifferenceGaloisGroupsUnderSpecialization}.
\end{proof}	

Roughly speaking, the following lemma shows that we can specialize an element in an algebraic group without changing the nature of the closed subgroup it generates. 

\begin{lemma} \label{lemma: specialize generator}
	Let $k\subseteq k'$ be an inclusion of algebraically closed fields and let $G$ be an algebraic group over $k$. For $g\in G(k')$ there exists a finitely generated $k$-subalgebra $R$ of $k'$, a closed subgroup scheme $H$ of $G_R$ and a finitely generated subgroup $\Omega$ of $\Gm(R)$ such that
	\begin{enumerate}
		\item $g\in H(R)$,
		\item $H_{k'}=\langle g\rangle$ (as closed subgroups of $G_{k'}$),
		\item $H_\xi=\langle \xi(g)\rangle$ (as closed subgroups of $G$) for all $\xi\in W_X(\Gm,\Omega)$, (where $X=\spec(R)$). 
	\end{enumerate}	   
\end{lemma}
\begin{proof}
	We fix an embedding of $G$ into some $\Gl_n$. Let $h\in \Gl_n(k')$ be such that $J=hgh^{-1}\in\Gl_n(k')$ is in Jordan canonical form. So $J$ is a block diagonal matrix with Jordan blocks $J_1,\ldots,J_m$ on the diagonal, where $$J_i=\begin{bmatrix}\lambda_i & 1 & & \\
		& \lambda_i & \ddots &  \\
		& & \ddots & 1 \\
		& & & \lambda_i
	\end{bmatrix}$$ with $\lambda_i\in k'$ for $i=1,\ldots,m$. The (multiplicative) Jordan decomposition of $J$ is then of the form $J=J_sJ_u$, where $J_s$ is the diagonal of $J$ and $J_u$ is a block diagonal matrix with blocks $I_1,\ldots,I_m$, where
	$$I_i=\begin{bmatrix}1 & \lambda_i^{-1} & & \\
		& 1 & \ddots &  \\
		& & \ddots & \lambda_i^{-1} \\
		& & & 1
	\end{bmatrix}$$
	for $i=1,\ldots,m$.	Let $\mu_1,\ldots,\mu_n\in k'$ be the diagonal entries of $J_s$ and set $M=\{(\alpha_1,\ldots,\alpha_n)\in \mathbb{Z}^n|\ \mu_1^{\alpha_1}\ldots\mu_n^{\alpha_n}=1\}$. Then $\langle J_s\rangle$ consists of all diagonal matrices with diagonal entries $d_1,\ldots,d_n\in k'$ such that  $d_1^{\alpha_1}\ldots d_n^{\alpha_n}=1$ for all  $(\alpha_1,\ldots,\alpha_n)\in M$. 
	
	Recall (see e.g. \cite[14.28]{Milne:AlgebraicGroupsTheTheoryOfGroupSchemesOfFiniteTypeOverAField}) that for a $k$-algebra $S$, the assignment $N\mapsto \exp(N)=\sum_{j=0}^\infty \frac{1}{j!}N^j$ defines a bijection between the nilpotent matrices $N\in S^{n\times n}$ and the unipotent matrices $U\in S^{n\times n}$, with inverse $U\mapsto \log(U)=\log(I+(U-I))=\sum_{j=1}^\infty(-1)^{j+1}\frac{(U-I)^j}{j}$, where $I$ is the $n\times n$ identity matrix.
	
	%Write $J_u=I+N$, where $I$ is the $n\times n$ identity matrix. Then $N$ is nilpotent and $\log(J_u-I)=$

	If $J$ is not a diagonal matrix, i.e., $J_u$ is not the identity (equivalently not all Jordan blocks have size one), then $\langle J_u\rangle$ is the image of the closed embedding $\mathbb{G}_{a,k'}\to\Gl_{n,k'}$ that sends $t\in \Ga(k')$ to $\exp(t\log(J_u))$ (cf. \cite[14.29, 14.30]{Milne:AlgebraicGroupsTheTheoryOfGroupSchemesOfFiniteTypeOverAField}).
%	
%	 the block diagonal matrix with blocks $I_1(a),\ldots,I_m(a)$, where 
%	\begin{equation} \label{eq: Ia}	
%		I_i(a)=\begin{bmatrix}1 & \lambda_i^{-1}a & & \\
%			& 1 & \ddots &  \\
%			& & \ddots & \lambda_i^{-1}a \\
%			& & & 1
%		\end{bmatrix}
%	\end{equation}
	%for $i=1,\ldots,m$. 
	Because the Jordan decomposition of $J$ exists within the algebraic group $\langle J\rangle\leq\Gl_n(k')$, we have $\langle J\rangle=\langle J_s\rangle \langle J_u\rangle$.
	
	Let $R$ be the $k$-subalgebra of $k'$ generated by the entries of $h$, $h^{-1}$, $\lambda_1,\ldots,\lambda_m$,  $\lambda_1^{-1},\ldots,\lambda_m^{-1}$.	Let $H'_s$ be the closed subgroup scheme of $\Gl_{n,R}$ consisting of all diagonal matrices whose diagonal entries $d_1,\ldots,d_n$ satisfy $d_1^{\alpha_1}\ldots d_n^{\alpha_n}=1$ for all $(\alpha_1,\ldots,\alpha_n)\in M$. Similarly, if $J$ is not a diagonal matrix, let $H'_u$ denote the closed subgroup scheme of $\Gl_{n,R}$, that is the image of the closed embedding $\mathbb{G}_{a,R}\to \Gl_{n,R}$, that, for an $R$-algebra $S$, sends $t\in S$ to $\exp(t\log(J_u))$.
	
	Set $H'=H'_sH'_u$. As $H'_s$ and $H'_u$ commute and have trivial intersection, this is in fact a direct product. By the above discussion, we have $H'_{k'}=(H'_s)_{k'}(H'_u)_{k'}=\langle J_s\rangle\langle J_u\rangle=\langle J\rangle$. Set $\widetilde{H}=h^{-1}H'h$. Then $\widetilde{H}$ is a closed subgroup scheme of $\Gl_{n,R}$ containing $g=h^{-1}Jh$.
	Moreover, $$\widetilde{H}_{k'}=(h^{-1}H'h)_{k'}=h^{-1}H'_{k'}h=h^{-1}\langle J\rangle h=\langle h^{-1} J h\rangle=\langle g\rangle.$$
	
	Let $H$ be the smallest closed subgroup scheme of $\Gl_{n,R}$ containing $g$. Then $g\in H(R)$ and $H$ is contained in $\widetilde{H}$ and in $G_R$. Clearly, $\langle g\rangle \leq H_{k'}$. As $H_{k'}\leq \widetilde{H}_{k'}=\langle g\rangle$, we see that $H_{k'}=\langle g\rangle$.
	
	Let $\Omega$ be the subgroup of $R^\times$ generated by $\mu_1,\ldots,\mu_n$ and let $\xi\in W_X(\Gm,\Omega)$. As $g\in H(R)$, we have $\xi(g)\in H_\xi(k)$ and so $\langle \xi(g)\rangle\leq H_\xi$. For the reverse inclusion, note that $\xi(J)=\xi(h)\xi(g)\xi(h)^{-1}$ is a Jordan canonical form of $\xi(g)$ and that $\xi(J)_s=\xi(J_s)$ and $\xi(J)_u=\xi(J_u)$. The diagonal entries $\xi(\mu_1),\ldots,\xi(\mu_n)$ of $\xi(J)_s$ satisfy $\xi(\mu_1)^{\alpha_1}\ldots\xi(\mu_n)^{\alpha_n}=1$ for all $(\alpha_1,\ldots,\alpha_n)\in M$. On the other hand, as $\xi\colon \Omega\to k^\times$ is injective, $\xi(\mu_1)^{\alpha_1}\ldots\xi(\mu_n)^{\alpha_n}\neq 1$ for $(\alpha_1,\ldots,\alpha_n)\notin M$. This shows that $\langle \xi(J)_s\rangle=(H'_s)_\xi$. If $J$ is not a diagonal matrix, then $\xi(J)_u=\xi(J_u)$ is not the identity matrix and generates an additive group. In fact, $\langle \xi(J)_u\rangle=(H'_u)_\xi$. So $\langle \xi(J)\rangle=\langle \xi(J)_s\rangle\langle \xi(J)_u\rangle=(H'_s)_\xi(H'_u)_\xi=(H'_sH'_u)_\xi=H'_\xi$. Therefore
	$$\langle\xi(g)\rangle=\langle\xi(h)^{-1}\xi(J)\xi(h)\rangle=\xi(h)^{-1}\langle\xi(J)\rangle\xi(h)=\xi(h)^{-1}H'_\xi\xi(h)=(h^{-1}H'h)_\xi=\widetilde{H}_\xi.$$
	Thus $H_\xi\leq \widetilde{H}_\xi=\langle \xi(g)\rangle$ and $\langle\xi(g)\rangle=H_\xi$ as desired.	
\end{proof}	

We are now prepared to show that the minimal number of generators of an algebraic group is invariant under base change.

\begin{prop} \label{prop: d under base change}
	Let $k\subseteq k'$ be an inclusion of algebraically closed fields (of characteristic zero) and let $G$ be an algebraic group over $k$. Then $d(G_{k'})=d(G)$. 
\end{prop}	
\begin{proof}
	Assume that $g_1,\ldots,g_d\in G(k)$ generate $G$. Then $g_1,\ldots,g_d\in G(k)\subseteq G(k')=G_{k'}(k')$ generate $G_{k'}$ by \cite[Lemma 2.8]{SerafiniWibmer:RegularSingularInverseProblem}. Thus $d(G_{k'})\leq d=d(G)$.
	
	For the reverse inequality, let $g_1,\ldots,g_d\in G(k')$ be generators of $G_{k'}$. We have to show that $G$ can be generated by $d$ elements. The idea for the proof is to find a specialization that specializes $g_1,\ldots,g_d$ to generators of $G$. Let $\Omega$ be the (abstract) subgroup of $G(k')$ generated by $g_1,\ldots,g_d$. %and let $R$ be a finitely generated $k$-subalgebra of $k'$ such that $g_1,\ldots,g_d\in G(R)$. Then also $\Omega\subseteq G(R)$.
	By Lemma \ref{lemma: finite step generation} there exist $\omega_1,\ldots,\omega_n\in \Omega$ such that $G_{k'}=\langle \omega_{1}\rangle\ldots \langle \omega_{n}\rangle$. By Lemma \ref{lemma: specialize generator} there exists, for every $i=1,\ldots,n$, a finitely generated $k$-subalgebra $R_i$ of $k'$, a closed subgroup scheme $H_i$ of $G_{R_i}$ and a finitely generated subgroup $\Omega_i$ of $R_i^\times$ such that $\omega_i\in H_i(R_i)$, $(H_i)_{k'}=\langle\omega_i\rangle$ and $(H_i)_\xi=\langle\xi(\omega_i)\rangle$ for all $\xi\in W_{X_i}(\Gm,\Omega_i)$, where $X_i=\spec(R_i)$. Let $R$ be a finitely generated $k$-subalgebra of $k'$ such that $g_1,\ldots,g_d\in G(R)$ and $R_i\subseteq R$ for $i=1,\ldots,n$. Set $G_i=(H_i)_R$ for $i=1,\ldots,n$. Then $\omega_i\in H_i(R_i)\subseteq H_i(R)=G_i(R)$ and $(G_i)_{k'}=(H_i)_{k'}=\langle\omega_i\rangle$ for $i=1,\ldots,n$. Let $\Omega'$ denote the subgroup of $R^\times$ generated by $\Omega_1,\ldots,\Omega_n$. Then a $\xi\in W_X(\Gm,\Omega')$, where $X=\spec(R)$, restricts to an element of $W_{X_i}(\Gm,\Omega_i)$ for $i=1,\ldots,n$. Therefore $(G_i)_\xi=(H_i)_\xi=\langle \xi(\omega_i)\rangle$ for all $\xi\in W_X(\Gm,\Omega')$ and all $i=1,\ldots,n$.
	
	Consider the multiplication map $\f\colon G_1\times\ldots\times G_n\to G_R$ of schemes over $R$. Then $\f_{k'}$ is the multiplication map $\langle\omega_1\rangle\times\ldots\times \langle \omega_n\rangle\to G_{k'}$, which is surjective. (Note that a morphism of schemes of finite type over $k'$ is surjective on $k'$-points if and only if it is surjective on the underlying topological spaces.) Let $\zeta$ denote the generic point of $X$ so that $\f_\zeta=\f_K$, where $K\subseteq k'$ is the field of fractions of $R$. As $(\f_\zeta)_{k'}=\f_{k'}$ is surjective, it follows that also $\f_\zeta$ is surjective (\cite[Prop. 2.6.1]{Grothendieck:EGAIV2}). The set of all $\xi\in X$ such that $\f_\xi$ is surjective is constructible \cite[Prop. 9.6.1]{Grothendieck:EGAIV3}. Because it contains $\zeta$, it therefore contains a non-empty open subset $U$ of $X$. By Theorem \ref{theo: Feng}, the set $W_X(\Gm,\Omega')\cap U$ is non-empty. Let $\xi\in W_X(\Gm,\Omega')\cap U$. Then $\f_\xi\colon (G_1)_\xi\times\ldots\times (G_n)_\xi\to G$ is surjective. Because $(G_i)_\xi=\langle \xi(\omega_i)\rangle$, this shows that $G=\langle \xi(\omega_{1})\rangle\ldots \langle\xi(\omega_{n})\rangle$. Therefore $G=\langle \xi(g_1),\ldots,\xi(g_d)\rangle$ as desired.
\end{proof}

\section{Embedding problems and free proalgebraic groups}

\label{sec: Embedding problems and free proalgebraic groups}

Recall that throughout $k$ is an algebraically closed field of characteristic zero and that all proalgebraic groups (=affine group schemes) are over $k$ unless the contrary is indicated.

An \emph{embedding problem} for a proalgebraic group $\Gamma$ is a pair $(G\twoheadrightarrow H,\ \Gamma\twoheadrightarrow H)$ of quotient maps of proalgebraic groups.
The embedding problem is \emph{algebraic} if $G$ (and therefore also $H$) is an algebraic group.
 A \emph{solution} is a quotient map $\Gamma\twoheadrightarrow G$ such that
$$
	\xymatrix{
		\Gamma \ar@{..>>}[d] \ar@{->>}[rd] & \\
		G \ar@{->>}[r]& H
	}
$$
commutes. A characterization of free proalgebraic groups on infinitely many generators in terms of embedding problems was obtained in \cite{Wibmer:FreeProalgebraicGroups}. The purpose of this section is to discuss a characterization of free proalgebraic groups on finitely many generators closely connected to the Gasch\"{u}tz Lemma for (pro)algebraic groups.

\subsection{Free proalgebraic groups} We begin by recalling the definition of free proalgebraic groups. For more background see \cite{Wibmer:FreeProalgebraicGroups}. Let $X$ be a set and $\Gamma$ a proalgebraic group. A map $X\to \Gamma(k)$ \emph{converges to $1$} if for every quotient map $\Gamma\twoheadrightarrow G$ with $G$ algebraic (i.e. of finite type), the composition $X\to \Gamma(k)\to G(k)$ maps all but finitely many elements of $X$ to $1$.

The \emph{free proalgebraic group} $\Gamma_X$ on $X$ (over $k$) comes equipped with a map $X\to\Gamma_X(k)$ converging to $1$ satisfying the following universal property: If $G$ is a proalgebraic group and $X\to G(k)$ a map converging to $1$, then there exists a unique morphism $\Gamma_X\to G$ of proalgebraic groups such that
$$
\xymatrix{
	X \ar[r] \ar[rd] & \Gamma_X(k) \ar@{..>}[d] \\
	& G(k)	
}
$$   
commutes. The map $X\to \Gamma_X(k)$ is injective, in fact, even the map $F_X\to\Gamma_X(k)$ is injective (\cite[Lemma 1.8]{Wibmer:SubgroupsOfFreeProalgebraicGroupsAndMatzatsconjecture}), so that we can identify $X$ with a subset of $\Gamma_X(k)$. With this identification, $X$ generates $\Gamma_X$ (\cite[Theorem 2.17]{Wibmer:FreeProalgebraicGroups}).

We are mainly interested in the case of finitely many generators, i.e., $X=\{x_1,\ldots,x_d\}$ is a finite set. In this case, the ``convergence to $1$'' condition is always satisfied and can therefore be omitted. So for $X$ finite, $\Gamma_X$ is the proalgebraic completion of the (abstract) free group $F_X$ on $X$. %It can be constructed as the fundamental group of the neutral tannakian category of all finite dimensional $k$-linear representations of $F_X$.
Of course only the cardinality $d$ of $X$ is relevant and so we will usually write $\Gamma_d$ for the free proalgebraic group on a set of cardinality $d$, and refer to it as the \emph{free proalgebraic group on $d$ generators}. A morphism $\f\colon\Gamma_d\to G$ of proalgebraic groups is determined by the images $g_1,\ldots,g_d\in G(k)$ of the free generators $x_1,\ldots,x_d$ of $\Gamma_d$. As $\f(\Gamma_d)=\f(\langle X\rangle)=\langle \f(X)\rangle$ (\cite[Lemma 2.16]{Wibmer:FreeProalgebraicGroups}), it is a quotient map if and only if $g_1,\ldots,g_d$ generate $G$.

Recall that for an abelian group $M$, we denote with $D(M)$ the diagonalizable proalgebraic group with character group $M$, i.e., $D(M)(R)=\Hom(M,R^\times)$ for any $k$-algebra $R$.

\begin{ex} \label{ex: Gamma1}
The free proalgebraic group $\Gamma_1$ on one generator is of the form $\Gamma_1=\Ga\times D(k^\times)$ with generator $x=(1,\id)\in\Gamma_1(k)=k\times\Hom(k^\times,k^\times)$. See \cite[Appendix A]{Sauloy:2002galoistheoryfuchsianqdifference} or \cite[Cor. 16.26]{Sauloy:DifferentialGaloisTheoryThroughRiemannHilbertCorrespondence}.
\end{ex}
As a word of warning, note that the above example shows that the formation of the free proalgebraic group does not commute with base change.

The Gasch\"{u}tz Lemma for algebraic groups (Theorem \ref{theo: Geschuetz}) translates into the following embedding property for the free proalgebraic group $\Gamma_d$ on $d$ generators. In fact, this embedding property is what motivated our interest in the Gasch\"{u}tz Lemma for algebraic groups in the first place.

\begin{prop} \label{prop: embedding property for Gammad}
	Every algebraic embedding problem $(G\twoheadrightarrow H,\ \Gamma_d\twoheadrightarrow H)$ for $\Gamma_d$ with $d(G)\leq d$ has a solution.
\end{prop}	 
\begin{proof}
Let $x_1,\ldots,x_d$ be the free generators of $\Gamma_d$. Then the images $h_1,\ldots,h_d$ of $x_1,\ldots,x_d$ in $H$ generate $H$. As $d(G)\leq d$, Theorem \ref{theo: Geschuetz} shows that $h_1,\ldots,h_d$ can be lifted to generators $g_1,\ldots,g_d$ of $G$. By construction, the morphism $\Gamma_d\to G$ determined by sending $x_i$ to $g_i$ for $i=1,\ldots,d$ is a solution of the  embedding problem $(G\twoheadrightarrow H,\ \Gamma_d\twoheadrightarrow H)$. \end{proof}

It is natural to wonder if Proposition \ref{prop: embedding property for Gammad} characterizes $\Gamma_d$. In other words, if $\Gamma$ is a proalgebraic group with $d(\Gamma)\leq d$ such that every algebraic embedding problem $(G\twoheadrightarrow H,\ \Gamma\twoheadrightarrow H)$ for $\Gamma$ with $d(G)\leq d$ has a solution, is $\Gamma$ isomorphic to $\Gamma_d$? We will see below that the answer is yes if $k$ is countable (Proposition~ \ref{prop: characterization for countable}) but no if $k$ is uncountable (Example \ref{ex: counterexample for algebraic characterization}). Note that this is a significant difference to what happens for free profinite groups. If $\Gamma$ is a profinite group that can be generated (as a topological group) by $d$ elements and every finite embedding problem $(G\to H,\ \Gamma\to H)$ such that $G$ can be generated by $d$ elements has a solution, then $\Gamma$ is isomorphic to the free profinite group on $d$ generators. In fact, if $\Gamma$ is a profinite group that can be generated by $d$ elements and such that every finite group that can be generated by $d$ elements is a quotient of $\Gamma$, then $\Gamma$ is isomorphic to the free profinite group on $d$ generators (\cite[Lemma 17.7.1]{FriedJarden:FieldArithmetic}).

As the property of Proposition \ref{prop: embedding property for Gammad} is not strong enough to characterize free proalgebraic groups for uncountable $k$, we suggest below a (conjectural) property that works for all (algebraically closed, characteristic zero) base fields. To this end we recall the definition of the \emph{rank} of a proalgebraic group $G$ from \cite{Wibmer:FreeProalgebraicGroups}. For brevity, we call a normal closed subgroup $N$ of $G$ \emph{coalgebraic} if the quotient $G/N$ is algebraic (i.e., of finite type over $k$). A set $\mathcal{N}$ of coalgebraic subgroups of $G$ is a \emph{neighborhood basis at $1$} for $G$ if for every coalgebraic subgroup $N'$ of $G$ there exists an $N\in\mathcal{N}$ such that $N\subseteq N'$. One then has $G=\varprojlim_{N\in\mathcal{N}} G/N$.

The rank $\rank(G)$ of a non-trivial proalgebraic group $G$ is defined as the smallest cardinality of a neighborhood basis at $1$. The rank of the trivial algebraic group is defined to be zero. In particular, $\rank(G)<\infty$ if and only if $G$ is algebraic, and in this case $\rank(G)=1$ unless $G$ is trivial. For other characterizations of the rank see \cite[Prop. 3.1]{Wibmer:FreeProalgebraicGroups}. For example, if $G$ is not algebraic, it can also be defined as the dimension of $k[G]$ as a $k$-vector space.

\begin{ex} \label{ex: rank of Gammad}
	For $d\geq 1$, we have $\rank(\Gamma_d)=|k|$. For $d=1$ this follows from Example \ref{ex: Gamma1}. As $\Gamma_1$ is a quotient of $\Gamma_d$ and the rank can only go down when taking quotients (\cite[Lemma 3.5]{Wibmer:FreeProalgebraicGroups}), we see that $\rank(\Gamma_d)\geq\rank(\Gamma_1)=|k|$. On the other hand, $\rank(\Gamma_d)\leq |k|d=|k|$ by \cite[Cor.~3.11]{Wibmer:FreeProalgebraicGroups}. So $\rank(\Gamma_d)=|k|$ as claimed.
\end{ex}

\begin{lemma} \label{lemma: rank of finitely gen G}
	If $G$ is a proalgebraic group with $d(G)$ finite, then $\rank(G)\leq |k|$.
\end{lemma}	
\begin{proof}
If $d=d(G)$ is finite, then $G$ is a quotient of $\Gamma_d$. Thus $\rank(G)\leq \rank(\Gamma_d)=|k|$ by Example~\ref{ex: rank of Gammad}.
\end{proof}	

As $d(G)\leq d$ is a necessary condition for an embedding problem $(G\twoheadrightarrow H,\ \Gamma_d\twoheadrightarrow H)$ for $\Gamma_d$ to have a solution, the obvious way to strengthen the property from Proposition \ref{prop: embedding property for Gammad} is to allow $G$ and $H$ to be proalgebraic groups and not just algebraic groups. However, Example \ref{ex: need rank assumption } below, shows that not every embedding problem $(G\twoheadrightarrow H,\ \Gamma_d\twoheadrightarrow H)$ for $\Gamma_d$ with $d(G)\leq d$ has a solution. We therefore make the following definition:

\begin{defi} \label{defi: d-saturated}
	Let $d\geq 1$ be an integer. A proalgebraic group $\Gamma$ is \emph{$d$-saturated} if $d(\Gamma)\leq d$ and every embedding problem $(G\twoheadrightarrow H,\ \Gamma\twoheadrightarrow H)$ for $\Gamma$ with $d(G)\leq d$ and $\rank(H)<|k|$ has a solution.
\end{defi}

Clearly, the property of Definition \ref{defi: d-saturated} is stronger than the property of Proposition \ref{prop: embedding property for Gammad}. Our next goal is to show that it is strong enough in the sense that any two $d$-saturated proalgebraic groups are isomorphic.

\begin{lemma} \label{lemma: rank of d-saturated}
	If $\Gamma$ is $d$-saturated, then $\rank(\Gamma)=|k|$.
\end{lemma}	
\begin{proof} It follows from Lemma \ref{lemma: rank of finitely gen G} that $\rank(\Gamma)\leq|k|$. As the embedding problem $(\Gamma_d\twoheadrightarrow 1,\ \Gamma\twoheadrightarrow 1)$ for $\Gamma$ has a solution, we see that $\Gamma_d$ is a quotient of $\Gamma$. Thus $\rank(\Gamma)\geq\rank(\Gamma_d)=|k|$ by Example~\ref{ex: rank of Gammad}.
\end{proof}
For a fixed $d\geq 1$, $d$-saturated proalgebraic groups are unique up to a non-canonical isomorphism. 

\begin{prop} \label{prop: uniqueness of d-saturated}
	Let $\Gamma$ and $\Gamma'$ be $d$-saturated proalgebraic groups. Then $\Gamma$ and $\Gamma'$ are isomorphic.
\end{prop}
\begin{proof}
	The proof is similar to the proof of \cite[Theorem~3.38]{Wibmer:FreeProalgebraicGroups}. For the sake of completeness, we include the details.
	
	Note that $\rank(\Gamma)=\rank(\Gamma')=|k|$ by Lemma \ref{lemma: rank of d-saturated}. 
	We can thus choose neighborhood bases $\mathcal{N}$ and $\mathcal{N}'$ at $1$ for $\Gamma$ and $\Gamma'$ 
	such that $|\mathcal{N}'|=|\mathcal{N}|=|k|$. Well-ordering these sets, we may write $\mathcal{N}=\{N_\mu|\ \mu<|k|\}$ and  $\mathcal{N}'=\{N'_\mu|\ \mu<|k|\}$. We may assume, without loss of generality, that $N_0=\Gamma$ and $N_0'=\Gamma'$. By transfinite induction, we will construct descending chains $\{M_\mu|\ \mu<|k|\}$ and $\{M'_\mu|\ \mu<|k|\}$ of normal closed subgroups of $\Gamma$ and $\Gamma'$ respectively, together with isomorphisms $\f_\mu\colon \Gamma/M_\mu\to \Gamma'/M'_\mu$ (for all $\mu<|k|$) such that
	\begin{enumerate}
		\item $$\xymatrix{
			\Gamma/M_\mu \ar^-{\f_\mu}[r] \ar@{->>}[d] & \Gamma'/M'_\mu  \ar@{->>}[d] \\
			\Gamma/M_\lambda \ar^-{\f_\lambda}[r] & \Gamma'/M'_\lambda 
		}$$
		commutes for all $\lambda<\mu$,
		\item $\rank(\Gamma/M_\mu)\leq \mu$ and
		\item $M_\mu\leq N_\lambda$ and $M'_\mu\leq N'_\lambda$ for $\lambda<\mu$.
	\end{enumerate}	
	
	We set $M_0=N_0=\Gamma$, $M'_0=N'_0=\Gamma'$ and we let $\f_0\colon \Gamma/M_0=1\to\Gamma'/M_0'=1$ be the trivial morphism. Then conditions (i), (ii) and (iii) above are satisfied for $\mu=0$. 
	
	Assume that for a fixed $\mu<|k|$ we have already constructed $M_\lambda$, $M'_\lambda$ and $\f_\lambda$ for all $\lambda<\mu$ such that (i), (ii) and (iii) hold for all $\lambda<\mu$. We have to construct appropriate $M_\mu, M'_\mu$ and $\f_\mu$.
	
	\medskip
	
	We first treat the case that $\mu$ is a successor ordinal, say $\mu=\lambda+1$. Define $P_\mu=M_\lambda\cap N_\lambda$ and consider the embedding problem $(\Gamma/P_\mu\twoheadrightarrow \Gamma/M_\lambda,\ \Gamma'\twoheadrightarrow \Gamma'/M'_\lambda\xrightarrow{\f_\lambda^{-1}}\Gamma/M_\lambda)$ for $\Gamma'$. As $d(\Gamma/P_\mu)\leq d(\Gamma)\leq d$, $\rank(\Gamma/M_\lambda)\leq \lambda<\mu<|k|$ and $\Gamma'$ is $d$-saturated, it has a solution $\f'\colon\Gamma'\twoheadrightarrow\Gamma/P_\mu$. For $P_\mu'=\ker(\f')\leq M'_\lambda$ we have an induced isomorphism $\overline{\f'}\colon \Gamma'/P'_\mu\to \Gamma/P_\mu$ such that
	\begin{equation}\label{eq: square2 commutes}
		\xymatrix{
			\Gamma/P_\mu \ar@{->>}[d]  & \Gamma'/P'_\mu \ar_-{\overline{\f'}}[l]  \ar@{->>}[d]\\
			\Gamma/M_\lambda & \Gamma'/M'_\lambda  \ar_-{\f_\lambda^{-1}}[l]
		}
	\end{equation}
	commutes. Define $M'_\mu=P'_\mu\cap N'_\lambda\leq M'_\lambda$ and consider the embedding problem 
	\begin{equation} \label{eq: embedding problem}
		(\Gamma'/M'_\mu\twoheadrightarrow \Gamma'/P'_\mu,\ \Gamma\twoheadrightarrow \Gamma/P_\mu\xrightarrow{\overline{\f'}^{-1}}\Gamma'/P'_\mu)
	\end{equation}
	for $\Gamma$. As the kernel $M_\lambda/P_\mu=M_\lambda/(M_\lambda\cap N_\lambda)=M_\lambda N_\lambda/N_\lambda\hookrightarrow \Gamma/N_\lambda$ of $\Gamma/P_\mu\twoheadrightarrow\Gamma/M_\lambda$ is algebraic, it follows from \cite[Lemma 3.6]{Wibmer:FreeProalgebraicGroups} that $\rank(\Gamma/P_\mu)=\rank(\Gamma/M_\lambda)\leq \lambda<\mu<|k|$. So $\rank(\Gamma'/P'_\mu)=\rank(\Gamma/P_\mu)<|k|$. As $d(\Gamma'/M'_\mu)\leq d(\Gamma')\leq d$ and $\Gamma$ is $d$-saturated, this implies that (\ref{eq: embedding problem}) has a solution $\f\colon \Gamma\twoheadrightarrow \Gamma'/M'_\mu$. Set $M_\mu=\ker(\f)\leq P_\mu=M_\lambda\cap N_\lambda$ and let $\f_\mu=\overline{\f}\colon \Gamma/M_\mu\to \Gamma'/M'_\mu$ be the isomorphism induced by $\f$. Then
	\begin{equation} \label{eq: square1 commutes}
		\xymatrix{
			\Gamma/M_\mu \ar^-{\f_\mu}[r] \ar@{->>}[d] & \Gamma'/M'_\mu \ar@{->>}[d] \\
			\Gamma/P_\mu \ar^-{\overline{\f'}^{-1}}[r] & \Gamma'/P'_\mu 	
		}
	\end{equation}
	commutes. Combining (\ref{eq: square1 commutes}) and (\ref{eq: square2 commutes}), we see that
	$$\xymatrix{
		\Gamma/M_\mu \ar^-{\f_\mu}[r] \ar@{->>}[d] & \Gamma'/M'_\mu  \ar@{->>}[d] \\
		\Gamma/M_\lambda \ar^-{\f_\lambda}[r] & \Gamma'/M'_\lambda 
	}$$
	commutes. As above, we see that the kernel of $\Gamma'/M'_\mu\twoheadrightarrow\Gamma'/P'_\mu$ is algebraic and therefore
	$\rank(\Gamma/M_\mu)=\rank(\Gamma'/M'_\mu)=\rank(\Gamma'/P'_\mu)\leq \mu$.

	This completes the construction of the $M_\mu$'s, $M'_\mu$'s and $\f_\mu$'s in the case that $\mu$ is a successor ordinal.
	
	In case $\mu$ is a limit ordinal, we set $M_\mu=\bigcap_{\lambda<\mu}M_\lambda$ and $M'_\mu=\bigcap_{\lambda<\mu}M'_\lambda$.
	Then $\Gamma/M_\mu=\varprojlim_{\lambda<\mu}\Gamma/M_\lambda$,  $\Gamma'/M'_\mu=\varprojlim_{\lambda<\mu}\Gamma'/M'_\lambda$ and we can define $\f_\mu\colon \Gamma/M_\mu\to \Gamma'/M'_\mu$ as the isomorphism induced by the $\f_\lambda$ for $\lambda<\mu$. Then (i) and (iii) are satisfied. Moreover, as $\rank(\Gamma/M_\lambda)\leq \mu$ for all $\lambda <\mu$ and $|\{\lambda|\ \lambda<\mu\}|\leq \mu$, it follows from \cite[Lemma 3.7]{Wibmer:FreeProalgebraicGroups} that $\rank(\Gamma/M_\mu)\leq \mu$. So also (ii) is satisfied. This completes the construction of the $M_\mu$'s, $M'_\mu$'s and $\f_\mu$'s.

	It follows from (iii) that $\bigcap_{\mu<|k|}M_\mu\subseteq \bigcap_{\mu<|k|}N_\mu=1$ and similarly for the corresponding subgroups of $\Gamma'$. Thus $\Gamma=\varprojlim_{\mu<|k|}\Gamma/M_\mu$,  $\Gamma'=\varprojlim_{\mu<|k|}\Gamma'/M'_\mu$ and the $\f_\mu$ ($\mu<|k|$) induce an isomorphism between $\Gamma$ and $\Gamma'$.
\end{proof}

The crux with Definition \ref{defi: d-saturated} is that we don't know if there exists a $d$-saturated proalgebraic group. However, there is a natural candidate:

\begin{conj} \label{conj: Gamma_d is d-saturated}
	The free proalgebraic group on $d$ generators is $d$-saturated.
\end{conj}	

Just like Proposition \ref{prop: embedding property for Gammad} corresponds to the Gasch\"{u}tz Lemma for algebraic groups, Conjecture~\ref{conj: Gamma_d is d-saturated} corresponds to a fairly general Gasch\"{u}tz Lemma for proalgebraic groups.

\begin{conj}[Gasch\"{u}tz Lemma for proalgebraic groups] \label{conj: Gaschuetz for proalgebraic}
	Let $\f\colon G\twoheadrightarrow H$ be a quotient map of proalgebraic groups such that $\rank(H)<|k|$. Assume that $h_1,\ldots,h_d\in H(k)$ generate $H$ and that $d(G)\leq d$. Then there exist $g_1,\ldots,g_d\in G(k)$ generating $G$ such that $\f(g_i)=h_i$ for $i=1,\ldots,d$.
\end{conj}

\begin{lemma} \label{lemma: conjectures equivalent}
	For a fixed $d\geq 1$, Conjecture \ref{conj: Gamma_d is d-saturated} and Conjecture \ref{conj: Gaschuetz for proalgebraic} are equivalent.	
\end{lemma}
\begin{proof}
	First assume that Conjecture \ref{conj: Gamma_d is d-saturated} holds. Let $\f\colon G\to H$ be a quotient map of proalgebraic groups such that $\rank(H)<|k|$, $d(G)\leq d$ and assume that $h_1,\ldots,h_d\in H(k)$ generate $H$. Let $\Gamma_d\twoheadrightarrow H$ be the quotient map defined by sending the free generators $x_1,\ldots,x_d$ of $\Gamma_d$ to $h_1,\ldots,h_d$. According to Conjecture  \ref{conj: Gamma_d is d-saturated}, the embedding problem $(G\twoheadrightarrow H,\ \Gamma_d\twoheadrightarrow H)$ for $\Gamma_d$ has a solution $\psi\colon \Gamma_d\twoheadrightarrow G$. Set $g_i=\psi(x_i)$ for $i=1,\ldots,d$. Then $\f(g_i)=h_i$ and $g_1,\ldots,g_d$ generate $G$.
	
	Conversely, assume that Conjecture \ref{conj: Gaschuetz for proalgebraic} holds. Let $(G\twoheadrightarrow H,\ \Gamma_d\twoheadrightarrow H)$ be an embedding problem for $\Gamma_d$ such that $\rank(H)<|k|$ and $d(G)\leq d$. Let $h_1,\ldots,h_d\in H(k)$ be the images of the free generators $x_1,\ldots,x_d$ of $\Gamma_d$. By Conjecture \ref{conj: Gaschuetz for proalgebraic} there exist $g_1,\ldots,g_d\in G(k)$ such that $G\twoheadrightarrow H$ maps $g_i$ to $h_i$ and $g_1,\ldots,g_d$ generate $G$. The morphism $\Gamma_d\twoheadrightarrow G$ determined by sending $x_i$ to $g_i$ is then a solution of the embedding problem $(G\twoheadrightarrow H,\ \Gamma_d\twoheadrightarrow H)$.
\end{proof}	

Below we will prove the above conjectures for $k$ countable (Theorem \ref{theo: Conjectures for k countable}), based on the Gasch\"{u}tz Lemma for algebraic groups (Theorem~\ref{theo: Geschuetz}). We also show that the conjectures hold for $d=1$ (Theorem~\ref{theo: conjecture for d=1}). Example \ref{ex: need rank assumption } below shows that the Gasch\"{u}tz Lemma for proalgebraic groups does not hold without the assumption on the rank of $H$.

\begin{cor} \label{cor: characterization of Gammad}
	If Conjecture \ref{conj: Gaschuetz for proalgebraic} holds, a proalgebraic group is $d$-saturated if and only if it is isomorphic to a free proalgebraic group on $d$ generators.
\end{cor}
\begin{proof}
	This follows from Proposition \ref{prop: uniqueness of d-saturated} and Lemma \ref{lemma: conjectures equivalent}.
\end{proof}

The point of the above corollary is that in some situations it might be difficult to directly show that a given proalgebraic group is isomorphic to the free proalgebraic group on $d$ generators because this would require a more or less explicit description of the free generators. The characterization of the free proalgebraic group on $d$-generators in terms of embedding problems yields an alternative approach that does not require the construction of generators.

%This is not true. Need some condition to ensure that $\Gamma$ to $H$ is not an isomorphism, maybe $\rank(H)<\rank(\Gamma)$? Get counterexample for $d=1$. The passage from algebraic to proalgebraic embedding problems may be non-trivial.

\begin{theo} \label{theo: Conjectures for k countable}
		Conjectures \ref{conj: Gamma_d is d-saturated} and \ref{conj: Gaschuetz for proalgebraic} hold for $k$ countable.	
\end{theo}	
\begin{proof}
	By Lemma \ref{lemma: conjectures equivalent}, it suffices to prove Conjecture \ref{conj: Gaschuetz for proalgebraic}. So let $G\twoheadrightarrow H$ be a quotient map of proalgebraic groups with $\rank(H)<|k|$. Furthermore, let $h_1,\ldots,h_d$ be generators of $H$ and $d(G)\leq d$. As $k$ is countable, the condition $\rank(H)<|k|$ means that $H$ is an algebraic group. Because $\rank(G)\leq |k|$ (Lemma \ref{lemma: rank of finitely gen G}) is countable, we see that $G$ can be written as the projective limit of a sequence of algebraic groups of the form
	\begin{equation} \label{eq: quotient maps}
	H\twoheadleftarrow G_1\twoheadleftarrow G_2\twoheadleftarrow\ldots .
	\end{equation}
	As the $G_n$'s are quotients of $G$, we have $d(G_n)\leq d$. By the Gasch\"{u}tz Lemma for algebraic groups (Theorem \ref{theo: Geschuetz}) we can lift $h_1,\ldots,h_d$ to generators $g^1_1,\ldots,g_d^1$ of $G_1$. Applying Theorem \ref{theo: Geschuetz} again, we can lift $g^1_1,\ldots,g_d^1$ to generators $g^2_1,\ldots,g_d^2$ of $G_2$. Continuing this way, we find for every $n\geq 1$ generators $g_1^n,\ldots,g_d^n$ of $G_n$ compatible with the quotient maps in (\ref{eq: quotient maps}). For $i=1,\ldots,d$, the sequence $(g_i^n)_{n\geq 1}$ defines an element $g_i$ of $\varprojlim_n G_n(k)=G(k)$. As the images of $g_1,\ldots,g_d$ in $G_n$ generate $G_n$ for all $n\geq 1$, we see that $g_1,\ldots,g_d$ generate $G$. They are therefore the sought for lifts of $h_1,\ldots,h_d$ to a generating set of $G$.
\end{proof}	

For a countable base field $k$, there is a particularly nice characterization of free proalgebraic groups that only requires algebraic embedding problems.

\begin{prop} \label{prop: characterization for countable}
For $d\geq 1$, if $k$ is countable, the following conditions on a proalgebraic group $\Gamma$ are equivalent:
	\begin{enumerate}
		\item The proalgebraic group $\Gamma$ is isomorphic to the free proalgebraic group on $d$ generators.
		\item We have $d(\Gamma)\leq d$ and every algebraic embedding problem $(G\twoheadrightarrow H,\ \Gamma\twoheadrightarrow H)$ for $\Gamma$ with $d(G)\leq d$ has a solution.
		\end{enumerate}
\end{prop}	
\begin{proof}
	The implication (i)$\Rightarrow$(ii) holds by Proposition \ref{prop: embedding property for Gammad}. Conversely, assume that $\Gamma$ satisfies (ii). We claim that $\Gamma$ is $d$-saturated. So let $(G\twoheadrightarrow H,\ \Gamma\twoheadrightarrow H)$ be an embedding problem for $\Gamma$ with $d(G)\leq d$ and $\rank(H)< |k|$. As $k$ is countable, $\rank(H)<|k|$ means that $H$ is algebraic. Furthermore, as $G$ is finitely generated, Lemma \ref{lemma: rank of finitely gen G} shows that the rank of $G$ is countable. As in the proof of Theorem \ref{theo: Conjectures for k countable}, we can thus write $G$ as the projective limit of a sequence of algebraic groups of the form
$$
		H\twoheadleftarrow G_1\twoheadleftarrow G_2\twoheadleftarrow\ldots .
	$$
	As the $G_n$'s are quotients of $G$, we have $d(G_n )\leq d$ for all $n\geq 1$. By assumption, the embedding problem $(G_1\twoheadrightarrow H,\ \Gamma\twoheadrightarrow H)$ has a solution $\f_1\colon \Gamma\twoheadrightarrow G_1$. Similarly, the embedding problem $(G_2\twoheadrightarrow G_1,\ \f_1\colon \Gamma\twoheadrightarrow G_1)$ has a solution $\f_2$. Continuing in this fashion	
	$$
	\xymatrix{
&		&	\Gamma \ar@{->>}[rd] \ar@{->>}^-{\f_1}[d] \ar@{->>}^-{\f_2}[ld] \ar@{->>}^-{\f_3}[lld] & \\
\cdots G_3 \ar@{->>}[r] & G_2 \ar@{->>}[r] &	G_1 \ar@{->>}[r] & H \\	
	}
	$$
	we obtain a sequence $(\f_n)_{n\geq 1}$ of quotient maps $\f_n\colon \Gamma \twoheadrightarrow G_n$ that defines a solution $\f\colon \Gamma \twoheadrightarrow \varprojlim_n G_n =G$ of our embedding problem  $(G\twoheadrightarrow H,\ \Gamma\twoheadrightarrow H)$. Thus $\Gamma$ is $d$-saturated.
	
	As $\Gamma_d$ is $d$-saturated by Theorem \ref{theo: Conjectures for k countable}, it follows from Proposition \ref{prop: uniqueness of d-saturated} that $\Gamma$ and $\Gamma_d$ are isomorphic.
\end{proof}	

Example \ref{ex: counterexample for algebraic characterization} below shows that Proposition \ref{prop: characterization for countable} fails if $k$ is not countable.

\subsection{One-generated proalgebraic groups} This subsection is not needed for the proof of our main results. Its purpose is to put Conjectures \ref{conj: Gamma_d is d-saturated} and \ref{conj: Gaschuetz for proalgebraic} in perspective. We prove the conjectures for $d=1$ and show that the conjectures do not hold without the assumption $\rank(H)<|k|$. We first obtain a full description of the proalgebraic groups that can be generated by one element. Recall that $D(M)$ denotes the diagonalizable proalgebraic group with character group $M$.

\begin{lemma} \label{lemma: one generation of diagonalizable}
	Let $M$ be an abelian group. Then $g\in D(M)(k)=\Hom(M,k^\times)$ generates $D(M)$ (as a proalgebraic group) if and only if $g\colon M\to k^\times $ is injective.
\end{lemma}
\begin{proof}
	First assume $\langle g\rangle=D(M)$. Let $H$ be the closed subgroup of $D(M)$ defined by 
	$$H(R)=\{h\in\Hom(M,R^\times)|\ h(m)=1 \ \forall \ m\in \ker(g)\}$$ for any $k$-algebra $R$. Then $g\in H(k)$ and so $H=D(M)$. In particular, for $h\colon M\to k[M]^\times$ the inclusion map, we have $h\in H(k[M])$. This shows that $\ker(g)=1$, i.e., $g$ is injective.
	
	Conversely, assume that $g$ is injective.
	Let $H=\langle g\rangle$ be the closed subgroup of $D(M)$ generated by $g$. Then $H$ is diagonalizable and there exists a subgroup $N$ of $M$ such that
	$$H(R)=\{h\in\Hom(M,R^\times)|\ h(n)=1 \ \forall \ n\in N\}$$ for all $k$-algebras $R$ (\cite[Chapter IV, \S 1, Cor. 1.7 and Prop. 1.3]{DemazureGabriel:GroupesAlgebriques}). As $g\in H(k)$, we have $g(n)=1$ for all $n\in N$. Because $g$ is injective, this shows that $N=1$. Thus $\langle g\rangle=H=D(M)$ as desired.
\end{proof}	

\begin{cor} \label{cor: one generation of diagonalizable}
	The proalgebraic group $D(M)$ can be generated by one element if and only if there exists an embedding $M\hookrightarrow k^\times$ of abelian groups. \qed
\end{cor}

Having understood the diagonalizable case, we now add a unipotent part.

\begin{lemma} \label{lemma: one generation}
	Let $g_1\colon M\to k^\times$ be an injective morphism of abelian groups and $g_2\in k^\times$. Then $g=(g_1,g_2)\in(D(M)\times\Ga)(k)$ is a generator of $D(M)\times \Ga$.
\end{lemma}	
\begin{proof}
	Set $H=\langle g\rangle$. As $g_1$ generates $D(M)$ (Lemma \ref{lemma: one generation of diagonalizable}) and $g_2$ generates $\Ga$, we see that the projections $H\to D(M)$ and $H\to \Ga$ are quotient maps. Thus, by Lemma \ref{lemma: Goursat}, there exists a proalgebraic group $G$ and quotient maps $D(M)\twoheadrightarrow G$ and  $\Ga\twoheadrightarrow G$ such that $H=D(M)\times_G \Ga$. But the only non-trivial quotient of $\Ga$ is $\Ga$ and any quotient of a diagonalizable proalgebraic group is diagonalizable. This implies that $G=1$, so that $H=D(M)\times \Ga$ as desired. 
\end{proof}	

The following lemma gives a complete description of all proalgebraic groups that can be generated by one element.

\begin{lemma} \label{lemma: structure of one-generated group}
	A proalgebraic group $G$ can be generated by one element if and only if it is of the form $G=D(M)\times\Ga^e$, where $e\in \{0,1\}$ and $M$ is a subgroup of $k^\times$.
\end{lemma}	
\begin{proof} It follows from Lemmas \ref{lemma: one generation of diagonalizable} and \ref{lemma: one generation} that $D(M)\times\Ga^e$
	can be generated by one element.
	Conversely, if $G$ can be generated by one element, then $G$ is abelian. Any abelian proalgebraic group is of the form $G=G^s\times G^u$, where $G^s=D(M)$ is a diagonalizable proalgebraic group and $G^u$ is a unipotent proalgebraic group (\cite[Chapter IV, \S3, Thm 1.1]{DemazureGabriel:GroupesAlgebriques}). 
	As $D(M)$ can be generated by one element, it follows from Corollary \ref{cor: one generation of diagonalizable} that $M$ embeds into $k^\times$.

	An abelian unipotent algebraic group (over an algebraically closed field of characteristic zero) is isomorphic to $\Ga^n$ for some $n\geq 0$ (\cite[Chapter IV, \S2, Prop 4.2]{DemazureGabriel:GroupesAlgebriques}). As $d(\Ga^n)=n$, we see that any non-trivial algebraic quotient of $G^u$ must be isomorphic to $\Ga$. This implies that $G^u$ is trivial or isomorphic to $\Ga$.
\end{proof}	

We now prove Conjecture \ref{conj: Gaschuetz for proalgebraic} for $d=1$. The proof is based on the explicit description of the one-generated proalgebraic groups and a key property of the abelian group $k^\times$.

\begin{theo} \label{theo: conjecture for d=1}
	Conjecture \ref{conj: Gaschuetz for proalgebraic} (and therefore also Conjecture \ref{conj: Gamma_d is d-saturated}) holds for $d=1$, i.e., if $G$ is a proalgebraic group with $d(G)=1$ and $G\twoheadrightarrow H$ is a quotient map with $\rank(H)<|k|$, then every generator of $H$ lifts to a generator of $G$.
\end{theo}	
\begin{proof}
	%Let $\f\colon G\twoheadrightarrow H$ be a quotient map of proalgebraic groups with $\rank(H)<|k|$ and $d(G)\leq 1$.
	 Assume that $h\in H(k)$ generates $H$. We first assume that $G$ is diagonalizable. Then also $H$ is diagonalizable, say $G=D(M)$ and $H=D(N)$ for abelian groups $M$ and $N$. By Theorem~\ref{theo: Conjectures for k countable} we may assume that $k$ is uncountable. We then have $|N|<|k|$ because $|N|=\rank(D(N))$ if $N$ is not finitely generated (i.e., $D(N)$ is not algebraic).

	The quotient map $ G\twoheadrightarrow H$ corresponds to an injection $N\to M$ and $h\colon N\to k^\times$ is injective (Lemma~\ref{lemma: one generation of diagonalizable}). We see that $h$ lifts to a generator of $G$ if and only if the embedding $h\colon N\to k^\times$ can be extended to an embedding $g\colon M\to k^\times$, i.e., the diagram
	$$
	\xymatrix{
		& k^\times \\
		M  \ar@{^(..>}^-{g}[ru]   & N	 \ar@{_(->}[l] \ar@{^(->}_-{h}[u] 
	}
	$$ 
	can be completed. As $G$ can be generated by one element, we know that $M$ can be embedded into $k^\times$. However, we have to show that $M$ can be embedded into $k^\times$ over $N$, using that $|N|<|k|$.
	
	Note that $k^\times$ is a divisible group. Let $E(N)$ and $E(M)$ denote the divisible hulls of $N$ and $M$ respectively. 
	Then $N\hookrightarrow M$ extends to an embedding $E(N)\hookrightarrow E(M)$. As $k^\times$ is divisible and $M$ embeds into $k^\times$, also $E(M)$ embeds into $k^\times$. Identifying $N$ with a subgroup of $k^\times$, we have
	$E(N)=\{a\in k^\times|\ \exists\ n\geq 1 :\ a^n\in N\}$. So $|E(N)|=\max\{|\mathbb{N}|,|N|\}$ and therefore $|E(N)|<|k|$. An embedding of $E(M)$ over $E(N)$ into $k^\times$ induces an embedding of $M$ over $N$ into $k^\times$. We can thus assume that $M$ and $N$ are divisible. 
	
	Any divisible group is a direct sum of Pr\"{u}fer $p$-groups $\mathbb{Z}(p^\infty)$ and the additive group of $\mathbb{Q}$ (\cite[Chapter 4, Thm. 3.1]{Fuchs:AbelianGroups}). The direct sums involved here are of arbitrary but uniquely determined cardinality. We also know that $k^\times=\mathbb{Q}/\mathbb{Z}\oplus\bigoplus_{|k|}\mathbb{Q}$ (\cite[Chapter 18, Thm.~8.3]{Fuchs:AbelianGroups}). Moreover, $\mathbb{Q}/\mathbb{Z}=\bigoplus_p \mathbb{Z}(p^\infty)$, where the direct sum is taken over all prime numbers. As a divisible subgroup is a direct summand (\cite[Chapter 4, Cor.~2.4]{Fuchs:AbelianGroups}), we can write $M=N\oplus N'$ for some (necessarily divisible) subgroup $N'$ of $M$. It suffices to see that we can embed $N'$ into $k^\times$ such that the intersection with $N\subseteq k^\times$ is trivial.
	 Write $N=\bigoplus_{p}\oplus_{\kappa_p} \mathbb{Z}(p^\infty)\oplus \bigoplus_{\kappa_\mathbb{Q}}\mathbb{Q}$ and $N'=\bigoplus_{p}\oplus_{\kappa'_p} \mathbb{Z}(p^\infty)\oplus \bigoplus_{\kappa'_\mathbb{Q}}\mathbb{Q}$. Then $\kappa_p\in\{0,1\}$ and $\kappa_\mathbb{Q}<|k|$. As $N$ is a direct summand of $k^\times$, we can write
	 $k^\times=N\oplus \widehat{N}$ with $\widehat{N}=\bigoplus_{p}\oplus_{\widehat{\kappa}_p} \mathbb{Z}(p^\infty)\oplus \bigoplus_{\widehat{\kappa}_\mathbb{Q}}\mathbb{Q}$. It suffices to see that $N'$ can be embedded into $\widehat{N}$. Because $M=N\oplus N'$ embeds into $k^\times$, we must have $\kappa_p+\kappa'_p\leq 1$ and $\kappa'_\mathbb{Q}\leq |k|$. As $\kappa_p+\widehat{\kappa}_p=1$ and $\widehat{\kappa}_\mathbb{Q}=|k|$ (since $\kappa_\mathbb{Q}<|k|$), we see that $\kappa'_p\leq \widehat{\kappa}_p$ and $\kappa'_\mathbb{Q}\leq \widehat{\kappa}_\mathbb{Q}$. Thus $N'$ embeds into $\widehat{N}$ and the theorem holds in case $G$ is diagonalizable.

If $G$ is not diagonalizable, then $G$ is of the form $G=D(M)\times\Ga$ for some abelian group $M$ (Lemma~\ref{lemma: structure of one-generated group}). Similarly, $H=D(N)\times \Ga^e$ with $e\in\{0,1\}$. As there are no nontrivial morphisms from $D(M)$ to $\Ga$ or from $\Ga$ to $D(N)$, we see that $G\twoheadrightarrow H$ is a product of two quotient maps $D(M)\twoheadrightarrow D(N)$ and $\Ga\twoheadrightarrow\Ga^e$. Write $h=(h_1,h_2)\in (D(N)\times\Ga)(k)$. Then $h_1$ generates $D(N)$ and $h_2$ generates $\Ga^e$. By the diagonalizable case established above, $h_1$ lifts to a generator $g_1$ of $D(M)$. The morphism  $\Ga\twoheadrightarrow\Ga^e$ is either trivial ($e=0$) or an isomorphism ($e=1$). In any case, $h_2$ lifts to a generator $g_2$ of $\Ga$. Then $g=(g_1,g_2)$ is a generator of $G$ (Lemma \ref{lemma: one generation}).	
\end{proof}

The following example shows that Conjecture \ref{conj: Gaschuetz for proalgebraic} (and therefore also Conjecture \ref{conj: Gamma_d is d-saturated}) does not hold without the assumption that $\rank(H)<|k|$.

\begin{ex} \label{ex: need rank assumption }
	As noted in the proof of Theorem \ref{theo: conjecture for d=1}, the abelian group $k^\times$ is isomorphic to $(\mathbb{Q}/\mathbb{Z})\oplus V$, where $V$ is a $\mathbb{Q}$-vector space of dimension $|k|$. It is therefore possible to find a proper subgroup of $k^\times$ that is isomorphic to $k^\times$. In particular, there exists an abelian group $M$ containing $k^\times$ as a proper subgroup such that $M$ can be embedded into $k^\times$. 
	
	We claim that the morphism $\f\colon D(M)\twoheadrightarrow D(k^\times)$ (corresponding to the inclusion $k^\times \subseteq M$) and the element $h=\id\in D(k^\times)(k)=\Hom(k^\times,k^\times)$ does not satisfy the conclusion of the Gasch\"{u}tz Lemma. Note that $h$ generates $D(k^\times)$ (Lemma \ref{lemma: one generation of diagonalizable}) and $d(D(M))\leq 1$ (Corollary \ref{cor: one generation of diagonalizable}). Suppose there exists a generator $g\in D(M)(k)=\Hom(M,k^\times)$ of $D(M)$ such that $\f(g)=h$. We would then have a commutative diagram
	$$
	\xymatrix{
		k^\times \ar@{^(->}[r] \ar^{\id}[rd] & M \ar^-{g}[d] \\
		& k^\times
	}
	$$
	with $g$ injective. This is impossible as $k^\times$ is a proper subgroup of $M$. 
	Thus $h$ does not lift to a generator of $D(M)$. Note that $\rank(D(k^\times))=|k|$.
\end{ex}

The following example shows that the implication (ii)$\Rightarrow$(i) in Proposition \ref{prop: characterization for countable} does not hold if $k$ is not countable.

\begin{ex} \label{ex: counterexample for algebraic characterization}
	We will construct a proalgebraic group $\Gamma$ with $d(\Gamma)=1$ such that every algebraic embedding problem 
	\begin{equation} \label{eq: embedding problem2}
		\xymatrix{
			\Gamma \ar@{..>}[d] \ar@{->>}^-{\beta}[rd] & \\
			G \ar@{->>}^-{\alpha}[r]& H
		}
	\end{equation} 
	for $\Gamma$ with $d(G)\leq 1$ has a solution, but $\Gamma$ is not isomorphic to $\Gamma_1$ unless $k$ is countable.
	
	Indeed, consider $\Gamma=D(L)\times \Ga$ with $L=\mathbb{Q}/\mathbb{Z}\oplus V$, where $V$ is $\mathbb{Q}$-vector space of countably infinite dimension.
	% Fix an embedding $g\colon M\hookrightarrow k^\times$. Then $x=(g,1)\in \Gamma(k)$ is a generator of $\Gamma$ (Lemma \ref{lemma: one generation}).
	As $L$ embeds into $k^\times$, it follows from Lemma \ref{lemma: one generation} that $d(\Gamma)=1$.
	
	Consider embedding problem (\ref{eq: embedding problem2}). We know (Lemma \ref{lemma: structure of one-generated group}) that $G$ is of the form $G=D(M)\times\Ga^e$, for a finitely generated abelian group $M$ and $e\in\{0,1\}$. 
	
	Let us first assume that $G$ is diagonalizable (i.e., $e=0$). Then also $H$ is diagonalizable, say $H=D(N)$. Then (\ref{eq: embedding problem2}) dualizes to
		$$
	\xymatrix{
		 L &  \\
		M  \ar@{^(..>}[u]   & N.	 \ar@{_(->}[l] \ar@{_(->}[lu] 
	}
	$$ 
That is, given embeddings $N\hookrightarrow M$ and $N\hookrightarrow L$, we have to show that $M$ embeds into $L$ over $N$. Note that $M$ embeds into $k^\times=\mathbb{Q}/\mathbb{Z}\oplus\bigoplus_{|k|}\mathbb{Q}$ by Corollary \ref{cor: one generation of diagonalizable}. As $M$ is countable, this implies that $M$ embeds into $L$.
 Because $L$ is divisible, we can, as in the proof of Theorem \ref{theo: conjecture for d=1}, replace $M$ and $N$ by their divisible hulls $E(M)$ and $E(N)$. Note that because $N$ is finitely generated, it has finite rank and so also $E(N)$ has finite rank. Thus, assuming without loss of generality that $M$ and $N$ are divisible, if we write $L=N\oplus \widehat{N}$ and 
 $\widehat{N}=\bigoplus_{p}\oplus_{\widehat{\kappa}_p} \mathbb{Z}(p^\infty)\oplus \bigoplus_{\widehat{\kappa}_\mathbb{Q}}\mathbb{Q}$, then $\widehat{\kappa}_\mathbb{Q}=|\mathbb{N}|$. So if $M=N\oplus N'$, we can embed $N'$ into $\widehat{N}$ (as in the proof of Theorem \ref{theo: conjecture for d=1}), thereby finding an embedding of $M$ into $L$ over $N$.
 
 If $G$ is not diagonalizable, then $G=D(M)\times \Ga$ and $N=D(N)\times\Ga^e$ with $e\in\{0,1\}$. Moreover, $\alpha\colon G\twoheadrightarrow H$ is a product of two quotient maps $\alpha_1\colon D(M)\twoheadrightarrow D(N)$ and $\alpha_2\colon\Ga\twoheadrightarrow \Ga^e$. Similarly $\beta\colon \Gamma\twoheadrightarrow H$ is a product of quotient maps $\beta_1\colon D(L)\twoheadrightarrow D(N)$ and $\beta_2\colon\Ga\twoheadrightarrow \Ga^e$. By the above discussion, the embedding problem
 $$ 
 	\xymatrix{
 		D(L) \ar@{..>}_-{\f_1}[d] \ar@{->>}^-{\beta_1}[rd] & \\
 		D(M) \ar@{->>}^-{\alpha_1}[r]& D(N)
 	}
 $$
 has a solution $\f_1$. The embedding problem 
  $$ 
 \xymatrix{
 	\Ga \ar@{..>}_-{\f_2}[d] \ar@{->>}^-{\beta_2}[rd] & \\
 	\Ga \ar@{->>}^-{\alpha_2}[r]& \Ga^{e}
 }
 $$
 also has a solution $\f_2$: For $e=0$, any isomorphism will do and for $e=1$, $\alpha_2,\beta_2$ are isomorphisms and we have a unique solution $\f_2=\alpha_2^{-1}\beta_2$. Then $\f=\f_1\times\f_2$ is a sought for solution to embedding problem~(\ref{eq: embedding problem2}).
 
 Note that $\Gamma$ is isomorphic to $\Gamma_1$ if and only if $k$ is countable (Example \ref{ex: Gamma1}).
\end{ex}

\section{Regular singular differential embedding problems}

\label{sec: Regular singular differential embedding problems}

In this section we show that every regular singular differential embedding problem over $k(x)$ has a regular singular solution (Theorem \ref{theo: solve regular singular embedding problem}) and deduce from this our main results concerning the structure of the differential Galois group of the family of all regular singular differential equations over $k(x)$ with singularities contained in a given subset of $\mathbb{P}^1(k)$.

\subsection{Differential Galois theory}

We recall the basic definitions and results from differential Galois, introducing our notation for the later sections along the way. For more background, see any of the textbooks \cite{CrespoHajto:AlgebraicGroupsAndDifferentialGaloisTheory}, \cite{Magid:LecturesOnDifferentialGaloisTheory}, \cite{Sauloy:DifferentialGaloisTheoryThroughRiemannHilbertCorrespondence} or \cite{SingerPut:differential}. 

Throughout this subsection we let $(K,\delta)$ be a differential field with field of constants $k=K^\de=\{a\in K|\ \de(a)=0\}$ algebraically closed and of characteristic zero.

A differential field extension $L/K$ is a \emph{Picard-Vessiot extension} for $\delta(y)=Ay$, where $A\in K^{n\times n}$, if $L^\de=K^\de$ and there exists $Y\in\Gl_n(L)$ with $\de(Y)=AY$ and $L=K(Y_{ij}|\ 1\leq i,j\leq n)$. Such a $Y$ is called a \emph{fundamental solution matrix} for $\de(y)=Ay$. The $K$-$\de$-subalgebra $R=K[Y,\frac{1}{\det(Y)}]=K[Y_{ij},\frac{1}{\det(Y)}]$ is $\de$-simple, i.e., has no non-trivial $\de$-ideals and is called a \emph{Picard-Vessiot ring} for $\delta(y)=Ay$. The functor $G(L/K)$ from the category of $k$-algebras to the category of groups, given by $$G(L/K)(C)=\operatorname{Aut}(R\otimes_kC/K\otimes_k C)$$ for any $k$-algebra $C$, is the \emph{differential Galois group} of $L/K$ ($R/K$ or $\de(y)=Ay$). Here $C$ is considered as a constant differential ring and the automorphisms are understood to be $K\otimes_k C$\=/$\de$\=/automorphisms. The functor $G(L/K)$ is representable by a finitely generated $k$-algebra, i.e., $G(L/K)$ is an algebraic group. The above definitions generalize from a single differential equation $\delta(y)=Ay$ to a potentially infinite family $(\de(y)=A_iy)_{i\in I}$ of differential equations (where $A_i\in K^{n_i\times n_i}$) in a straightforward fashion. In this generality, $G(L/K)$ is a proalgebraic group (but not necessarily algebraic). More precisely, $G=G(L/K)$ is represented by $k[G]=(R\otimes_K R)^\de$, the canonical map $\eta\colon R\otimes_k k[G]\to R\otimes_K R$ is an isomorphism and the (functorial) action of $G$ on $R$ is algebraically described by the map $R\to R\otimes_k k[G],\ r\mapsto \eta^{-1}(1\otimes r)$. A Picard-Vessiot ring or extension that is a Picard-Vessiot ring or extension for a single differential equation (equivalently for a finite family of differential equations) is called \emph{of finite type}. 

The differential Galois correspondence asserts that the assignment $M\mapsto G(L/M)$ is an inclusion reversing bijection between the intermediate differential fields of $L/K$ and the closed subgroups of $G(L/K)$ with inverse $H\mapsto L^H$. Furthermore, a closed subgroup $N$ of $G(L/K)$ is normal if and only if $M=L^N$ is a Picard-Vessiot extension of $K$ and in this case the restriction map $G(L/K)\to G(M/K)$ is a quotient map with kernel $N$. In particular, $G(M/K)=G(L/K)/G(L/M)$.

A \emph{differential module} $(M,\nabla)$ over $(K,\de)$ is a finite dimensional $K$-vector space $M$ together with an additive map $\nabla\colon M\to M$ satisfying $\nabla(am)=\de(a)m+a\nabla(m)$ for all $a\in K$ and $m\in M$. To a differential equation $\de(y)=Ay$, where $A\in K^{n\times n}$, we can associate the differential module $(K^n,\nabla_A)$, where $\nabla_A(\xi)=\de(\xi)-A\xi$ for $\xi\in K^n$. Conversely, if $(M,\nabla)$ is a differential module over $(K,\delta)$ with $K$-basis $\underline{e}=(e_1,\ldots,e_n)$, we can write $\nabla(\underline{e})=\underline{e}(-A)$ for some $A\in K^{n\times n}$. Then $\nabla(\underline{e}\xi)=\underline{e}\de(\xi)+\nabla(\underline{e})\xi=\underline{e}(\de(\xi)-A\xi)$.
So up to the choice of the basis $\underline{e}$, $(M,\nabla)$ is of the form $(K^n,\nabla_A)$. For a different choice of basis, say $\underline{e}=\underline{e}'T$ with $T\in\Gl_n(K)$ and $\nabla(\underline{e}')=\underline{e}'(-A')$, we have
\begin{equation} \label{eq: Gauge equivalent}
	A'=TAT^{-1}+\delta(T)T^{-1}.
\end{equation}
Two differential equations $\de(y)=Ay$ and $\de(y)=A'y$ with $A,A'\in K^{n\times n}$ are \emph{Gauge equivalent} (over $K$) if there exists a $T\in\Gl_n(K)$ such that (\ref{eq: Gauge equivalent}) is satisfied.
In this case, if $Y$ is a fundamental solution matrix for $\de(y)=Ay$, then $TY$ is a fundamental solution matrix for $\de(y)=A'y$.

If $(M,\nabla)$ is a differential module over $K$ and $K'/K$ is an extension of differential fields, then we have a differential module $(M',\nabla')$ obtained from $(M,\nabla)$ via base change from $K$ to $K'$. Explicitly, $M'=M\otimes_KK'$ and $\nabla' (m\otimes a)=\nabla (m)\otimes a+m\otimes\delta(a)$ for $m\in M$ and $a\in K'$. In terms of differential equations, this simply means that a differential equation $\de(y)=Ay$ over $K$ is considered as a differential equation over $K'$.

\subsection{Regular singular differential equations}

In this subsection we recall or establish several basic properties of regular singular differential equations and related concepts.

\subsubsection{Regular singular differential modules and Picard-Vessiot rings} \label{subsubsection: Regular singular differential modules and Picard-Vessiot rings}
We begin by recalling the basic definitions and results concerning regular singular differential equations.
For more background see \cite[Chapters 5 and 6]{SingerPut:differential}, \cite{Sauloy:DifferentialGaloisTheoryThroughRiemannHilbertCorrespondence}, \cite[Part I]{MitschiSauzin:DivergentSeriesSummabilityAndResurgenceI}, \cite[Section 11]{Katz:NilpotentConnectionsAndTheMonodromyTheorem} or \cite[Chapter III]{AndreBaldassarriCailotto:DeRhamCohomologyOfDifferentialModulesOnAlgebraicVarieties}.

 We first consider the differential field $(k((t)),\frac{d}{dt})$ of formal Laurent series equipped with the usual derivation with respect to $t$. A \emph{lattice} $\Lambda$ in a finite dimensional $k((t))$-vector space $M$ is a $k[[t]]$-submodule of $M$ of the form $\Lambda=k[[t]]e_1+\ldots+k[[t]]e_n$, where $e_1,\ldots,e_n$ is a $k((t))$-basis of $M$.  
 A differential module $(M,\nabla)$ over $(k((t)),\frac{d}{dt})$ is \emph{regular} if there exists a lattice $\Lambda$ in $M$ such that $\nabla(\Lambda)\subseteq\Lambda$. It is \emph{regular singular} if there exists a lattice $\Lambda$ in $M$ such that $\nabla(\Lambda)\subseteq\frac{1}{t}\Lambda$. %For $A\in k((t))^{n\times n}$, we call $\de(y)=Ay$ \emph{regular singular} if the corresponding differential module $(k((t))^n,\nabla_A)$ is regular singular. 
 
 \medskip
 
 We now focus on the field $k(x)$ of rational functions over $k$. Throughout this article, unless something else is indicated, we consider $k(x)$ as a differential field equipped with the usual derivation $\frac{d}{dx}$ with respect to $x$. For $p\in k$, we set 
 $\widehat{k(x)}^p=k((x-p))$. We also set $\widehat{k(x)}^\infty=k((x^{-1}))$. Writing $t=x-p$, we see that $(\widehat{k(x)}^p, \frac{d}{dt})$ is a differential field extension of $(k(x),\frac{d}{dx})$. Similarly, writing $t=x^{-1}$, so that $\frac{d}{dt}x=\frac{d}{dt}t^{-1}=-t^{-2}$, we see that $(\widehat{k(x)}^\infty,-t^2\frac{d}{dt})$ is a differential field extension of $k(x)$.
 
 Let $(M,\nabla)$ be a differential module over $k(x)$. Then $(M,\nabla)$ defines for every $p\in \mathbb{P}^1(k)=k\cup\{\infty\}$ a differential module $(M_p,\nabla_p)$ over $\widehat{k(x)}^p=k((t))$: For $p\in k$, $(M_p,\nabla_p)$ is the base change of $(M,\nabla)$ from $k(x)$ to $\widehat{k(x)}^p=k((t))$. For $p=\infty$ we can also base change $(M,\nabla)$ from $k(x)$ to $\widehat{k(x)}^\infty$ to obtain a differential module $(M',\nabla')$ over $(k((t)),-t^2\frac{d}{dt})$. To have a differential module $(M_\infty,\nabla_\infty)$ over $(k((t)),\frac{d}{dt})$ we set $M_\infty=M'$ and $\nabla_\infty(m')=-\frac{1}{t^2}\nabla'(m')$ for $m'\in M'$.
 
 The differential module $(M,\nabla)$ is \emph{regular singular} if $(M_p,\nabla_p)$ is regular singular for every $p\in\mathbb{P}^1(k)$. A point $p\in\mathbb{P}^1(k)$ is a \emph{singularity} of $(M,\nabla)$ if $(M_p,\nabla_p)$ is not regular. For $A\in k(x)^{n\times n}$, the differential equation $\de(y)=Ay$ is \emph{regular singular} if the corresponding differential module $(M,\nabla)=(k(x)^n,\nabla_A)$ is regular singular. A point $p\in\mathbb{P}^1(k)$ is a \emph{singularity} of $\de(y)=Ay$ if it is a singularity of $(M,\nabla)$. In this case we also say that $\de(y)=Ay$ is \emph{singular at $p$}. If $\de(y)=Ay$ is not singular at $p$, then it is \emph{regular at $p$}. I.e., $\de(y)=Ay$ is regular at $p$ if and only if $(M_p,\nabla_p)$ is regular.
 It follows that $\de(y)=Ay$ is regular (or regular singular) at $p\in k$ if $\de(y)=A(x+p)y$ is regular (or regular singular) at $0$ and $\de(y)=Ay$ is regular (or regular singular) at $\infty$ if $\de(y)=\frac{-1}{x^2}A(\frac{1}{x})y$ is regular (or regular singular) at $0$.

 If $p\in k$ is not a pole of (one of the entries of) $A$, then $k[[t]]^n$ is a lattice in $(M_p,\nabla_p)$ and therefore $\de(y)=Ay$ is regular at $p$. In particular, $\de(y)=Ay$ only has finitely many singularities. Conversely, if $p\in k$ is a pole of $A$, then $p$ may or may not be a singularity of $\de(y)=Ay$. A pole of $A$ that is not a singularity of $\de(y)=Ay$ is often called an \emph{apparent singularity} in the literature. But note that an apparent singularity is not a singularity in our terminology.

\begin{defi} \label{defi: regsing PV extension}
	Let $S$ be a subset of $\mathbb{P}^1(k)$. A Picard-Vessiot extension $L/k(x)$ (or a Picard-Vessiot ring $R/k(x)$) is \emph{regular singular} with respect to $S$ if it is the Picard-Vessiot extension (or Picard-Vessiot ring) for a regular singular differential equation over $k(x)$ with singularities contained in $S$.  
\end{defi}
For brevity, we also say that a differential equation $\de(y)=Ay$, or a differential module over $k(x)$, is regular singular with respect to $S$ if it is regular singular with singularities contained in $S$. We refer to \cite[Section 2.2]{SingerPut:differential} for the basic linear algebra constructions with differential modules.

\begin{lemma} \label{lemma: regular singular with respect to S is closed}
	Let $S\subseteq\mathbb{P}^1(k)$. Then differential submodules, quotient modules, direct sums, duals and tensor products of differential modules over $k(x)$ that are regular singular with respect to $S$ are regular singular with respect to $S$.	
\end{lemma}	
\begin{proof}
	This is well-known, see e.g., \cite[Lemma 3.10]{SingerPut:differential} for the local case, which implies the global case.
\end{proof}	

Lemma \ref{lemma: regular singular with respect to S is closed} shows that the differential modules over $k(x)$ that are regular singular with respect to $S$ are a tannakian category, indeed a tannakian subcategory of the tannakian category of all differential modules over $k(x)$.

At first sight, it may seem conceivable that a Picard-Vessiot extension that is regular singular with respect to $S$ is also a Picard-Vessiot extension for a differential equation that is not regular singular, or that is regular singular but has singularities outside of $S$. However, we will soon see that this is not possible. To this end, we briefly need to discuss \emph{$\de$-finite elements}. Let $(K,\de)$ be a differential field and $R$ a $K$-$\de$-algebra. An element $r\in R$ is $\delta$-finite if it is contained in a finite dimensional $K$-subspace $V$ of $R$ that is $\de$-stable, i.e., $\de(V)\subseteq V$, equivalently, the $K$-subspace generated by $r,\de(r),\de^2(r),\ldots$ is finite dimensional. One readily verifies that sums and products of $\de$-finite elements are $\de$-finite, so that the $\de$-finite elements form a $K$-$\de$-subalgebra of $R$. If $R/K$ is a Picard-Vessiot ring for $\de(y)=Ay$ so that $R=K[Y,\frac{1}{\det(Y)}]$ with $\de(Y)=AY$, then the $K$-subspace generated by all $Y_{ij}$'s is $\de$-stable and $\de(\frac{1}{\det(Y)})=-\operatorname{tr}(A)\frac{1}{\det(Y)}$, where $\operatorname{tr}(A)$ denotes the trace of $A$. Thus $R$ is generated by $\de$-finite elements and therefore all elements of $R$ are $\de$-finite. It follows that $R$ is the directed union of its
finite dimensional $\de$-stable $K$-subspaces. Note that any finite dimensional $\de$-stable $K$-subspace of $R$ is a differential module over $K$.

\begin{lemma} \label{lemma: characterize regular singular}
	Let $S\subseteq\mathbb{P}^1(k)$, $A\in k(x)^{n\times n}$ and $R/k(x)$ a Picard-Vessiot ring for $\de(y)=Ay$. Then $\de(y)=Ay$ is regular singular with respect to $S$ if and only if every finite dimensional $\de$-stable $k(x)$\=/subspace of $R$ is regular singular with respect to $S$.
\end{lemma}
\begin{proof}
	Let $M=(k(x)^n,\nabla_A)$ denote the differential module corresponding to $\de(y)=Ay$. We first assume that $\de(y)=Ay$ is regular singular with respect to $S$. Let $X$ be an $n\times n$ matrix of indeterminates over $k(x)$. The Picard-Vessiot ring $R$ is a quotient of the universal solution ring $U=k(x)\big[X,\frac{1}{\det(X)}\big]$, where $\de$ on $U$ is determined by $\de(X)=AX$. By Lemma \ref{lemma: regular singular with respect to S is closed}, it suffices to show that $U$ is the sum of differential modules that are regular singular with respect to $S$. For $d\geq 0$ let $M_d$ denote the $k(x)$-subspace of $U$ generated by all monomials in the entries of $X$ of degree at most $d$. 
	
	As $M$ is regular singular with respect to $S$, by Lemma \ref{lemma: regular singular with respect to S is closed}, so is its dual $M^*\simeq (k(x)^n,\nabla_{-A^{\operatorname{T}}})$. We have 
	$$\de(X_{1j},\ldots,X_{nj})=(X_{1j},\ldots,X_{nj})A^{\operatorname{T}}$$ for $j=1,\ldots,n$, which shows that $M_1$ is an $n$-fold direct sum of copies of $M^*$. So $M_1$ is regular singular with respect to $S$. 
	Since $M_d$ is a quotient of $M_1^{\otimes d}$, it follows that $M_d$ is regular singular with respect to $S$. 
	
	Note that $\de(\det(X))=\operatorname{tr}(A)\det(X)$ and therefore
	$\frac{M_d}{\det(X)^e}=\{\frac{m}{\det(X)^e}|\ m\in M_d\}$ is a $\de$-stable $k(x)$\=/subspace of $U$ for all $d,e\geq 0$. As $U$ is the sum of the differential modules of the form $\frac{M_d}{\det(X)^e}$, it suffices to see that those differential modules are regular singular with respect to $S$. The $n$-th exterior power $\wedge^nM$ of $M$ is isomorphic to $(k(x),\nabla_{\operatorname{tr}(A)})$. It follows that $\frac{M_d}{\det(X)^e}$ is isomorphic to $M_d\otimes(\wedge^nM)^{\otimes e}$ and therefore is regular singular with respect to $S$ (Lemma \ref{lemma: regular singular with respect to S is closed}).
	
	For the reverse direction, note that $M$ becomes trivial over $R$, because $R$ contains a fundamental solution matrix for $\de(y)=Ay$, i.e., $M\otimes_{k(x)}R\simeq R^n$ as differential modules over $R$. In particular, $M$ embeds into $R^n$ as a $\de$-stable $k(x)$-subspace and therefore must be regular singular with respect to $S$.	 
\end{proof}

\begin{cor} \label{cor: reg sing stable}
	For $S\subseteq\mathbb{P}^1(k)$ let $k(x)\subseteq L\subseteq M$ be an inclusion of Picard-Vessiot extensions of $k(x)$ such that $M/k(x)$ is regular singular with respect to $S$ and let $A\in k(x)^{n\times n}$ be such that $L/k(x)$ is a Picard-Vessiot extension for $\de(y)=Ay$. Then $\de(y)=Ay$ is regular singular with respect to $S$. In particular, $L/k(x)$ is regular singular with respect to $S$.
\end{cor}
\begin{proof}
As the Picard-Vessiot ring for $L$ is contained in the Picard-Vessiot ring for $M$, this follows from Lemma \ref{lemma: characterize regular singular}.
\end{proof}	

Corollary \ref{cor: reg sing stable} shows, in particular, that a Picard-Vessiot extension for $\de(y)=Ay$ is regular singular with respect to $S$ if and only if $\de(y)=Ay$ is regular singular with respect to $S$.

%
%\begin{lemma} \label{lemma: no poles outside S}
%Let $S$ be a finite subset of $\mathbb{P}^1(\C)$ and let $R/\C(x)$ be a Picard-Vessiot ring that is regular singular with respect to $S$. Then there exists an $n\geq 1$ and a matrix $A\in \C(x)^{n\times n}$ such that $\de(y)=Ay$ has no poles outside $S$ and $R/\C(x)$ is a Picard-Vessiot ring for $\de(y)=Ay$.
%\end{lemma}
%\begin{proof}
%By the Riemann-Hilbert correspondence (\cite[Theorem 6.15]{SingerPut:differential}) there is a bijection between the isomorphism classes of regular singular differential modules with singularities in $S$ and isomorphism classes of (finite dimensional, $\C$-linear) representation of the topological fundamental group $\pi_1(\mathbb{P}^1(\C)\smallsetminus S)$. By the (weak) solution of the Riemann-Hilbert problem, also known as Plemelj's Theorem (see \cite[Theorem 3.2.1]{AnasovBolibruch:TheRiemannHilbertProblem} or \cite[Theorem 12.4]{Sauloy:DifferentialGaloisTheoryThroughRiemannHilbertCorrespondence}) any representation of $\pi_1(\mathbb{P}^1(\C)\smallsetminus S)$ can be realized by a differential equation $\de(y)=Ay$ with no poles outside of $S$ (and all but one point of $S$ Fuchsian but this is irrelevant for our purpose). So, by the above bijection, any regular singular differential equation with singularities in $S$ is Gauge equivalent over $\C(x)$ to a differential equation with no poles outside of $S$. As Gauge equivalent differential equations define the same Picard-Vessiot ring, the claim follows.
%\end{proof}	
%
%

\subsubsection{Gauge equivalence over $k(x)$}
Let $A\in k(x)^{n\times n}$. Recall (Section \ref{subsubsection: Regular singular differential modules and Picard-Vessiot rings}) that whether or not $\de(y)=Ay$ is regular or regular singular, say at $p=0$, is defined via Gauge transformations over $k((x))$. Explicitly, $\de(y)=Ay$ is regular (or regular singular) at $p=0$, if there exist a matrix $T\in\Gl_n(k((x)))$ such that $A'=TAT^{-1}+\de(T)T^{-1}$ has no pole (or at most a simple pole) at $p=0$. It turns out that the existence of such a matrix $T\in\Gl_n(k((x)))$ is equivalent to the existence of such a matrix $T\in\Gl_n(k(x))$. This is useful because when dealing with base change and specializations later, it will be easier to work with matrices over $k(x)$ than over $k((x))$. Also, for $k=\mathbb{C}$, whether $p=0$ is classified as regular or regular singular is usually defined via local meromorphic Gauge transformations (see e.g., \cite[Section 5.1.2]{SingerPut:differential}), i.e., one considers matrices $T\in \Gl_n(\C(\{x\}))$. For our later arguments to be sound, it is important to know that our algebraic definition is equivalent to the usual definition for $k=\C$.  
%
%
%For $k=\mathbb{C}$, in the regular singular case, this follows for example from \cite[Theorem 12.4]{Sauloy:DifferentialGaloisTheoryThroughRiemannHilbertCorrespondence} or \cite{Moser:TheOrderOfASingularityInFuchsTheory} (cf. \cite{Barkatou:ARationalVersionOfMosersAlgorithm}).
%For the regular singular case we can use \cite[Lemma 2.7]{SerafiniWibmer:RegularSingularInverseProblem} as a reference.

\begin{lemma} \label{lemma: Gauge over rational}
	Let $A\in k(x)^{n\times n}$.
\begin{enumerate}
	\item If $\de(y)=Ay$ is regular at $p=0$, then $\de(y)=Ay$ is Gauge equivalent over $k(x)$ to $\de(y)=A'y$, where $A'\in k(x)^{n\times n}$ has no pole at $p=0$.
	\item If $\de(y)=Ay$ is regular singular at $p=0$, then $\de(y)=Ay$ is Gauge equivalent over $k(x)$ to $\de(y)=A'y$, where $A'\in k(x)^{n\times n}$ has at most a simple pole at $p=0$, i.e., $xA'\in k[x]_{(x)}^{n\times n}$.	
\end{enumerate}	
\end{lemma}
\begin{proof}
	We consider $k((x))$ as a topological field equipped with the $x$-adic topology and we let $v_x$ denote the $x$-adic valuation. Let $U$ be the open subset of $k((x))^{n\times n}$ consisting of all matrices $A'\in k((x))^{n\times n}$ with $v_x(A'_{ij})\geq 0$ for $1\leq i,j\leq n$. As the map $\psi\colon \Gl_n(k((x)))\to k((x))^{n\times n},\ T\mapsto TAT^{-1}+\de(T)T^{-1}$ is continuous, $\psi^{-1}(U)\subseteq \Gl_n(k((x)))$ is open. By assumption,  $\psi^{-1}(U)$ is non-empty. Because $\Gl_n(k(x))$ is dense in $\Gl_n(k((x)))$, there exists a $T\in \Gl_n(k(x))\cap \psi^{-1}(U)$. The matrix $A'=TAT^{-1}+\de(T)T^{-1}\in k(x)^{n\times n}$ then has the required property. This proves (i). The proof for (ii) is similar, with $U$ defined by $v_x(A'_{ij})\geq -1$. 
\end{proof}

\subsubsection{Base change for regular singular differential equations and Picard-Vessiot rings}

We show that several properties related to regular singular differential equations remain unchanged if we enlarge the constants. The following lemma shows that whether or not a point is regular (or regular singular) does not depend on the field of constants over which the differential equation is considered.
\begin{lemma} \label{lemma: extends constants for equation}
	Let $k'/k$ be an extension of algebraically closed fields, $p\in \mathbb{P}^1(k)$ and $A\in k(x)^{n\times n}$.
	Then $\de(y)=Ay$ is regular (or regular singular) at $p$, where $\de(y)=Ay$ is considered as a differential equation over $k(x)$, if and only if $\de(y)=Ay$ is regular (or regular singular) at $p$, where $\de(y)=Ay$ is considered as a differential equation over $k'(x)$.  
\end{lemma}
\begin{proof}
	Clearly, if $\de(y)=Ay$ is regular (or regular singular) at $p$, when considered over $k(x)$, then $\de(y)=Ay$ is regular (or regular singular) at $p$, when considered over $k'(x)$. 
	
	For the reverse implication, we may assume that $p=0$. By Lemma \ref{lemma: Gauge over rational} there exists a matrix $T\in\Gl_n(k'(x))$ such that
	$A'=TAT^{-1}+\de(T)T^{-1}$ has entries in $k'[x]_{(x)}$ (or $\frac{1}{x}k'[x]_{(x)}$). The existence of such a matrix $T$ can be expressed via a system of algebraic equations over $k$. In more detail, if $f\in k'[x]\smallsetminus\{0\}$ is such that $fT\in k'[x]^{n\times n}$, then writing $A'=\sum_{i\geq m}A_i'x^i\in k'((x))^{n\times n}$ with $A_i'\in k'^{n\times n}$, the coefficients of the $A'_i$'s are rational functions over $k$ in the coefficients of $f$ and the coefficients of the entries of $fT$. That $A'$ has entries in $k'[x]_{(x)}$ (or $\frac{1}{x}k'[x]_{(x)}$) means that $A_i'=0$ for $i<0$ (or $i< -1$), i.e., all the corresponding rational functions must vanish.
	Because $k$ is algebraically closed, a system of algebraic equations (in finitely many variables) over $k$ that has a solution in a field extension of $k$ (in this case in $k'$) must already have a solution in $k$. Thus there exists a matrix $T_0\in\Gl_n(k(x))$ such that $T_0AT_0^{-1}+\de(T_0)T_0^{-1}$ has entries in $k[x]_{(x)}$ (or $\frac{1}{x}k[x]_{(x)}$). This shows that $\de(y)=Ay$ is regular (or regular singular) at $p=0$, when $\de(y)=Ay$ is considered as a differential equation over $k(x)$.
\end{proof}

The following lemma shows that we can extend the constants of a Picard-Vessiot ring and that this operation is well-behaved with respect to being regular singular.

\begin{lemma} \label{lemma: extend constants for PV}
	Let $k'/k$ be an extension of algebraically closed fields and let $S\subseteq \mathbb{P}^1(k)$.
	Let $R/k(x)$ be a Picard-Vessiot ring with differential Galois group $G$ and set $R'=R\otimes_{k(x)}k'(x)$. Then $R'/k'(x)$ is a Picard-Vessiot ring with differential Galois group $G_{k'}$. Moreover, $R/k(x)$ is regular singular with respect to $S$ if and only if $R'/k'(x)$ is regular singular with respect to $S$.
\end{lemma}	
\begin{proof}
	It is well-known that $R'/k'(x)$ is a Picard-Vessiot ring with differential Galois group $G_{k'}$ (see e.g., \cite[Lemma 5.2]{BachmayrHarbaterHartmannWibmer:TheDifferentialGaloisGroupOfRationalFunctionField}). 
	
	Assume that $R/k(x)$ is a Picard-Vessiot ring for $\de(y)=Ay$ with $A\in k(x)^{n\times n}$. Then $R'/k'(x)$ is a Picard-Vessiot ring for $\de(y)=Ay$. First assume that $R/k(x)$ is regular singular with respect to $S$, i.e., $\de(y)=Ay$ is regular singular with singularities contained in $S$, when considered as a differential equation over $k(x)$. As $\de(y)=Ay$ is regular at all points $p\in\mathbb{P}^1(k')$ that are not poles of $A$, we see that the singularities of $\de(y)=Ay$, considered as a differential equation over $k'(x)$, are contained in $\mathbb{P}^1(k)\subseteq\mathbb{P}^1(k')$. Thus the singularities are contained in $S$ and they are all regular singularities. 
	
	Conversely, assume that $R'/k'(x)$ is regular singular with respect to $S$, i.e., $\de(y)=Ay$ is regular singular with singularities contained in $S$, when considered as a differential equation over $k'(x)$. By Lemma \ref{lemma: extends constants for equation} a point $p\in\mathbb{P}^1(k)$ is regular (or regular singular) for $\de(y)=Ay$ considered over $k(x)$ if and only if $p$ is regular (or regular singular) for $\de(y)=Ay$ considered over $k'(x)$. Thus $\de(y)=Ay$ is regular singular with singularities contained in $S$, when considered over $k(x)$. Therefore $R/k(x)$ is regular singular with respect to $S$.
\end{proof}	

It is well-known that over the complex numbers the differential Galois group of a regular singular differential equation with $d+1$ singularities can be generated by $d$ elements. It is far from obvious that this also holds for arbitrary algebraically closed fields of characteristic zero. For example, for $k=\overline{\mathbb{Q}}$ there does not seem to be a natural way to construct or to choose those generators.

\begin{lemma} \label{lemma: Galois of reg sing generated}
	Let $\de(y)=Ay$ be a regular singular differential equation over $k(x)$ with $d+1$ singularities. Then the differential Galois group of $\de(y)=Ay$ can be generated by $d$ elements.  
\end{lemma}
\begin{proof}
	For $k=\C$ this is an immediate consequence of Schlesinger's density theorem (\cite[Theorem~5.8]{SingerPut:differential}) or the Riemann-Hilbert correspondence. We will reduce the general case to the complex case via Proposition \ref{prop: d under base change}. In detail: Let $S\subseteq\mathbb{P}^1(k)$ be the set of singularities of $\de(y)=Ay$, $R/k(x)$ a Picard-Vessiot ring for $\de(y)=Ay$ and $G$ the corresponding differential Galois group.
	There exists a countable algebraically closed subfield $k_1$ of $k$ such that $A\in k_1(x)^{n\times n}$ and $S\subseteq \mathbb{P}^1(k_1)$. Let $R_1/k_1(x)$ be a Picard-Vessiot ring for $\de(y)=Ay$ and $G_1$ the corresponding differential Galois group. By the uniqueness of Picard-Vessiot rings and Lemma \ref{lemma: extend constants for PV} we have $R_1\otimes_{k_1(x)}k(x)\cong R$ and $(G_1)_k\cong G$. Furthermore, $R_1/k_1(x)$ is regular singular with respect to $S$ because $R/k(x)$ is.
	
	Since $k_1$ is countable, we can choose an embedding of $k_1$ into $\C$. Then, again by Lemma \ref{lemma: extend constants for PV}, $R_1\otimes_{k_1(x)}\C(x)$ is a Picard-Vessiot ring over $\C(x)$ and regular singular with respect to $S$. The corresponding differential Galois group is $(G_1)_\C$. From the classical case $k=\C$ discussed above, we know that $d((G_1)_\C)\leq |S|-1=d$. Because $d((G_1)_\C)=d(G_1)=d((G_1)_k)=d(G)$ by Proposition \ref{prop: d under base change}, we find $d(G)\leq d$ as desired.	
\end{proof}	

Let $S\subseteq \mathbb{P}^1(k)$ be a finite subset with $d+1$ elements and let $\Gamma_S$ denote the 
differential Galois group of the family of all regular singular differential equations over $k(x)$ with singularities in $S$. Conjecture \ref{conj: equivalent} from the introduction asserts that $\Gamma_S\cong\Gamma_d$. While we only prove this for $k$ countable, in general we at least have:

\begin{cor} \label{cor: same quotients}
	The proalgebraic groups $\Gamma_S$ and $\Gamma_d$ have the same algebraic quotients, i.e., for an algebraic group $G$, there exists a quotient map $\Gamma_S\twoheadrightarrow G$ if and only if there exists a quotient map $\Gamma_d\twoheadrightarrow G$.
\end{cor}	
\begin{proof}
	It was already shown in \cite[Cor. 2.16]{SerafiniWibmer:RegularSingularInverseProblem} that every quotient of $\Gamma_d$ is a quotient of $\Gamma_S$. For the reverse conclusion, let $L_S/k(x)$ denote the Picard-Vessiot extension for the family of all differential equations over $k(x)$ that are regular singular with respect to $S$. Let $\Gamma_S\twoheadrightarrow G$ be a quotient map with $G$ algebraic and let $N$ be its kernel. By the differential Galois correspondence, $L_S^N/k(x)$ is a Picard-Vessiot extension with differential Galois group $\Gamma_S/N\simeq G$. Because $G$ is algebraic (i.e. of finite type over $k$), $L_S^N/k(x)$ is the Picard-Vessiot extension for a single differential equation $\de(y)=Ay$ (\cite[Remark~2.3]{BachmayrHarbaterHartmannWibmer:FreeDifferentialGaloisGroups}). By Corollary \ref{cor: reg sing stable} the differential equation $\de(y)=Ay$ is regular singular with respect to $S$ and because $G$ is its differential Galois group, it follows from Lemma \ref{lemma: Galois of reg sing generated} that $G$ can be generated by $d$ elements. So $G$ is a quotient of $\Gamma_d$.
\end{proof}

In fact, later in this section we will prove a much stronger statement than Corollary~\ref{cor: same quotients}: Corollary~\ref{cor: GammaS solves embedding problems} combined with Proposition \ref{prop: embedding property for Gammad} shows that $\Gamma_S$ and $\Gamma_d$ solve the same algebraic embedding problems.

\begin{cor} \label{cor: d for countable}
	Let $S\subseteq\mathbb{P}^1(k)$ be finite and let $\Gamma_S$ denote the differential Galois group of the family of all regular singular differential equations over $k(x)$ with singularities contained in $S$. If $k$ is countable, then $d(\Gamma_S)=|S|-1$.
\end{cor}
\begin{proof}
	Set $d=|S|-1$. As $\Ga^d$ can be generated by $d$ elements, it follows from Corollary \ref{cor: same quotients} that $\Ga^d$ is a quotient of $\Gamma_S$. As $d(\Ga^d)=d$, we must have $d(\Gamma_S)\geq d$.

	Let $R_S/k(x)$ denote the Picard-Vessiot ring for the family of all regular singular differential equations over $k(x)$ with singularities contained in $S$. Because $k$ is countable, there are only countably many linear differential equations over $k(x)$, and we can find Picard-Vessiot rings $R_i/k(x)$ of finite type such that $R_0\subseteq R_1\subseteq R_2\subseteq\ldots$ and $R_S=\bigcup_{i\geq 0}R_i$. 
	The inclusions $R_i\subseteq R_{i+1}$ induce quotient maps $G_{i+1}\twoheadrightarrow G_i$ between the corresponding differential Galois groups and $\Gamma_S=\varprojlim G_i$. By Lemma \ref{lemma: Galois of reg sing generated} we have $d(G_i)\leq d$ for all $i\geq 0$. Let $g_{0,1},\ldots,g_{0,d}$ be a generating set for $G_0$. By the Gasch\"{u}tz Lemma for algebraic groups (Theorem \ref{theo: Geschuetz}), there exists generators $g_{1,1},\ldots,g_{1,d}$ of $G_1$ lifting the generators $g_{0,1},\ldots,g_{0,d}$ of $G_0$. Continuing in this way, we find for every $i\geq 0$ a generating set $g_{i,1},\ldots,g_{i,d}$ of $G_i$ such that $g_{i+1,j}$ maps to $g_{i,j}$ under $G_{i+1}\twoheadrightarrow G_i$ for every $j=1,\ldots,d$. So $(g_{i,j})_{i\geq 0}$ defines an element $g_j$ of $\varprojlim G_i(k)=\Gamma_S(k)$ for every $j=1,\ldots,d$. As the closed subgroup $\langle g_1,\ldots, g_d\rangle$ of $\Gamma_S$ surjects onto every $G_i$, we must have $\langle g_1,\ldots, g_d\rangle=\Gamma_S$. So $d(\Gamma_S)\leq d$.
\end{proof}

\subsection{Differential embedding problems}
In this subsection we introduce regular singular differential embedding problems and explain how they relate to embedding problems for the corresponding differential Galois group.

We begin by recalling the fundamental notion of a differential embedding problem (\cite[Def.~ 3.6]{BachmayrHarbaterHartmannWibmer:FreeDifferentialGaloisGroups}.) Let $K$ be a differential field with $K^\de=k$. A \emph{differential embedding problem} over $K$ is a pair $(G\twoheadrightarrow H,\ L/K)$, where $L/K$ is a Picard-Vessiot extension with differential Galois group $H$ and $G\twoheadrightarrow H$ is a quotient map of proalgebraic groups over $k$. A \emph{solution} is a Picard-Vessiot extension $M/K$ containing $L$ together with an isomorphism $G\cong G(M/K)$ that identifies $G\twoheadrightarrow H$ with the restriction map $G(M/K)\twoheadrightarrow G(L/K)$ of differential Galois groups. In the trivial case $L=K$ (so that $H=1$), a solution is simply a realization of $G$ as a differential Galois group over $K$.

We note that the Picard-Vessiot extension $L/K$ in the above definition could be a Picard-Vessiot extension for an infinite family of linear differential equations and that $H$ and $G$ are allowed to be proalgebraic (rather than just algebraic) groups. We may sometimes simply refer to $M/K$ as a solution, thereby omitting the isomorphism $G\cong G(M/K)$ from the notation. Moreover, if $R\subseteq L$ is the Picard-Vessiot ring, it will sometimes be more convenient to write $(G\twoheadrightarrow H,\ R/K)$ for the differential embedding problem $(G\twoheadrightarrow H,\ L/K)$. Similarly, if $T\subseteq M$ is the Picard-Vessiot ring, we may refer to $T/K$ (rather than $M/K$) as the solution of the differential embedding problem. With this notation, the compatibility of the isomorphism $G\simeq G(M/K)$ with the restriction map means that
$$
\xymatrix{
	R \ar[d] \ar[r] & R\otimes_kk[H] \ar[d] \\
	T \ar[r] & T\otimes_k k[G]
}
$$
commutes.

Differential embedding problems over $k(x)$ have been studied extensively in \cite{BachmayrHarbaterHartmannWibmer:DifferentialEmbeddingProblems}, \cite{BachmayrHarbaterHartmann:DifferentialEmbeddingProblemsOverLaurentSeriesFields}, \cite{BachmayrHartmannHarbaterPopLarge}, \cite{BachmayrHarbaterHartmannWibmer:FreeDifferentialGaloisGroups}, \cite{BachmayrHarbaterHartmannWibmer:TheDifferentialGaloisGroupOfRationalFunctionField} and \cite{FengWibmer:DifferentialGaloisGroupsSpecializationsAndMatzatsConjecture} with the upshot that all differential embedding problems over $k(x)$ that one might hope to be solvable are indeed solvable. More precisely, every differential embedding problem $(G\twoheadrightarrow H,\ L/k(x))$ such that $\rank(H)<|k|$ and $\rank(G)\leq |k|$ has a solution and the latter is equivalent to Matzat's conjecture (now a theorem), stating that the absolute differential Galois group of $k(x)$ is the free proalgebraic group on a set of cardinality $|k|$. (See \cite[Theorem 3.7]{BachmayrHarbaterHartmannWibmer:FreeDifferentialGaloisGroups} and \cite[Theorem 5.15]{FengWibmer:DifferentialGaloisGroupsSpecializationsAndMatzatsConjecture}.)

Ultimately, one might hope for a regular singular analog of Matzat's conjecture, stating that, for a (not necessarily finite) subset $S$ of $\mathbb{P}^1(k)$, the differential Galois group of the family of all regular singular differential equations over $k(x)$ with singularities in $S$ is the free proalgebraic group on a set of cardinality $|S|-1$ (cf. \cite[Section 4.1]{Wibmer:RegularSingularDifferentialEquationsAndFreeProalgebraicGroups}). In particular, the differential Galois group of the family of all regular singular differential equations over $k(x)$ should be the free proalgebraic group on a set of cardinality~ $|k|$.  We will later (Theorem \ref{theo: main main} and Corollary \ref{cor: free for countable}) prove these claims for $k$ countable. 
It suggests itself to approach this question by solving regular singular differential embedding problems, which we now define.
\begin{defi}
Let $S$ be a finite subset of $\mathbb{P}^1(k)$. A differential embedding problem
	$(G\twoheadrightarrow H,\ L/k(x))$ over $k(x)$ is \emph{regular singular} with respect to $S$ if $L/k(x)$ is regular singular with respect to $S$ (Definition~\ref{defi: regsing PV extension}) and $G$ is an algebraic group with $d(G)\leq|S|-1$. A solution is \emph{regular singular} with respect to $S$ if the corresponding Picard-Vessiot extension is regular singular with respect to $S$.
\end{defi}	 

We note that the condition that $G$ can be generated by $|S|-1$ elements is necessary for the existence of a regular singular solution, because the differential Galois group of a regular singular equation with singularities in $|S|$ can be generated by $|S|-1$ elements (Lemma \ref{lemma: Galois of reg sing generated}).

The following proposition relates regular singular embedding problems to embedding problems for the corresponding differential Galois group. 

\begin{prop} \label{prop: equivalence of embedding property and differential embedding problem}
	Let $S$ be a finite subset of $\mathbb{P}^1(k)$. The following statements are equivalent:
	\begin{enumerate}
		\item Every regular singular (with respect to $S$) differential embedding problem over $k(x)$ has a regular singular (with respect to $S$) solution.
	\item The differential Galois group $\Gamma_S$ of the family of all regular singular differential equations over $k(x)$ with singularities in $S$, satisfies the following embedding property: Every algebraic embedding problem $(G\twoheadrightarrow H,\ \Gamma_S\twoheadrightarrow H)$ for $\Gamma_S$ with $d(G)\leq |S|-1$ has a solution.

	\end{enumerate}
\end{prop}
\begin{proof}
Let $L_S/k(x)$ be the Picard-Vessiot extension of the family of all regular singular differential equations with singularities in $S$, so that $\Gamma_S=G(L_S/k(x))$.
	
We begin with (i)$\Rightarrow$(ii). Let $(G\twoheadrightarrow H,\ \Gamma_S\twoheadrightarrow H)$ be an algebraic embedding problem for $\Gamma_S$ with $d(G)\leq |S|-1$. The kernel $N$ of $\Gamma_S\twoheadrightarrow H$ is a normal closed subgroup of $\Gamma_S$ and, by the differential Galois correspondence, the corresponding intermediate differential field $L=L_S^N$ is a Picard-Vessiot extension of $k(x)$ with differential Galois group $\Gamma_S/N=H$. Because $H$ is algebraic (i.e., of finite type over $k$) $L/K$ is a Picard-Vessiot extension for a single differential equation $\de(y)=Ay$ (\cite[Remark 2.3]{BachmayrHarbaterHartmannWibmer:FreeDifferentialGaloisGroups}). By Corollary \ref{cor: reg sing stable} the differential equation $\de(y)=Ay$ is regular singular with respect to $S$, i.e., $L/K$ is regular singular with respect to $S$. By (i) the regular singular differential embedding problem $(G\twoheadrightarrow H,\ L/k(x))$ has a regular singular solution $M/k(x)$. We can embed $M$ into $L_S$ over $L$. %While such an embedding is not unique, the image of the embedding is unique because $M/L$ is a Picard-Vessiot extension. 
So we can consider $M$ as a differential subfield of $L_S$. In particular, we have a restriction map $\Gamma_S\twoheadrightarrow G(M/k(x))$. The commutative diagram
$$ 
\xymatrix{
\Gamma_S \ar@{->>}[rd] \ar@{->>}[d] & \\
G(M/k(x)) \ar[d] \ar@{->>}[r] & G(L/k(x)) \ar[d] \\
G \ar@{->>}[r]  & H	
}
$$
shows that the embedding problem $(G\twoheadrightarrow H,\ \Gamma_S\twoheadrightarrow H)$ for $\Gamma_S$ has a solution.

Conversely, let $(G\twoheadrightarrow H,\ L/k(x))$ be a regular singular (with respect to $S$) differential embedding problem. We may assume that $L$ is contained in $L_S$. We thus have a restriction morphism $\Gamma_S\twoheadrightarrow H$ of differential Galois groups. By (ii), the embedding problem $(G\twoheadrightarrow H,\ \Gamma_S\twoheadrightarrow H)$ for $\Gamma_S$ has a solution $\Gamma_S\twoheadrightarrow G$. For the kernel $N$ of this solution, $M=L_S^N$ is a Picard-Vessiot extension of $k(x)$ containing $L$ (because $N$ is contained in the kernel of $\Gamma_S\twoheadrightarrow H$.) Furthermore $G(M/k(x))=\Gamma_S/N\simeq G$. Because $G$ is algebraic, $M$ is the Picard-Vessiot extension for a single differential equation and it follows from Corollary \ref{cor: reg sing stable} that $M/k(x)$ is regular singular (with respect to $S$).
In the diagram
$$ 
\xymatrix{
	\Gamma_S \ar@{->>}[rd] \ar@/_3pc/@{->>}[dd] \ar@{->>}`r[rr]`[dd][rdd] \ar@{->>}[d] & &\\
	G(M/k(x)) \ar[d] \ar@{->>}[r] & G(L/k(x)) \ar[d] & \\
	G \ar@{->>}[r]  & H	&}
$$
the left, right, top and outer triangle commute. Therefore also the square commutes. So $M/k(x)$ and $G\simeq G(M/k(x))$ constitute a regular singular solution of the embedding problem $(G\twoheadrightarrow H,\ L/k(x))$.
\end{proof}	

We note that the embedding property for $\Gamma_S$ in (ii) above is identical to the embedding property of $\Gamma_d$ in Proposition \ref{prop: embedding property for Gammad} where $d=|S|-1$.

Our main goal is to show that the equivalent statements in Proposition \ref{prop: equivalence of embedding property and differential embedding problem} do indeed hold. Our strategy to prove this is as follows: For $k=\mathbb{C}$ and $S\subseteq\mathbb{P}^1(\C)$ finite, it follows from the classical Riemann-Hilbert correspondence that $\Gamma_S$ is the free proalgebraic group on $|S|-1$ generators. Therefore, by the Gasch\"{u}tz lemma for algebraic groups, applied in the form of Proposition \ref{prop: embedding property for Gammad}, property (ii) of Proposition \ref{prop: equivalence of embedding property and differential embedding problem} holds. Thus also (i) holds for $k=\C$. Via a specialization argument we then deduce from the validity of (i) for $k=\C$, the validity of (i) for any algebraically closed field of characteristic zero. The specialization argument will be discussed in the next section. For this strategy to work, we also need to know that regular singular differential embedding problems are compatible with base change.

\begin{lemma} \label{lemma: extend constants}
	Let $k'/k$ be an extension of algebraically closed fields and let $S\subseteq \mathbb{P}^1(k)$ be a finite set.
	Let $(\f\colon G\twoheadrightarrow H,\ R/k(x))$ be a differential embedding problem over $k(x)$ and set $R'=R\otimes_{k(x)}k'(x)$. Then $(\f_{k'}\colon G_{k'}\twoheadrightarrow H_{k'},\ R'/k'(x))$ is a differential embedding problem over $k'(x)$ and $(\f\colon G\twoheadrightarrow H,\ R/k(x))$ is regular singular (with respect to $S$) if and only if $(\f_{k'}\colon G_{k'}\twoheadrightarrow H_{k'},\ R'/k'(x))$ is regular singular (with respect to $S$).
	
	Furthermore, if this is the case and $T/k(x)$ is a regular singular (with respect to $S$) solution, then $T'/k'(x)$ is a regular singular (with respect to $S$) solution, where $T'=T\otimes_{k(x)}k'(x)$.
\end{lemma}	
\begin{proof}
	We already noted in Lemma \ref{lemma: extend constants for PV} that $R'/k'(x)$ is a Picard-Vessiot ring with differential Galois group $H_{k'}$. Thus $(\f_{k'}\colon G_{k'}\twoheadrightarrow H_{k'},\ R'/k'(x))$ is a differential embedding problem over $k'(x)$.
	By Lemma~\ref{lemma: extend constants for PV} we also know that $R/k(x)$ is regular singular with respect to $S$ if and only if $R'/k'(x)$ is regular singular with respect to $S$. Furthermore,  $d(G_{k'})=d(G)$ by Proposition \ref{prop: d under base change}. So $(\f\colon G\twoheadrightarrow H,\ R/k(x))$ is regular singular with respect to $S$ if and only if $(\f_{k'}\colon G_{k'}\twoheadrightarrow H_{k'},\ R'/k'(x))$ is regular singular with respect to $S$.
	
	If $T/k(x)$ is a regular singular (with respect to $S$) solution, then, as above, $T'/k'(x)$ is a Picard-Vessiot ring with differential Galois group isomorphic to $G_{k'}$. Moreover $R'\subseteq T'$ and $T'/k'(x)$ is regular singular (with respect to $S$).
	As the diagram
	$$
	\xymatrix{
		G(T/k(x)) \ar@{->>}[r] \ar[d] & G(R/k(x)) \ar[d]\\
		G \ar@{->>}[r] & H	
	}
	$$
	commutes, also its base change to $k'$
	$$
	\xymatrix{
		G(T/k(x))_{k'} \ar@{->>}[r] \ar[d] & G(R/k(x))_{k'} \ar[d]\\
		G_{k'} \ar@{->>}[r] & H_{k'}	
	}
	$$
	commutes.
	Combining this with the commutative diagram
	$$
	\xymatrix{
		G(T'/k'(x)) \ar@{->>}[r] \ar[d] & G(R'/k'(x)) \ar[d] \\
		G(T/k(x))_{k'} \ar@{->>}[r]  & G(R/k(x))_{k'}
	}
	$$
	we see that
	$$
	\xymatrix{
		G(T'/k'(x)) \ar@{->>}[r] \ar[d] & G(R'/k'(x)) \ar[d]\\
		G_{k'} \ar@{->>}[r] & H_{k'}	
	}
	$$
	commutes. So $T'/k'(x)$ is a regular singular (with respect to $S$) solution.
\end{proof}

\subsection{Specializations}
Let $k\subseteq k'$ be an inclusion of algebraically closed fields and consider a differential equation $\de(y)=Ay$ with $A\in k'(x)^{n\times n}$. Then there exists a finitely generated $k$-subalgebra $\B$ of $k'$ and a monic polynomial $f\in \B[x]$ such that $A\in \B[x]_f^{n\times n}$. Every morphism of $k$-algebras $c\in\Hom_k(\B,k)$ extends to a morphism $c\colon \B[x]_f\to k(x)$ of $k$-$\de$-algebras by $c(x)=x$. Applying $c$ to the entries of $A$ one thus obtains a matrix $A^c\in k(x)^{n\times n}$ and correspondingly a differential equation
$$\de(y)=A^cy$$ over $k(x)$. In this context, we usually refer to $c\in\Hom_k(\B,k)$ as a \emph{specialization} and to $\de(y)=A^cy$ as a specialization of $\de(y)=Ay$. Clearly, a specialization can completely change the character of a differential equation. For example, the specialization could be sending all entries of $A$ to $0$. However, one might expect that for some, or even many, specializations all relevant properties are preserved. In this section, we first show that this is true for the property of being regular singular, and then we explain the specialization result from \cite{FengWibmer:DifferentialGaloisGroupsSpecializationsAndMatzatsConjecture} that, roughly speaking, shows that also the differential Galois group is preserved under many specializations.

\subsubsection{Specializations of regular singular differential equations} A specialization of a regular singular differential equation does not need to be regular singular. For example, if $a$ and $b$ are distinct, the differential equation $\de(y)=\frac{1}{(x-a)(x-b)}y$ is regular singular with singularities $a$ and $b$, but a specialization sending $a$ and $b$ to $0$, will turn the differential equation into $\de(y)=\frac{1}{x^2}y$, which is irregular singular at $0$. However, this kind of bad specializations are easily avoided. The following lemma is a slight generalization of \cite[Lemma 2.12]{SerafiniWibmer:RegularSingularInverseProblem}. As a notational convenience, we set $\infty^c=\infty$ for any specialization $c$.

\begin{lemma} \label{lemma: specialization of reg or reg sing}
	Let $k\subseteq k'$ be an inclusion of algebraically closed fields and let $A\in k'(x)^{n\times n}$ be such that $\de(y)=Ay$ is regular (or regular singular) at $p\in \mathbb{P}^1(k')$. Then there exists a finitely generated $k$-subalgebra $\B$ of $k'$ and a monic $f\in \B[x]$ such that $p\in\B\cup\{\infty\}$, $A\in \B[x]_f^{n\times n}$ and $\de(y)=A^cy$ is regular (or regular singular) at $p^c$ for every specialization $c\colon \B\to k$.
\end{lemma}	
\begin{proof}
We first treat the regular case. We also first assume that $p\in k'$. By assumption, $\de(y)=A(x+p)y$ is regular at $0$. By Lemma \ref{lemma: Gauge over rational} (i) there exists a $T\in\Gl_n(k'(x))$ such that $A'=TA(x+p)T^{-1}+\de(T)T^{-1}$ has no pole at $x=0$. Let $a_1,\ldots,a_r\in k'\smallsetminus\{0\}$ denote the distinct poles of $A'$ in $k'$. Then
	$(x-a_1)^{n_1}\ldots(x-a_r)^{n_r}A'\in k'[x]^{n\times n}$ for some $n_1,\ldots,n_r\geq 1$. Let $\B\subseteq k'$ be a finitely generated $k$-algebra and $f\in \B[x]$ monic such that $p,a_1,a_1^{-1},\ldots,a_r,a_r^{-1}\in \B$ and $A,A(x+p),T,T^{-1}\in \B[x]_f^{n\times n}$.
	Then $(x-a_1)^{n_1}\ldots(x-a_r)^{n_r}A'\in\B[x]^{n\times n}$ and for any specialization $c\in\Hom_k(\B,k)$ we have $(x-a^c_1)^{n_1}\ldots(x-a^c_r)^{n_r}A'^c\in k[x]^{n\times n}$ and $$A'^c=T^cA(x+p)^c(T^c)^{-1}+\de(T^c)(T^c)^{-1}=T^cA^c(x+p^c)(T^c)^{-1}+\de(T^c)(T^c)^{-1}$$ with $T^c\in\Gl_n(k(x))$. As all $a_i^c$ are non-zero, this shows that $\de(y)=A^cy$ is regular at $p^c$.
	
	In case $p=\infty$, the differential equation $\de(y)=\frac{-1}{x^2}A(\frac{1}{x})$ is regular at $0$ and we can proceed similarly to the above.
	
	For the regular singular case, one can proceed as in the regular case with Lemma \ref{lemma: Gauge over rational} (i) replaced by Lemma \ref{lemma: Gauge over rational} (ii) and $(x-a_1)^{n_1}\ldots(x-a_r)^{n_r}A'\in k'[x]^{n\times n}$ replaced by $x(x-a_1)^{n_1}\ldots(x-a_r)^{n_r}A'\in k'[x]^{n\times n}$.
\end{proof}

The following lemma is a variation of \cite[Lemma 2.13]{SerafiniWibmer:RegularSingularInverseProblem}. It shows that one can always arrange things so that a regular singular differential equation specializes to a regular singular differential equation.

\begin{lemma} \label{lemma: specializaiton of reg sing}
	Let $k\subseteq k'$ be an inclusion of algebraically closed fields and let $S$ be a finite subset of $\mathbb{P}^1(k)\subseteq\mathbb{P}^1(k')$. Let $A\in k'(x)^{n\times n}$ be such that $\de(y)=Ay$ is regular singular with respect to $S$. Then there exists a finitely generated $k$-subalgebra $\B$ of $k'$ and a monic $f\in \B[x]$ such that $A\in \B[x]_f^{n\times n}$ and $\de(y)=A^cy$ is regular singular with respect to $S$ for every specialization $c\colon \B\to k$.
\end{lemma}	
\begin{proof}
	We first assume that $\infty\notin S$.	Write $S=\{s_1,\ldots,s_{d+1}\}\subseteq \mathbb{P}^1(k)\subseteq\mathbb{P}^1(k')$. We can find distinct $a_1,\ldots,a_r\in k'$ and $n_1,\ldots,n_r\geq 1$ such that $(x-a_1)^{n_1}\ldots(x-a_r)^{n_r}A\in k'[x]^{n\times n}$. Without loss of generality, we may assume that $r\geq d+1$ and $a_i=s_i$ for $i=1,\ldots,d+1$. So $\de(y)=Ay$ is regular singular at $a_1,\ldots,a_{d+1}$ and regular at $a_{d+2},\ldots,a_r$.
	By Lemma \ref{lemma: specialization of reg or reg sing} there exists, for $i=1,\ldots,d+1$, a finitely generated $k$-subalgebra $\B_i$ of $k'$ and $f_i\in \B_i[x]$ monic such that $A\in \B_i[x]_{f_i}^{n\times n}$ and $\de(y)=A^cy$ is regular singular at $a_i^c=a_i=s_i$ for all specializations $c\colon\B_i\to k$.
	Similarly, for $i=d+2,\ldots,r$, there exists a finitely generated $k$-subalgebra $\B_i$ of $k'$ and a monic $f_i\in \B_i[x]$ such that $a_i\in \B_i$, $A\in \B_i[x]_{f_i}^{n\times n}$ and $\de(y)=A^cy$ is regular at $a_i^c$ for every specialization $c\colon \B_i\to k$. Because $\de(y)=Ay$ is regular at $\infty$, there exists, again by Lemma \ref{lemma: specialization of reg or reg sing}, a finitely generated $k$\=/subalgebra $\B_\infty\subseteq k'$ and $f_\infty\in \B_\infty[x]$ monic such that $A\in \B_\infty[x]_{f_\infty}^{n\times n}$ and $\de(y)=A^cy$ is regular at $\infty$ for every specialization $c\colon \B_\infty\to k$.
	
	Let $\B$ be the $k$-subalgebra of $k'$ generated by $\B_1,\ldots,\B_r$ and $\B_\infty$ and set $f=f_1\ldots f_rf_\infty\in\B[x]$. Then $a_1,\ldots,a_r\in \B$ and $A\in \B[x]_f^{n\times n}$.	
	Let $c\in\Hom_k(\B,k)$. From $(x-a_1)^{n_1}\ldots(x-a_r)^{n_r}A\in \B[x]^{n\times n}$, it follows that $(x-a_1^c)^{n_1}\ldots(x-a_r^c)^{n_r}A^c\in k[x]^{n\times n}$. This shows that $\de(y)=A^cy$ is regular at all points $p\in k\smallsetminus\{a_1^c,\ldots,a_r^c\}$. By the choice of $\B_i$ for $i=1,\ldots,d+1$, it follows that $\de(y)=A^cy$ is regular singular at $a_i^c=a_i=s_i$. By the choice of $\B_i$ for $i=d+2,\ldots,r$, it follows that $\de(y)=A^cy$ is regular at $a_i^c$. Finally, by the choice of $\B_\infty$, we see that $\de(y)=A^cy$ is regular at $\infty$. In summary, we see that $\de(y)=A^cy$ is regular singular with respect to $S$. 
	
	In case $\infty\in S$, a similar argument applies.
\end{proof}

\subsubsection{The specialization setup}

Let $k'/k$ be an extension of algebraically closed fields and $R'/k'(x)$ a Picard-Vessiot ring of finite type. Roughly speaking, in \cite{FengWibmer:DifferentialGaloisGroupsSpecializationsAndMatzatsConjecture} it was shown that $R'/k'(x)$ can be specialized to a Picard-Vessiot $R/k(x)$ without changing the differential Galois group. More generally, a similar statement holds for solutions of differential embedding problems.
However, more rigorously, to specialize the Picard-Vessiot ring, its differential Galois group or the solution of a differential embedding problem, they first need to be spread out in families. So one first has to specify precisely what kind of properties those spread out families are supposed to have, which leads to a fairly elaborate specialization setup. We now recall this specialization setup and the main specialization result from \cite{FengWibmer:DifferentialGaloisGroupsSpecializationsAndMatzatsConjecture}.

\medskip

Let $\B$ be a $k$-algebra (considered as a constant differential ring) and let $\mathcal{Q}$ be a $\B$-$\de$-algebra. An \emph{action} of an affine group scheme $\G$ over $\B$ on a $\mathcal{Q}$-$\de$-algebra $\mathcal{R}$ is a morphism of group functors $\G\to\underline{\operatorname{Aut}}(\mathcal{R}/\mathcal{Q})$, where $\underline{\operatorname{Aut}}(\mathcal{R}/\mathcal{Q})(\mathcal{C})=\operatorname{Aut}(\mathcal{R}\otimes_\B\mathcal{C}/\mathcal{Q}\otimes_\B\mathcal{C})$ for any $\B$-algebra $\mathcal{C}$ and the automorphisms are assumed to be differential automorphisms. To give such an action is equivalent to specifying a morphism $\rho\colon \mathcal{R}\to \mathcal{R}\otimes_\B\B[\G]$ of $\mathcal{Q}$-$\de$-algebras satisfying two diagrams, representing that the action is associative and unital (\cite[Lemma 3.2.1]{FengWibmer:DifferentialGaloisGroupsSpecializationsAndMatzatsConjecture}). We say that $\mathcal{R}/\mathcal{Q}$ is a \emph{differential $\G$-torsor} if the induced morphism $\mathcal{R}\otimes_{\mathcal{Q}}\mathcal{R}\to \mathcal{R}\otimes_\B\B[\G],\ r_1\otimes r_2\mapsto (r_1\otimes 1)\rho(r_2)$ is an isomorphism. This implies that $\spec(\mathcal{R})$ is a $\G_\mathcal{Q}$-torsor in the usual sense.
For $\mathcal{A}\in \mathcal{Q}^{n\times n}$ a differential $\G$-torsor $\mathcal{R}/\mathcal{Q}$ is a \emph{differential $\G$-torsor for $\de(y)=\mathcal{A}y$} if there exists a $\mathcal{Y}\in\Gl_n(\mathcal{R})$ such that $\de(\mathcal{Y})=\mathcal{A}\mathcal{Y}$, $\mathcal{R}=\mathcal{Q}[\mathcal{Y},\frac{1}{\det(\mathcal{Y})}]$ and $g(\mathcal{Y}\otimes 1)=(\mathcal{Y}\otimes 1)(1\otimes [g])$ for some $[g]\in\Gl_n(\mathcal{C})$ for all $g\in\G(\mathcal{C})$ and all $\B$-algebras $\mathcal{C}$. We call such a $\mathcal{Y}\in\Gl_n(\mathcal{R})$ a fundamental solution matrix for $\de(y)=\mathcal{A}y$.

The most basic example of the above definitions (cf. \cite[Example 3.2.3]{FengWibmer:DifferentialGaloisGroupsSpecializationsAndMatzatsConjecture}) is a Picard-Vessiot ring $R/k(x)$ for $\de(y)=Ay$ with fundamental solution matrix $Y\in\Gl_n(R)$ and differential Galois group $G$. I.e., with $\B=k$, $\mathcal{Q}=k(x)$, $\mathcal{R}=R$, $\mathcal{A}=A$, $\mathcal{Y}=Y$ and $\mathcal{G}=G$, we have a differential $\G$\=/torsor $\mathcal{R}/\mathcal{Q}$ for $\de(y)=\mathcal{A}y$ with fundamental solution matrix $\mathcal{Y}$. In this case, $k[G]=(R\otimes_{k(x)}R)^\de$ and the canonical map $\eta\colon R\otimes_kk[G]\to R\otimes_{k(x)}R$ is an isomorphism. The coaction $\rho$ is then given by $\rho\colon R\to R\otimes_k k[G], \ r\mapsto \eta^{-1}(1\otimes r)$ and the induced isomorphism $R\otimes_{k(x)}R\to R\otimes_kk[G]$ is $\eta^{-1}$. 

The following lemma explains how we can spread out a Picard-Vessiot ring into a nice family suitable for specializing.

\begin{lemma} \label{lemma: spread out PV ring}
	Let $k\subseteq k'$ be an inclusion of algebraically closed fields, $A\in k'(x)^{n\times n}$ and $R/k'(x)$ a Picard-Vessiot ring for $\de(y)=Ay$ with fundamental solution matrix $Y\in\Gl_n(R)$. Let $\B_0$ be a finitely generated $k$-subalgebra of $k'$ and $f_0\in\B_0[x]$ monic.
	Assume that the differential Galois group of $R/k'(x)$ is of the form $G_{k'}$ for some algebraic group $G$ over $k$. Then there exist
	\begin{itemize}
		\item a finitely generated $k$-subalgebra $\B$ of $k'$ containing $\B_0$,
		\item a monic polynomial $f\in \B [x]$ such that $f_0$ divides $f$ in $\B[x]$ and $A\in \B[x]_f^{n\times n}$,
		\item a differential $G_\B$-torsor $\mathcal{R}/\B[x]_f$ for $\de(y)=Ay$ with fundamental solution matrix $\mathcal{Y}$ such that $\mathcal{R}$ is flat over $\B[x]_f$, $\mathcal{R}\otimes_{\B[x]_f}K(x)$ is $\de$-simple, where $K\subseteq k'$ is the algebraic closure of the field of fractions of $\B$ and $\mathcal{R}\otimes_{\B[x]_f}k'(x)\simeq R$ via an isomorphism mapping $\mathcal{Y}$ to $Y$.
	\end{itemize}		
\end{lemma}	
\begin{proof}
	This is just a special case of \cite[Lemma 3.2.9]{FengWibmer:DifferentialGaloisGroupsSpecializationsAndMatzatsConjecture}.
\end{proof}	

The following theorem is the key specialization result.

\begin{theo}[{\cite[Theorem 4.6.2]{FengWibmer:DifferentialGaloisGroupsSpecializationsAndMatzatsConjecture}}] \label{theo: specialization}
	Let $\B$ be a finitely generated $k$-algebra that is an integral domain, $f\in\B[x]$, $\G$ an affine group scheme of finite type over $\B$ and $\mathcal{R}/\B[x]_f$ a differential $\G$-torsor for $\de(y)=Ay$ with $A\in\B[x]_f^{n\times n}$ such that $\mathcal{R}$ is flat over $\B[x]_f$ and $\mathcal{R}\otimes_{\B[x]_f}K(x)$ is $\de$-simple, where $K$ is the algebraic closure of the field of fractions of $\B$. Then there exists a specialization $c\colon \B\to k$ such that $R^c=\mathcal{R}\otimes_{\B[x]_f}k(x)$ is a Picard-Vessiot ring over $k(x)$.
\end{theo}	

We note that, in the above theorem, $R^c/k(x)$ automatically is a Picard-Vessiot ring for $\de(y)=A^cy$ with differential Galois group $G^c=\G_k$, where the implied base change is formed using $c\colon \B\to k$.

The following lemma explains how we can spread out the solution of a regular singular embedding problem.

\begin{lemma} \label{lemma: spread out solution to differential embedding problem}
	Let $k'/k$ be an extension of algebraically closed fields and $S\subseteq \mathbb{P}^1(k)$ finite. Let $(\f\colon G\twoheadrightarrow H,\ R/k(x))$ be a regular singular (with respect to $S$) differential embedding problem over $k(x)$ and let $T'/k'(x)$ be a regular singular (with respect to $S$) solution of the induced regular singular (with respect to $S$) differential embedding problem $(\f_{k'}\colon G_{k'}\twoheadrightarrow H_{k'},\ R'/k'(x))$ over $k'(x)$, where $R'=R\otimes_{k(x)}k'(x)$. Furthermore, let $A\in k'(x)^{n\times n}$ be such that $T'/k'(x)$ is a Picard-Vessiot ring for $\de(y)=Ay$. Then there exist
	\begin{enumerate}
		\item a monic polynomial $h\in k[x]$, and a $k[x]_h$-$\de$-subalgebra $R_0$ of $R$ with an action of $H$ on $R_0/k[x]_h$
		such that $R_0\otimes_{k[x]_h}k(x)=R$  (equivalently, $U^{-1}R_0=R$, where $U=k[x]_h\smallsetminus \{0\}$) and the induced action of $H$ on $U^{-1}R_0/U^{-1}k[x]_h$ is the natural action of $H$ on $R/k(x)$,
		%		\rf{Since $S'$ and $S^{\gen}$ are used to denote the rings, to avoid the confusion, I changed $S$ into $U$ here. }
		\item a finitely generated $k$-subalgebra $\mathcal{B}$ of $k'$ and a monic polynomial $f\in \mathcal{B}[x]$ such that $h\in k[x]\subseteq \mathcal{B}[x]$ divides $f$ in $\mathcal{B}[x]$, $A\in \B[x]_f^{n\times n}$ and $\de(y)=A^cy$ is regular singular with respect to $S$ for every specialization $c\colon \B\to k$,
		\item a differential $G_\B$-torsor $\mathcal{T}/\B[x]_f$ for $\de(y)=Ay$ such that $\mathcal{T}$ is flat over $\B[x]_f$
		and %$S^{\operatorname{gen}}
		$\mathcal{T}\otimes_{\B[x]_f}K(x)$ is $\de$-simple, where $K\subseteq k'$ is the algebraic closure of the field of fractions of $\B$,
		
		\item a morphism $\mathcal{R}\to \mathcal{T}$ of $\mathcal{B}[x]_f$-$\de$-algebras that is equivariant with respect to the morphism $\f_{\mathcal{B}}\colon G_{\mathcal{B}}\to H_{\mathcal{B}}$, where $\mathcal{R}=R_0\otimes_{k[x]_h}\mathcal{B}[x]_f=(R_0\otimes_k\B)_f$ and $\mathcal{R}/\B[x]_f$ has an $H_\B$-action induced from \rm{(i)} above. (The equivariance means that for every $\mathcal{B}$-algebra $\mathcal{C}$ and $g\in G_{\B}(\mathcal{C})$ the diagram
		$$
		\xymatrix{
			\mathcal{T}\otimes_\B\mathcal{C} \ar^g[r] & \mathcal{T}\otimes_\B\mathcal{C} \\
			\mathcal{R}\otimes_\B\mathcal{C} \ar[u] \ar^-{\f(g)}[r]& \mathcal{R}\otimes_\B \mathcal{C} \ar[u]
		}
		$$
		commutes.)
	\end{enumerate}
\end{lemma}
\begin{proof}
	This is more or less Lemma 5.3.3 in \cite{FengWibmer:DifferentialGaloisGroupsSpecializationsAndMatzatsConjecture}. The main difference is the requirement in (ii) that $\de(y)=A^cy$ is regular singular with respect to $S$ for every specialization $c\colon \B\to k$. To achieve this, we use Lemma \ref{lemma: specializaiton of reg sing} to produce a finitely generated $k$-subalgebra $\B_0$ of $k'$ and a monic $f_0\in \B_0[x]$ such that $A\in \B_0[x]_{f_0}^{n\times n}$ and $\de(y)=A^cy$ is regular singular with respect to $S$ for every specialization $c\colon \B_0\to k$. In the proof of Lemma 5.3.3 in \cite{FengWibmer:DifferentialGaloisGroupsSpecializationsAndMatzatsConjecture}, where Lemma 3.2.9 from \cite{FengWibmer:DifferentialGaloisGroupsSpecializationsAndMatzatsConjecture} is used, we then use Lemma \ref{lemma: spread out PV ring} with this $\B_0$ and $f_0$.
\end{proof}

\subsection{Proof of the main results}

We now have all the ingredients to establish the solvability of all regular singular differential embedding problems.

\begin{theo} \label{theo: solve regular singular embedding problem}
Let $S$ be a finite subset of $\mathbb{P}^1(k)$. Then every regular singular (with respect to $S$) differential embedding problem over $k(x)$ has a regular singular (with respect to $S$) solution.
\end{theo}	
\begin{proof}
	For $k=\mathbb{C}$ and $d=|S|-1$, it follows from the classical Riemann-Hilbert correspondence (\cite[Theorem 6.15]{SingerPut:differential}), paired with the tannakian formalism, that the differential Galois group $\Gamma_S$ of the family of all regular singular linear differential equations over $\C(x)$ with singularities in $S$, is isomorphic to $\Gamma_d$. It thus follows from Proposition \ref{prop: embedding property for Gammad} that $\Gamma_S$ satisfies condition (ii) of Proposition \ref{prop: equivalence of embedding property and differential embedding problem}. Therefore condition (i) of  Proposition~\ref{prop: equivalence of embedding property and differential embedding problem} holds. That is, the theorem holds for $k=\mathbb{C}$. We will deduce the validity of the theorem in general, from the case $k=\C$ via a specialization argument.
	
	Let $(G\twoheadrightarrow H,\ R/k(x))$ be a regular singular (with respect to $S$) differential embedding problem over $k(x)$.	Let $k_1$ be a countable algebraically closed subfield of $k$ such that $S\subseteq \mathbb{P}^1(k_1)\subseteq\mathbb{P}^1(k)$ and such that $R/k(x)$ is a Picard-Vessiot ring for $\de(y)=Ay$ with $A\in k_1(x)^{n\times n}$. Let $R_1/k_1(x)$ be a Picard-Vessiot ring for $\de(y)=Ay$ over $k_1(x)$ and let $H_1$ be the corresponding differential Galois group. Then $R_1\otimes_{k_1(x)}k(x)/k(x)$ is a Picard-Vessiot ring for $\de(y)=Ay$ with differential Galois group $(H_1)_{k}$ (Lemma~\ref{lemma: extend constants}). By the uniqueness of Picard-Vessiot rings, we have $R_1\otimes_{k_1(x)}k(x)\cong R$ and $(H_1)_k\cong H$. Enlarging $k_1$ if necessary, we can also assume that $G\twoheadrightarrow H$ is the base change to $k$ of a quotient map $G_1\twoheadrightarrow H_1$ of algebraic groups over $k_1$. We thus have a differential embedding problem $(G_1\twoheadrightarrow H_1, \ R_1/k_1(x))$ over $k_1(x)$.
	
	The base change of $(G_1\twoheadrightarrow H_1, \ R_1/k_1(x))$ from $k_1$ to $k$ is the original embedding problem $(G\twoheadrightarrow H,\ R/k(x))$. As the latter is regular singular with respect to $S$, it follows from Lemma \ref{lemma: extend constants} that  $(G_1\twoheadrightarrow H_1, \ R_1/k_1(x))$ is regular singular with respect to $S$. Because $k_1$ is countable, it can be embedded into $\C$ and moving forward we now fix such an embedding. By Lemma \ref{lemma: extend constants} again, the base change  $(G_2\twoheadrightarrow H_2, \ R_2/\C(x))$ of $(G_1\twoheadrightarrow H_1, \ R_1/k_1(x))$ from $k_1$ to $\C$ is a regular singular (with respect to $S$) differential embedding problem over $\C(x)$. From the case $k=\mathbb{C}$ of the theorem established in the first paragraph of this proof, we know that $(G_2\twoheadrightarrow H_2, \ R_2/\C(x))$ has a regular singular (with respect to $S$) solution $T_2/\C(x)$. Let $A\in \C(x)^{n\times n}$ be such that $T_2/\C(x)$ is a Picard-Vessiot ring for $\de(y)=Ay$.

	Let $h$, $R_0$, $\B$, $f$, $\mathcal{T}/\B[x]_f$ and $\mathcal{R}\to\mathcal{T}$ be as in Lemma \ref{lemma: spread out solution to differential embedding problem} applied with $k=k_1$ and $k'=\C$ to the regular singular (with respect to $S$) differential embedding problem $(G_1\twoheadrightarrow H_1, \ R_1/k_1(x))$ over $k_1(x)$, the regular singular (with respect to $S$) solution $T_2/\C(x)$ of $(G_2\twoheadrightarrow H_2, \ R_2/\C(x))$ and the matrix $A\in\C(x)^{n\times n}$. By Theorem \ref{theo: specialization}, there exists a specialization $c\colon \B\to k_1$ such that $T^c=\mathcal{T}\otimes_{\B[x]_f}k_1(x)$ is a Picard-Vessiot ring over $k_1(x)$ for $\de(y)=A^cy$ with differential Galois group $(G_\B)_{k_1}=G_1$. Note that $\mathcal{R}\otimes_{\B[x]_f}k_1(x)=(R_0\otimes_{k_1[x]_h}\B[x]_f)\otimes_{\B[x]_f}k_1(x)=R_0\otimes_{k_1[x]_h}k_1(x)=R_1$.
	So the commutative diagram
	$$
	\xymatrix{
		\mathcal{R} \ar[d] \ar[r] & \mathcal{R}\otimes_\B \B[H_1]=\mathcal{R}\otimes_{\B[x]_f}\B[x]_f[H_1] \ar[d] \\
		\mathcal{T} \ar[r] & \mathcal{T}\otimes_\B \B[G_1]=\mathcal{T}\otimes_{\B[x]_f}\B[x]_f[G_1]
	}
	$$ 
	induces, via base change $c\colon\B[x]_f\to k_1$, the commutative diagram
	 $$
	 \xymatrix{
	 	R_1 \ar[d] \ar[r] & R_1\otimes_{k_1} k_1[H_1] \ar[d] \\
	 	T^c \ar[r] & T^c\otimes_{k_1} k_1[G_1], 
	 }
	 $$ 
	which shows that $T^c/k_1(x)$ is a solution of $(G_1\twoheadrightarrow H_1, \ R_1/k_1(x))$. Because $T^c/k_1(x)$ is a Picard-Vessiot ring for $\de(y)=A^cy$, point (ii) of Lemma \ref{lemma: spread out solution to differential embedding problem} guarantees that the solution $T^c/k_1(x)$ is regular singular with respect to $S$. Finally, $T^c\otimes_{k_1(x)}k(x)$ is a regular singular (with respect to $S$) solution of $(G\twoheadrightarrow H,\ R/k(x))$ by Lemma \ref{lemma: extend constants}.	
\end{proof}	

Combining Theorem \ref{theo: solve regular singular embedding problem} with Proposition \ref{prop: equivalence of embedding property and differential embedding problem} we obtain:

\begin{cor} \label{cor: GammaS solves embedding problems}
	Let $S$ be a finite subset of $\mathbb{P}^1(k)$ and let $\Gamma_S$ denote the differential Galois group of the family of all regular singular differential equations over $k(x)$ with singularities in $S$. Then every algebraic embedding problem $(G\twoheadrightarrow H,\ \Gamma_S\twoheadrightarrow H)$ for $\Gamma_S$ with $d(G)\leq |S|-1$ has a solution. \qed
\end{cor}

We are now prepared to prove Conjectures \ref{conj: RH} and \ref{conj: equivalent} from the introduction in case the base field $k$ is countable. 
In fact, we prove the conjectures without the assumption that $S$ is finite.

\begin{theo} \label{theo: main main}
Assume that $k$ is countable and let $S$ be a subset of $\mathbb{P}^1(k)$. Then the differential Galois group of the family of all regular singular differential equations over $k(x)$ with singularities contained in $S$, is isomorphic to the free proalgebraic group on a set of cardinality $|S|-1$.
\end{theo}	
\begin{proof}
	Let $\Gamma_S$ be the differential Galois group of the family of all regular singular differential equations over $k(x)$ with singularities contained in $S$. We first treat the case that $S$ is finite.
	By Corollary \ref{cor: GammaS solves embedding problems}, every algebraic embedding problem $(G\twoheadrightarrow H,\ \Gamma_S\twoheadrightarrow H)$ for $\Gamma_S$ with $d(G)\leq |S|-1$ has a solution. According to Proposition \ref{prop: characterization for countable}, to conclude that $\Gamma_S$ is the free proalgebraic group on $|S|-1$ generators, it suffices to know that $d(\Gamma_S)\leq |S|-1$. This is provided by Corollary \ref{cor: d for countable}. 
	So the case when $S$ is finite is proved.
	
	In the general case, we use the following characterization of free proalgebraic groups over countable fields on countably infinite sets from \cite[Cor. 3.43]{Wibmer:FreeProalgebraicGroups}: A proalgebraic group $\Gamma$ of countable rank is free on a countably infinite set if and only if every algebraic embedding problem for $\Gamma$ is solvable. Firstly, note that $\Gamma_S$ has countable rank because there are only countably many linear differential equations over $k(x)$. So it suffices to show that every algebraic embedding problem for $\Gamma_S$ is solvable.
	
	For a subset $S'$ of $S$ let $R_{S'}/k(x)$ denote the Picard-Vessiot ring for the family of all regular singular differential equations over $k(x)$ with singularities in $S'$ and let $\mathcal{S}$ be the directed set of all finite subsets of $S$. Then $R_S$ is the directed union of all $R_{S'}$ with $S'\in \mathcal{S}$ and accordingly $\Gamma_S=\varprojlim_\mathcal{S} \Gamma_{S'}$. Let $(G\twoheadrightarrow H,\ \Gamma_S\twoheadrightarrow H)$ be an algebraic embedding problem for $\Gamma_S$. Because $H$ is algebraic, the morphism $\Gamma_S\twoheadrightarrow H$ factors through $\Gamma_S\twoheadrightarrow \Gamma_{S'}$ for some $S'\in\mathcal{S}$. Because $G$ is algebraic, $d(G)$ is finite. Without loss of generality, we may assume that $d(G)\leq |S'|-1$. By Corollary \ref{cor: GammaS solves embedding problems}, the embedding problem
	$(G\twoheadrightarrow H,\ \Gamma_{S'}\twoheadrightarrow H)$ has a solution $\Gamma_{S'}\twoheadrightarrow G$. The composition $\Gamma_S\twoheadrightarrow \Gamma_{S'}\twoheadrightarrow G$ then is a solution to $(G\twoheadrightarrow H,\ \Gamma_S\twoheadrightarrow H)$.	
\end{proof}	
In the special case $S=\mathbb{P}^1(k)$, Theorem \ref{theo: main main} implies:

\begin{cor} \label{cor: free for countable}
	Assume that $k$ is countable. Then the differential Galois group of the family of all regular singular differential equations over $k(x)$ is the free proalgebraic group on a set of cardinality $|k|$. \qed
\end{cor}

Theorem \ref{theo: solve regular singular embedding problem} is not strong enough to imply the freeness of $\Gamma_S$ for $k$ uncountable. However, Theorem~\ref{theo: solve regular singular embedding problem} implies a property that is weaker than freeness, namely, $\Gamma_S$ is a \emph{projective proalgebraic group}. Recall (\cite[Section 2.3]{BassLubotzkyMagidMozes:TheProalgebraicCompletionOfRigidGroups}) that a proalgebraic group $P$ is \emph{projective} if for every quotient map $G\twoheadrightarrow H$ of proalgebraic groups and every morphism $P\to H$, there exists a morphism $P\to G$ making
$$
\xymatrix{
	P \ar[rd] \ar@{..>}[d] & \\
	G \ar@{->>}[r] & H
}
$$
commutative. We note that by Proposition 4 in \cite{BassLubotzkyMagidMozes:TheProalgebraicCompletionOfRigidGroups}, in the above definition, it suffices to consider the case that $G$ is algebraic.

\begin{cor} \label{cor: projective}
	Let $S\subseteq\mathbb{P}^1(k)$. Then the differential Galois group of the family of all regular singular differential equations over $k(x)$ with singularities contained in $S$ is a projective proalgebraic group.
\end{cor}
\begin{proof}
	Let $\Gamma_S$ denote the differential Galois group of the family of all regular singular differential equations over $k(x)$ with singularities contained in $S$. Let $\alpha\colon G\twoheadrightarrow H$ be a quotient map of algebraic groups and $\beta\colon \Gamma_S\to H$ a morphism. We have to show that there exists a morphism $\f\colon \Gamma_S\to G$ making
	\begin{equation} \label{eq: diag for proj}
	\xymatrix{
		\Gamma_S \ar^-{\beta}[rd] \ar_-{\f}@{..>}[d] & \\
		G \ar^-{\alpha}@{->>}[r] & H
	}
	\end{equation}
	commutative. Replacing $H$ with $\beta(\Gamma_S)$ and $G$ with $\alpha^{-1}(\beta(\Gamma_S))$, we may assume that $\beta\colon \Gamma_S\to H$ is a quotient map.

	We first consider the case that $S$ is finite. Then, as $H$ is an algebraic quotient of $\Gamma_S$, it follows that $H$ can be generated by $d=|S|-1$ elements (Corollary \ref{cor: same quotients}). Let $h_1,\ldots,h_d$ be generators of $H$. As we don't know if $d(G)\leq d$, we may not be able to lift $h_1,\ldots,h_d$ to generators of $G$, but we can surely lift each $h_i$ individually. So let $g_1,\ldots,g_d\in G(k)$ be such that $\alpha(g_i)=h_i$ for $i=1,\ldots,d$. Set $G'=\langle g_1,\ldots,g_d\rangle\leq G$. Then $\alpha$ restricts to a quotient map $G'\twoheadrightarrow H$ and $d(G')\leq d$. By Corollary \ref{cor: GammaS solves embedding problems}, there exists a morphism $\f\colon \Gamma_S\to G'$ making 
		$$
	\xymatrix{
		\Gamma_S \ar^-{\beta}[rd] \ar_-{\f}@{..>}[d] & \\
		G' \ar^-{\alpha}@{->>}[r] & H
	}
	$$
	commutative. Then $\f$ also makes (\ref{eq: diag for proj}) commutative. This proves the claim for $S$ finite.
	
	In the general case, as in the proof of Theorem \ref{theo: main main}, we can write $\Gamma_S$ as $\Gamma_S=\varprojlim_{\mathcal{S}}\Gamma_{S'}$, with $\mathcal{S}$ the directed set of all finite subsets of $S$. Because $H$ is algebraic, $\beta\colon \Gamma_S\twoheadrightarrow H$ factors through $\Gamma_{S'}\twoheadrightarrow H$ for some $S'\in\mathcal{S}$. From the finite case established above, we know that there exists a morphism $\Gamma_{S'}\to G$ making
	$$
	\xymatrix{
		\Gamma_{S'} \ar@{->>}[rd] \ar@{..>}[d] & \\
		G \ar@{->>}[r]&  H
	}
	$$ 
	commutative. The composition $\Gamma_S\twoheadrightarrow\Gamma_{S'}\to G$ then makes (\ref{eq: diag for proj}) commutative.
\end{proof}	

In connection with Corollary \ref{cor: projective}, it seems interesting to note that it has recently been shown in \cite{Magid:TheDifferentialGaloisGroupOfTheMaximalProunipotentExtensionIsFree} that the maximal prounipotent quotient of the absolute differential Galois group of any differential field (with constants $k$) is a projective proalgebraic group (and hence a free prounipotent group).

\subsection{Future work} \label{subsec: Future work} Our proof of Theorem \ref{theo: main main} relies on the classical Riemann-Hilbert correspondence (over $\mathbb{C}$) and is restricted to countable base fields. It seems desirable to have a purely algebraic proof of Conjecture \ref{conj: RH} that does not rely on the classical Riemann-Hilbert correspondence and of course we would like to remove the assumption that $k$ is countable in Theorem~\ref{theo: main main}. A possible strategy to tackle this problem is as follows:

\begin{itemize}
	\item[(1)] Find a characterization of free proalgebraic groups in terms of embedding problems.
	\item[(2)] Translate the characterization in (1) to the language of differential embedding problems.
	\item[(3)] Prove that all differential embedding problems from (2) are solvable using patching.
\end{itemize}	
We note that the idea for the above strategy (without the patching) essentially goes back to K. Iwasawa and his proof of the solvable variant of the Shafarevich conjecture (\cite{Iwasawa:OnSolvableExtensionsOfAlgebraicNumberFields}). Furthermore, the above strategy has recently been applied successfully in the proof of Matzat's conjecture, stating that the absolute differential Galois group of $k(x)$ is the free proalgebraic group on a set of cardinality $|k|$ (\cite{BachmayrHarbaterHartmannWibmer:FreeDifferentialGaloisGroups}, \cite{BachmayrHarbaterHartmannWibmer:TheDifferentialGaloisGroupOfRationalFunctionField}, \cite{FengWibmer:DifferentialGaloisGroupsSpecializationsAndMatzatsConjecture}).

Point (1) has been achieved in \cite{Wibmer:FreeProalgebraicGroups} for free proalgebraic groups on a set $X$ over a field $k$ in the case that $|X|\geq |k|$. For $|X|$ finite, Conjecture \ref{conj: Gaschuetz for proalgebraic} provides the desired characterization (Corollary \ref{cor: characterization of Gammad}). 

Point (2) is routine. See Proposition \ref{prop: equivalence of embedding property and differential embedding problem} for an example of this. Point (3) requires regular singular variants of the patching results in \cite{BachmayrHarbaterHartmannWibmer:DifferentialEmbeddingProblems}, \cite{BachmayrHarbaterHartmann:DifferentialEmbeddingProblemsOverLaurentSeriesFields}, \cite{BachmayrHartmannHarbaterPopLarge}, \cite{BachmayrHarbaterHartmannWibmer:TheDifferentialGaloisGroupOfRationalFunctionField}.

\bibliographystyle{alpha}
 \bibliography{bibdata}
\end{document}